\documentclass[a4paper, 10pt, american]{amsart}
\usepackage{times,latexsym,amssymb}
\usepackage{amsmath,amsthm,bm}
\usepackage{xcolor}
\usepackage[colorlinks,pdfpagelabels,pdfstartview = FitH,bookmarksopen
= true,bookmarksnumbered = true,linkcolor = blue,plainpages =
false,hypertexnames = false,citecolor = red,pagebackref=false]{hyperref}

\usepackage{amsbsy}
\usepackage{amstext}
\usepackage{amssymb}
\usepackage{esint}
\usepackage{stmaryrd}
\usepackage[pdftex]{graphicx}
\usepackage{floatflt}
\usepackage{appendix}

\allowdisplaybreaks
\newtheorem{theorem}{Theorem}
\newtheorem{lemma}[theorem]{Lemma}
\newtheorem{definition}[theorem]{Definition}
\newtheorem{proposition}[theorem]{Proposition}
\newtheorem{corollary}[theorem]{Corollary}
\newtheorem{remark}[theorem]{Remark}
\numberwithin{theorem}{section}
\numberwithin{equation}{section}

\newcommand{\mint}{- \mskip-19,5mu \int}

\def\N{\mathbb{N}}
\def\R{\mathbb{R}}

\renewcommand{\d}{\mathrm{d}}
\newcommand{\dx}{\mathrm{d}x}

\newcommand{\dt}{\mathrm{d}t}
\newcommand{\ds}{\mathrm{d}s}

\renewcommand{\epsilon}{\varepsilon}

\DeclareMathOperator{\Div}{div}

\DeclareMathOperator{\loc}{loc}
\renewcommand{\epsilon}{\varepsilon}
\newcommand{\eps}{\varepsilon}
\renewcommand{\rho}{\varrho}

\def\eqn#1$$#2$${\begin{equation}\label#1#2\end{equation}}

\newcommand{\pl}{\partial}

\newcommand{\dsty}{\displaystyle}
\newcommand{\al}{\alpha}

\newcommand{\lm}{\lambda}

\newcommand{\varep}{\varepsilon}

\newcommand{\z}{\zeta}
\newcommand{\bq}{\boldsymbol{\mathfrak q}}
\newcommand{\bom}{\boldsymbol{\mathfrak m}}
\let\TeXchi\chi
\newbox\chibox
\setbox0 \hbox{\mathsurround0pt $\TeXchi$}
\setbox\chibox \hbox{\raise\dp0 \box 0 }
\def\chi{\copy\chibox}

\newcommand\sfm{{\boldsymbol{\mathsf m}}}
\newcommand\sfq{{\boldsymbol{\mathsf q}}}
\newcommand\sfk{{\boldsymbol{\mathsf k}}}
\newcommand\sfb{{\boldsymbol{\mathsf b}}}
\newcommand\sfB{{\boldsymbol{\mathsf B}}}

\newcommand\sfL{{\boldsymbol{\mathsf L}}}
\newcommand\sfI{{\boldsymbol{\mathsf I}}}
\newcommand\sfY{{\boldsymbol{\mathsf Y}}}
\newcommand\sfM{{\boldsymbol{\mathsf M}}}
\newcommand\sfA{{\boldsymbol{\mathsf A}}}
\newcommand\sfE{{\boldsymbol{\mathsf E}}}

\newcommand\sff{{\boldsymbol{\mathsf f}}}
\newcommand\sfK{{\boldsymbol{\mathsf K}}}
\newcommand\sfF{{\boldsymbol{\mathsf F}}}
\newcommand\sfR{{\boldsymbol{\mathsf R}}}
\newcommand\sfT{{\boldsymbol{\mathsf T}}}
\newcommand\sfJ{{\boldsymbol{\mathsf J}}}

\newcommand\sfr{{\boldsymbol{\mathsf r}}}
\newcommand\sfG{{\boldsymbol{\mathsf G}}}
\newcommand\sfH{{\boldsymbol{\mathsf H}}}
\def\Xint#1{\mathchoice
    {\XXint\displaystyle\textstyle{#1}}%
    {\XXint\textstyle\scriptstyle{#1}}%
    {\XXint\scriptstyle\scriptscriptstyle{#1}}%
    {\XXint\scriptscriptstyle\scriptscriptstyle{#1}}%
    \!\int}
\def\XXint#1#2#3{\setbox0=\hbox{$#1{#2#3}{\int}$}
    \vcenter{\hbox{$#2#3$}}\kern-0.5\wd0}
\def\bint{\Xint-}
\def\dashint{\Xint{\raise4pt\hbox to7pt{\hrulefill}}}

\def\Xiint#1{\mathchoice
    {\XXiint\displaystyle\textstyle{#1}}%
    {\XXiint\textstyle\scriptstyle{#1}}%
    {\XXiint\scriptstyle\scriptscriptstyle{#1}}%
    {\XXiint\scriptscriptstyle\scriptscriptstyle{#1}}%
    \!\iint}
\def\XXiint#1#2#3{\setbox0=\hbox{$#1{#2#3}{\iint}$}
    \vcenter{\hbox{$#2#3$}}\kern-0.5\wd0}
\def\biint{\Xiint{-\!-}}

\subjclass[2020]{35B65, 35K67, 35K40, 35K55}

\begin{document}
\title[Regularity for Sub-Critical Parabolic Systems]{Regularity theory for sub-critical $p$-parabolic systems with measurable coefficients}

\date{\today}

\author[V. B\"ogelein]{Verena B\"{o}gelein}
\address{Verena B\"ogelein\\
Fachbereich Mathematik, Universit\"at Salzburg\\
Hellbrunner Str. 34, 5020 Salzburg, Austria}
\email{verena.boegelein@plus.ac.at}

\author[F. Duzaar]{Frank Duzaar}
\address{Frank Duzaar\\
Fachbereich Mathematik, Universit\"at Salzburg\\
Hellbrunner Str. 34, 5020 Salzburg, Austria}
\email{frankjohannes.duzaar@plus.ac.at}

\author[U. Gianazza]{Ugo Gianazza}
\address{Ugo Gianazza\\
Dipartimento di Matematica ``F. Casorati",
Universit\`a di Pavia\\ 
via Ferrata 5, 27100 Pavia, Italy}
\email{ugogia04@unipv.it}

\author[N. Liao]{Naian Liao}
\address{Naian Liao\\
Fachbereich Mathematik, Universit\"at Salzburg\\
Hellbrunner Str. 34, 5020 Salzburg, Austria}
\email{naian.liao@plus.ac.at}

\begin{abstract}
A quantitative regularity theory is developed for weak solutions to the parabolic system
$$
\partial_t u-\mathrm{div}\,\sfA(x,t,Du)=0
\quad\text{in }E_T\subset \R^N\times\R,
$$
which features the $p$-Laplacian with measurable coefficients. We focus on the sub-critical range $1<p\le \tfrac{2N}{N+2}$
and obtain two main results. \emph{Local boundedness:} starting from an $L^\sfr$-control of $u$ with $\sfr>\frac{N(2-p)}{p}$, we derive sharp, scale-invariant $L^\infty$-estimates. \emph{Higher integrability of the gradient:} $|Du|$ self-improves from $L^p_{\rm loc}$ to $L^{p(1+\varepsilon)}_{\mathrm{loc}}$ for some $\varepsilon>0$ depending only on the data. The same results still hold given proper source terms.

\end{abstract}

\maketitle

%****************************************************************

\tableofcontents

\section{Introduction and results}
Gehring’s lemma on higher integrability exemplifies a fundamental self-improving phenomenon in nonlinear analysis: a reverse Hölder inequality leads to an automatic gain in integrability, thereby upgrading $L^p$-bounds to
$L^{p+\epsilon}$-regularity. Such effects are closely related to other quantitative self-improvement principles, including those arising from $p$-Poincaré inequalities on metric measure spaces, 
$p$-capacity conditions, and $p$-thickness assumptions. In each of these settings, one observes a characteristic openness of the underlying regularity condition: the validity of an estimate at an endpoint exponent automatically implies its persistence in a neighborhood of that exponent.

In the parabolic setting, this phenomenon governs the higher integrability of spatial gradients of weak solutions to degenerate parabolic systems, a cornerstone for partial regularity and fine quantitative estimates. Since the breakthrough of Kinnunen–Lewis \cite{Kinnunen-Lewis:1} in 2000, higher integrability for gradients of weak solutions to parabolic $p$-Laplace systems has been established in the \emph{super-critical} regime $p>\tfrac{2N}{N+2}$.
The complementary \emph{sub-critical} regime $1<p\le\tfrac{2N}{N+2}$ has remained completely open until now. 
This article closes that gap by developing a quantitative regularity theory for weak solutions $u\colon E_T\to\mathbb{R}^k$ to
\begin{equation}
    \label{eq-par-gen}
  \partial_t u - \operatorname{div}\,\sfA(x,t,u,Du)
  \;=\; \sff - \operatorname{div}\!\big(|\sfF|^{p-2}\sfF\big)
  \quad\text{in}\> E_T,
\end{equation}
with $E_T\equiv E\times(0,T]$, $E$ open in $\R^N$, and source terms $\sfF\in L^{qp}_{\mathrm{loc}}(E_T,\mathbb{R}^{kN})$ and $\sff\in L^{qp'}_{\mathrm{loc}}(E_T,\mathbb{R}^k)$, in the range
$$
  1<p\le p_c:=\frac{2N}{N+2},\qquad q>\frac{N+p}{p}.
$$
The vector field $\sfA\colon E_T\times\mathbb{R}^k\times\mathbb{R}^{kN}\to\mathbb{R}^{kN}$ is Carathéodory -- measurable in $(x,t)$, continuous in $(u,\xi)$ -- 
and satisfies the following: for structural constants $0<C_o\le C_1$, a fixed parameter $\mu\in [0,1]$, a.e. $(x,t)\in E_T$ and all $(u,\xi)\in\mathbb{R}^k\times\mathbb{R}^{kN}$,
\begin{equation}\label{growth-a*}
\left\{
\begin{array}{c}
	\sfA(x,t,u,\xi)\cdot\xi\ge C_o \big(\mu^2 +|\xi|^2\big)^\frac{p-2}{2}|\xi|^2\, ,\\[6pt]
	| \sfA(x,t,u,\xi)|\le C_1\big(\mu^2 +|\xi|^2\big)^\frac{p-2}{2}|\xi|\, ,\\[6pt]
    \displaystyle \sum_{\alpha =1}^N
    \sfA_\alpha(x,t,u,\xi)
    \cdot u\, \xi_\alpha \cdot u\ge 0,\\
\end{array}
\right.
\end{equation}
where $\sfA=\big(\sfA_1,\dots ,\sfA_N\big)$
and $\xi=\big(\xi_1,\dots,\xi_N\big)$
with $\sfA_\alpha,\xi_\alpha\in \R^k$.
The quantity $\mu$ measures the nondegeneracy
($\mu=0$ corresponds to the degenerate case). The first two inequalities in \eqref{growth-a*} are the standard $p$-structure conditions -- coercivity and $p$-growth -- in the gradient variable $\xi$. The third line encodes a quasi-diagonal (directional) monotonicity of the nonlinearity along $u$. The directional monotonicity assumption in \eqref{growth-a*} goes back, in the context of $L^{\infty}$-estimates for elliptic systems, to Meier~\cite{Meier}. See \cite[Appendix~A]{BDGL-25} for its significance. In the scalar case, the directional monotonicity condition is redundant: 
it follows directly from the coercivity inequality in~\eqref{growth-a*}.

Two principal conclusions are obtained. First, local boundedness of solutions in the sub-critical regime can be estimated by the $L^{\sfr}$-integrability of solutions for some suitable $\sfr>2$ and the size of the source terms $\sfF$ \& $\sff$. 
We postpone the rigorous definition of solutions and notation to \S~\ref{SS:2}, and proceed with the statement of the first main result. 

\begin{theorem}\label{thm:sup-quant}
Let $p\in\bigl(1,\tfrac{2N}{N+2}\bigr]$, $q>\tfrac{N+p}{p}$, and $\sfr>2$ be such that
\begin{equation}\label{Eq:lm>0}
\boldsymbol\lambda_{\sfr}:=N(p-2)+p\sfr>0.
\end{equation}
Assume that $u\in L^{\sfr}_{\mathrm{loc}}(E_T,\mathbb{R}^k)$ is a weak solution in $E_T$ to \eqref{eq-par-gen} under \eqref{growth-a*}, with $\sfF\in L^{qp}_{\mathrm{loc}}(E_T,\mathbb{R}^{kN})$ and $\sff\in L^{qp'}_{\mathrm{loc}}(E_T,\mathbb{R}^k)$. Then, we have
$$
|u|\in L^\infty_{\mathrm{loc}}(E_T).
$$
Moreover, there exists a constant $C=C(N,p,C_o,C_1,q,\sfr)$ such that for any cylinder $Q_{\varrho,\theta}\equiv B_\rho(x_o)\times(t_o-\theta, t_o]$ in $ E_T$ and any $\tau\in\bigl(0,1\bigr)$, 
\begin{align*}
    \sup_{Q_{\tau\varrho,\tau\theta}}
    |u|
    &\le
    C
    \bigg[
    \frac{1}{(1-\tau)^{N+p}}
     \Big(\frac{\varrho^p}{\theta}\Big)^\frac{N}{p}
    \biint_{Q_{\varrho,\theta}} |u|^{\sfr}
    \,\dx\dt
    \bigg]^\frac{p}{\boldsymbol\lambda_\sfr}\\
     &
    \phantom{\le\,}
    +C 
     \bigg[
    \Big(\frac{\theta}{\varrho^p}\Big)^{q-\frac{N}{p}}
    \biint_{Q_{\varrho,\theta}} |\varrho \,\sfF|^{qp}\,\dx\dt
    \bigg]^\frac{p}{\boldsymbol\lambda_{2q}}
   \\
    &\phantom{\le\,}+
    C \bigg[
    \Big(\frac{\theta}{\rho^p}\Big)^{p'(q-\frac{N}{p})}
    \biint_{Q_{\varrho,\theta}}|\rho^p \,\sff|^{qp'}\,\dx\dt
    \bigg]^\frac{p}{p'\boldsymbol\lambda_{q}}
    +
    C\Big(\frac{\theta}{\varrho^p}\Big)^\frac{1}{2-p}
    +
    C\mu\rho.
\end{align*}
\end{theorem}

\begin{remark}\label{Rmk:stability-1}
\upshape
The constant $C(N,p,C_o,C_1,q,\sfr)\to\infty$ either as $p\to1$ or as $q\to\frac{N+p}p$, and it stays stable as $p\uparrow \frac{2N}{N+2}$ for any fixed $\sfr>2$.
The situation becomes more delicate as $\sfr\to\frac{N(2-p)}p$, that is, as $\boldsymbol\lambda_\sfr\to0$.
Indeed, on the one hand, when $\boldsymbol\lambda_\sfr\to0$, the iteration in the proof of Lemma~\ref{lm:sup-est-qual} does not converge, and one cannot conclude -- by this argument -- that $u$ is bounded, even in a purely qualitative sense. On the other hand, under the assumption $u\in L^\sfr$, the final part of the proof of Theorem~\ref{thm:sup-quant} starts from a qualitative
$L^\infty$-bound and turns it into a quantitative one.
For technical reasons, two different ranges of $\sfr$
must be considered, namely either $2\le\frac{N(2-p)}p<\sfr<p'+2$ or $p'+2\le\sfr$.
If we restrict our attention to the former case, the proof shows that the constant $C(N,p,C_o,C_1,q,\sfr)\to\infty$
as $\sfr\to\frac{N(2-p)}p$.
This phenomenon is related to the difficulty of obtaining a quantitative $L^\infty$-bound when $p=p_c$, as mentioned below in Remark~\ref{Rmk:1:3}.
\end{remark}

The set $\{N,\,p,\,C_o,\,C_1,\,q,\,\sfr\}$ is referred to as the \emph{structural data}.  The intrinsic geometry pertinent to system \eqref{eq-par-gen} will be reflected in the $L^{\infty}$-estimate upon choosing the size of $Q_{\rho,\theta}$.
This will be crucial for our second result, that is, the gradient of solutions enjoys a Gehring–type self-improvement.
%As for the improvement of the gradient regularity we have the following result.

\begin{theorem}\label{thm:higherint}
Let $N\ge2$, $p\in\big(1,\tfrac{2N}{N+2}\big]$, $q>\tfrac{N+p}{p}$, $\mu=0$,
and let $\sfr>2$ be such that \eqref{Eq:lm>0} is satisfied.
Moreover, assume in the vector-valued case ($k>1$) that the vector field $\sfA$ is independent of $u$, that is, $\sfA=\sfA(x,t,\xi)$. 
Then, there exists $\varepsilon_o=\varepsilon_o(N,p,C_o,C_1,q,\sfr)\in(0,1]$
with the following property.  
If $\sfF\in L^{qp}_{\mathrm{loc}}(E_T,\R^{kN})$, $\sff\in L^{qp'}_{\mathrm{loc}}(E_T,\R^{kN})$, 
and $u$ is a weak solution in $E_T$ to \eqref{eq-par-gen} under \eqref{growth-a*} satisfying 
$u\in L^{\sfr}_{\mathrm{loc}}(E_T,\R^k)$, then
\[
|Du|\in L^{p(1+\varepsilon_o)}_{\mathrm{loc}}(E_T).
\]
Moreover, for every $\varepsilon\in(0,\varepsilon_o]$ and every cylinder 
$
Q_{8R}\equiv B_{8R}(x_o)\times (t_o-(8R)^{2},\,t_o+(8R)^{2})$ in  
$E_T$,
the following quantitative higher–integrability estimate holds:
\begin{align*}%\label{eq:higher-int}
	\biint_{Q_{R}} &
	|Du|^{p(1+\varepsilon)} \,\dx\dt \nonumber\\
    &\le
	C\, \sfH(8R)^{p\varepsilon\frac{\sfr}2}
	\biint_{Q_{2R}}
	|Du|^{p} \,\dx\dt +
	C\bigg[\biint_{Q_{2R}}  
	\Big(|\sfF|^{qp} +|R\,\sff|^{qp'}\Big)\,\dx\dt \bigg]^{\frac{1+\varepsilon}{q}},
\end{align*}
where
\begin{equation*}
    \sfH(8R):=
    1
    + \bigg[\biint_{Q_{8R}} 
 	\frac{|u-(u)_{8R}|^{\sfr}}{(8R)^{\sfr}} \,\dx\dt \bigg]^{\frac{2}{\boldsymbol\lambda_\sfr}}
 	+ \bigg[\biint_{Q_{8R}} 
 	\big(|\sfF|^{qp}+|R\,\sff|^{qp'}\big)\,\dx\dt  
    \bigg]^{\frac{\sfr}{q\boldsymbol\lambda_\sfr}},
\end{equation*}
and $(u)_{8R}$ denotes the average of $u$ on $Q_{8R}$, and $C$ depends only on $N,p,C_o,C_1$, $q,\sfr$.
\end{theorem}

\begin{remark}
\upshape In Theorem~\ref{thm:higherint} we assume $\mu=0$ mainly for simplicity, but a similar result could be proved also for $\mu\in(0,1]$.    
\end{remark}

\begin{remark}\label{Rmk:stability-2}
\upshape Since the proof of Theorem~\ref{thm:higherint} relies heavily on Theorem~\ref{thm:sup-quant}, the stability of $\varepsilon_o$ is strongly affected by the asymptotic behavior of the constant $C$, as discussed in Remark~\ref{Rmk:stability-1}. In particular, tracing the dependence of the constants and parameters on the structural data in the proof of Theorem~\ref{thm:higherint} shows that $\varepsilon_o\to0$ not only as either $p\to1$ or $q\to\frac{N+p}p$, but also as $\sfr\to\frac{N(2-p)}p$. Note for fixed $\sfr>2$, $\epsilon_o$ remains positive as $p\uparrow \frac{2N}{N+2}$.
\end{remark}

\begin{remark}\label{Rmk:1:3}\upshape
In the borderline case $p=p_c=\frac{2N}{N+2}$, the situation becomes particularly delicate. 
For the scalar parabolic $p$-Laplace equation ($k=1$), it has been shown in~\cite{CHSS} that weak solutions are locally bounded in a qualitative sense, although a quantitative $L^{\infty}$-estimate in terms of the $L^2$-norm is not known so far. 
This result implies that, in the model case, the assumption $u\in L^{\sfr}_{\mathrm{loc}}(E_T)$ for some $\sfr>2$ is automatically fulfilled, and can therefore be omitted in Theorem~\ref{thm:higherint}. It is furthermore plausible that the argument extends to more general parabolic $p$-Laplace systems of similar structure as in \eqref{eq-par-gen}, although such an extension has not yet been established rigorously.
\end{remark}

\subsection{Novelty and significance}%\label{SS:sign}
The phenomenon known as \emph{higher integrability} -- that is, the automatic gain of summability for the gradient beyond what is prescribed by the natural energy space -- has long been recognized as a cornerstone in the transition from quantitative \emph{a priori} control to qualitative regularity.  
Its conceptual roots trace back to Gehring’s fundamental work~\cite{Gehring}, and subsequent developments by Meyers and Elcrat~\cite{Meyers-Elcrat} extended the principle to general elliptic systems.  
After later refinements resolving technical issues concerning localization and supports~\cite{Stred,Giaq-Mod}, the method became a standard element of modern regularity theory (see also~\cite[Ch.~V, Thm.~2.1]{Giaquinta:book}, \cite[\S~6.4]{Giusti:book}).  
In the parabolic setting, a first coherent treatment was provided by Giaquinta and Struwe~\cite{Giaquinta-Struwe}, who handled quasi-linear systems with measurable coefficients.  
A decisive step forward was achieved by Kinnunen and Lewis~\cite{Kinnunen-Lewis:1}, who proved higher integrability for wide classes of parabolic $p$-Laplace-type systems under the \emph{super-critical} condition $p>\tfrac{2N}{N+2}$.

In contrast, in the \emph{sub-critical}  regime $1<p\le\tfrac{2N}{N+2}$, solutions can display an untamed analytic behavior.  
In fact, the natural scaling and compactness structures underlying the parabolic $p$-Laplacian deteriorate, and a comprehensive regularity theory has remained out of reach; see, for instance,~\cite{DiBe} and the discussion in~\cite[Appendices~A--B]{DBGV-book}.  The present work identifies a natural class of weak solutions to parabolic $p$-Laplace systems in the sub-critical regime and establishes a unified regularity theory for them that parallels the one in the super-critical range.
Precisely, we establish a quantitative reverse Hölder framework that yields
\[
|Du|\in L^{p(1+\varepsilon)}_{\mathrm{loc}},
\]
for some $\varepsilon>0$ depending solely on the structural data, and complement it by a sharp local $L^\infty$-estimate for $u$ within the sub-critical regime.

To the best of our knowledge, the first $L^\infty$-estimates for parabolic systems of $p$–Laplace-type with a general structure comparable to that considered here were established in \cite{Giorgi-OLeary}. This extended earlier work of O'Leary~\cite{OLeary}, which was restricted to the simpler scalar case. In particular, the directional monotonicity assumption \eqref{growth-a*}$_3$ appears already in these contributions. 
Although the authors pursue a two–step argument reminiscent of ours, their analysis does not yield the explicit and fully quantitative bounds obtained in Theorem~\ref{thm:sup-quant}; instead, they conclude only a qualitative boundedness statement.

Regarding the higher integrability of $Du$ for parabolic systems of $p$–Laplace-type, the structure $\sfA=\sfA(x,t,\xi)$ has a tradition -- see, for instance, \cite{Kinnunen-Lewis:1}. From a technical point of view, the absence of $u$-dependence ensures that if $u$ is a solution, then so is $u+$constant, a property required in Lemma~\ref{lem:intr-sup} and in several subsequent arguments that depend critically upon it. 
It remains unclear whether this restriction reflects an intrinsic structural necessity or merely a limitation of the present method.

\subsection{Sharpness of the result}
Now, we demonstrate some examples that test the sharpness of our results -- boundedness and higher integrability. Consider the prototype, scalar equation
\begin{equation}\label{P-lapl}
\partial_t u - \Div\big(|\nabla u|^{p-2}\nabla u\big)=0%\sff-\Div\!\big(|\sfF|^{p-2}\sfF\big)
\end{equation}
in \(\R^N\times\R\). %for $\sff=0$ and $\sfF=0$.  
We first discuss {\bf boundedness}.
Unbounded weak solutions occur when \(1<p<p_c:=\frac{2N}{N+2}\). 
In particular, the function
\begin{equation}\label{Ex:1}
\begin{array}{c}
{\dsty u(x,t)=\left[|\lm|\left(\frac{p}{2-p}\right)^{p-1}\right]^{\frac1{2-p}}
\frac{(T-t)_+^{\frac{1}{2-p}}}{|x|^{\frac{p}{2-p}}}},\\
{}\\
{\dsty 1<p<\frac{2N}{N+2},\qquad
\lm=N(p-2)+p}
\end{array}
\end{equation}
is a nonnegative local weak solution to \eqref{P-lapl}.
This solution is finite for \(t<T\) and \(x\neq0\), but blows up as \(x\to0\).  
This example shows that the threshold \(p>p_c\) is almost optimal for boundedness.

A second counterexample to boundedness for $1<p<\tfrac{2N}{N+2}$ can be constructed by seeking solutions to~\eqref{P-lapl} of the self--similar form
\begin{equation}\label{self-sim-II}
    u(x,t)= g\big(|x|(T-t)_+^{\beta}\big),
\end{equation}
where $T\in\mathbb{R}$ and the exponent $\beta\in\mathbb{R}$ is to be determined.
Here $g$ denotes a positive, radially symmetric, monotonically decreasing function, to be computed below.
If we assume $g$ to be regular and smooth, it is straightforward to verify that $g=g(r)$ must satisfy the ordinary differential equation
\begin{equation*}
 \frac1p r g'+\frac{N-1}{r}(-g')^{p-1}-(p-1)(-g')^{p-2}g''=0.
\end{equation*}
Here the sign in front of $g'$ reflects the fact that $g'\le 0$, and $\beta=-\frac1p$. The equation can be re-written as
\begin{equation}\label{Eq:radial-hom}
    \frac{r}{p}
    -\frac{N-1}{r}(-g')^{p-2}
    +(p-1)(-g')^{p-3}g''=0,
\end{equation}
and the substitution $\displaystyle y:=(-g')^{p-2}$ transforms \eqref{Eq:radial-hom} into the first–order linear ordinary differential equation
\begin{equation}\label{eq-for-y}
    y'
    =\frac{p-2}{p-1}\bigg[\frac{r}{p}-\frac{N-1}{r}y\bigg],
\end{equation}
which admits an explicit solution. Since $1<p<\tfrac{2N}{N+2}$, a straightforward integration of \eqref{eq-for-y} yields
\[
y = r^2\left[C\, r^{\frac{|\lambda|}{p-1}} + \frac{2-p}{p|\lambda|}\right],
\]
where $\lambda$ is as above, and $C>0$ is an arbitrary constant, taken positive for simplicity of presentation. Reverting to $g'$, we obtain
\[
g'(r)
=-\,r^{-\frac{2}{2-p}}
\left[
C\,r^{\frac{|\lambda|}{p-1}}
+\frac{2-p}{p|\lambda|}
\right]^{-\frac{1}{2-p}}.
\]
This expression coincides with the formula given in \cite[\S~1,~(1.9)]{bidaut}, 
within the broader analysis of explicit and semi-explicit solutions to the parabolic $p$-Laplacian in the singular range $1<p<2$. We refrain from detailing the explicit computation of $g$; for our purposes it suffices to observe that
\begin{equation*}
    g(r)
    \simeq 
    \bigg[\Big(\frac{p}{2-p}\Big)^{p-1}|\lambda|\bigg]^{\frac{1}{2-p}} 
    r^{-\frac{p}{2-p}}
    \qquad\text{as}\>\> r\downarrow 0,
\end{equation*}
which in turn implies
\begin{equation*}
    u(x,t)
    \simeq 
    \bigg[\Big(\frac{p}{2-p}\Big)^{p-1}|\lambda|\bigg]^{\frac{1}{2-p}}
    \left[\frac{(T-t)_+}{|x|^{p}}\right]^{\frac{1}{2-p}}
    \qquad\text{as}\>\> \frac{|x|}{(T-t)_+^{1/p}}\downarrow 0.
\end{equation*}
Notice that the behavior of $u$ near the origin coincides exactly with that of~\eqref{Ex:1}. 
The restriction $1<p<\tfrac{2N}{N+2}$ ensures that both $u$ and its spatial gradient possess 
the integrability required for the natural weak formulation of the problem, 
that is, $u$ belongs to the energy class associated with parabolic $p$-Laplace equation, and therefore represents a genuine weak solution to~\eqref{P-lapl} in the sense of 
Definition~\ref{def:weak_solution}.
 
As explained in Remark~\ref{Rmk:1:3}, when $p=\frac{2N}{N+2}$ and for the scalar equation ($k=1$), solutions are locally bounded. However, it is unclear whether a quantitative estimate as stated in Theorem~\ref{thm:sup-quant} can be achieved, or there exists a sequence of solutions that are uniformly bounded in the function space \eqref{func-space} but whose limit turns out to be essentially unbounded. Needless to say, one would like to understand the vector-valued case ($k>1$) as well.

We now turn to the question of {\bf higher integrability}. As pointed out in Remark~\ref{Rmk:1:3}, our approach actually shows, at least in the scalar equation case, that $\nabla u$ admits a higher integrability estimate that is uniform for all $p\ge\frac{2N}{N+2}$. However, in general one cannot expect such a uniform higher integrality for $p\le\frac{2N}{N+2}$ with the mere notion of solution. 
Indeed, a direct computation shows that the solution \eqref{Ex:1} satisfies
\[
\nabla u\in L^{p+\varepsilon}_{\mathrm{loc}}(\R^N\times\R)
\quad\text{for all }\varepsilon\in(0,\varepsilon_p),
\qquad 
\varepsilon_p\equiv-\frac{N(p-2)+2p}{2}>0,
\]
and that \(\varepsilon_p\downarrow0\) as \(p\uparrow p_c\).

This example also hints that the issue of higher integrability is intimately connected to the issue of boundedness.
To illustrate this point, 
suppose that $p\in(1,\frac{2N}{N+2}]$ and $u$ is a non-negative local weak solution to \eqref{P-lapl} satisfying $\nabla u\in L^{p+\varepsilon}_{\loc}(E_T)$ for some {\it a priori} determined $\varepsilon$ with $\varepsilon\in[\varepsilon_*,1]$ for some fixed $\varepsilon_*>0$ independent of $p$.
 Then, by the parabolic embedding theorem, we have $u\in L^{\sfr}_{\loc}(E_T)$ for $\sfr=(p+\varep)\frac{N+2}{N}$. By a straightforward algebraic computation, one finds a positive number 
$\delta=\delta(N,\varepsilon_*)>0$ such that, whenever 
$
  p\in(p_c-\delta,\;p_c]$,
the quantity \(N(p-2)+p\sfr\) remains strictly positive.
Indeed, recalling the definition of \(p_c\) and using \(p_c-\delta<p\le p_c\), one readily computes
\begin{align*}
    N(p-2)+p\sfr 
    &> N(p_c -2 -\delta) + (p_c-\delta)(p_c-\delta+\varepsilon)\tfrac{N+2}{N}\\
    &=\underbrace{N (p_c-2)+2\tfrac{2N}{N+2}}_{=0}
      -N\delta-2\delta+2(\varepsilon-\delta)
      -\delta(\varepsilon-\delta)\tfrac{N+2}{N}\\
    &=-(N+4)\delta+2\varepsilon
      -\delta(\varepsilon-\delta)\tfrac{N+2}{N}.
\end{align*}
To ensure that the last quantity is positive, it suffices, for instance, to take
\[
    \delta=\frac{\varepsilon_* N}{N(N+4)+N+2}.
\]
Since, in this range of \(p\), one also has \(\sfr>2\), it follows from Lemma~\ref{lm:sup-est-qual} that such a solution 
(for $p\in(p_c-\delta,\;p_c]$) must in fact be locally bounded in \(E_T\). In principle this argument also applies to the systems as in \eqref{eq-par-gen}, and justifies considering locally bounded solutions.

Second, consider equation~\eqref{P-lapl} with a forcing term $\sff$. 
Once again, we seek self-similar solutions to the form~\eqref{self-sim-II}, with the specific choice $\beta=-\tfrac1p$. 
In this setting, we prescribe
\[
\sff(x,t)=\frac1{(T-t)_+}\,f\Big(|x|(T-t)_+^{-\frac1p}\Big),
\]
where $f$ is a given radial function.
Under these assumptions, a direct computation shows that the corresponding ordinary differential equation for $g$ reads
\[
 \frac{r}{p}\, g'(r)
 + \frac{N-1}{r}\,(-g'(r))^{p-1}
 - (p-1)\,(-g'(r))^{p-2} g''(r)
 = f(r).
\]
If $p=\frac{2N}{N+2}$, and $f\colon (0,\tfrac12)\to\mathbb{R}$ is given by
\[
f(r)=r^{-\frac{N}{2}}(-\ln r)^{-\frac{N-2}{4}}
\left[\frac{N}{2}+\frac{N^{2}-4N-4}{4N}(-\ln r)^{-1}\right],
\]
then, the corresponding solution 
$u\colon B_{\frac12}\times(0,T]\to\mathbb{R}$ 
takes the form
\[
u(x,t)=g\Big(|x|(T-t)_+^{-\frac1p}\Big),
\]
where $g:(0,\tfrac12)\to\mathbb{R}$ is a primitive of
\[
g'(r)=-r^{-\frac{N+2}{2}}(-\ln r)^{-\frac{N+2}{4}}.
\]
Provided $N>4$, it can be verified that 
\[
u\in L^{p}\big(0,T;W^{1,p}(B_{\frac12})\big)
\quad\text{and}\quad 
\sff\in L^{2}\big(B_{\frac12}\times(0,T)\big),
\]
and that $u$ is a weak solution to the  parabolic $p$-Laplace equation with a source term
\[
\partial_t u - \Div\big(|\nabla u|^{p-2}\nabla u\big)=\sff
\quad\text{in }B_{\frac12}\times(0,T).
\]
At the same time, $\nabla u$ fails to exhibit any form of higher integrability. The restriction $N>4$
is essential to ensure that the source term
$\sff$ belongs to $L^2$.

\subsection{Strategy of the proof}
From the examples discussed in the preceding subsection, it is apparent that the sub--critical regime exhibits substantially subtler behavior. 
On the technical side, the main obstruction arises from the fact that the energy inequality (see~\S~\ref{sec:energy}), when combined with the parabolic Sobolev embedding, fails to provide any quantitative gain in the form of a reverse H\"older inequality. 
More precisely, since $p\le\frac{2N}{N+2}<2$, the dominant contribution on the right-hand side of the energy inequality comes from the $L^2$-term. 
In light of the parabolic Sobolev embedding (Lemma~\ref{lem:gag} below), the space 
$
L^\infty_t L^2_x \cap L^q_t W^{1,q}_x
$
embeds into $L^2$ only when $q\ge \tfrac{2N}{N+2}$, and hence $q\ge p$. 
As a consequence, this standard approach cannot yield a reverse H\"older inequality in the sub--critical range.

This inherent difficulty is overcome by employing a quantified supremum bound for the solution. 
The idea of using such bounds in the derivation of gradient estimates originates in~\cite{Saari-Schwarzacher}. 
For this reason, we establish sharp quantitative sup-estimates in \S~\ref{sec:degiorgi}, which are of independent interest. 
These estimates are, however, intrinsically non-homogeneous with respect to the solution. 
Consequently, any subsequent argument that relies on these bounds, including the proof of the reverse H\"older inequality, inherits such a non-homogeneous dependence on $u$.

This necessitates the introduction of an intrinsic geometry that depends on the solution itself: all estimates are carried out on parabolic cylinders whose space-time scaling includes the magnitude of $u$. 
At the same time, the underlying PDE is non-homogeneous with respect to the spatial gradient $Du$. 
It has been known since the seminal works of DiBenedetto and Friedman~\cite{DiBenedetto_Holder} and of Kinnunen and Lewis~\cite{Kinnunen-Lewis:1} that, in this context, gradient estimates must be performed on cylinders intrinsic to $Du$. 
This interplay between two distinct intrinsic 
scalings -- one involving $u$ and the other its gradient
-- creates a substantial technical obstacle: it precludes the use of cylinders intrinsic to a single quantity alone. In this sense, we operate within a framework of \emph{mixed intrinsic scaling}: 
our scaling parameter simultaneously controls the $L^{\sfr}$-oscillation of $u$ and is bounded in terms of the $L^{p}$-norm of $Du$ and the source terms; see~\S~\ref{sec:rev-hoel}. 
Within this setting, we establish a reverse H\"older inequality on the mixed intrinsic cylinders. 
This, however, comes at the expense of a highly delicate construction: 
the definition and selection of these cylinders are carried out in~\S\S~\ref{sec:cylinders}--\ref{SS:selection-radii}. 
The resulting structure yields a weak-type estimate for the super-level sets of $|Du|$, 
which ultimately leads to the higher integrability of the spatial gradient.

\vspace{.5cm}

{\bf Acknowledgments.} 
U. Gianazza belongs to the Gruppo Nazionale per l'Analisi Matematica, la Probabilità e le loro Applicazioni (GNAMPA) of the Istituto Nazionale di Alta Matematica (INdAM), and was partly supported by the GNAMPA-INdAM  Project 2023 ``Regolarità per problemi ellittici e parabolici con crescite non standard" (CUP\_E53C22001930001). This research was funded in whole or in part by the Austrian Science Fund (FWF) [10.55776/P36272]. For open access purposes, the author has applied a CC BY public copyright license to any author accepted manuscript version arising from this submission.

\section{Notation, definitions and tools}\label{SS:2}
Throughout the paper, $\mathbb{N}_0$ denotes $\mathbb{N}\cup\{0\}$, and $\mathbb{R}^k$, $\mathbb{R}^N$ are Euclidean spaces.  
Given a bounded open set $E\subset\mathbb{R}^N$ and $T>0$, we set $E_T:=E\times(0,T]$.  
For a \emph{scalar} function $v\colon E_T\to\mathbb{R}$, $\nabla v$ denotes its spatial gradient;  
for a \emph{vector-valued} function $u\colon E_T\to\mathbb{R}^k$, we write $D u$ for its spatial gradient.  
  
A space–time point is denoted by $z_o=(x_o,t_o)\in\mathbb{R}^N\times\mathbb{R}$, with $N\in\mathbb{N}$.  
By $B_\rho(x_o)$ we denote the open ball in $\mathbb{R}^N$ centered at $x_o$ with radius $\rho>0$;  
when $x_o=0$, we simply write $B_\rho$. For $R,S>0$, we define the \emph{backward parabolic cylinder}
\[
    Q_{R,S}(z_o):=B_R(x_o)\times(t_o-S,t_o].
\]
Occasionally, we abbreviate $\sigma Q_{R,S}(z_o):=Q_{\sigma R,\sigma S}(z_o)$ for $\sigma>0$.
For $\lambda>0$ and $z_o=(x_o,t_o)$, we introduce the \emph{intrinsic parabolic cylinder}
\begin{equation*}%\label{def:cyl-lambda}
	Q_\rho^{(\lambda)}(z_o)
	:=
	B_\rho^{(\lambda)}(x_o)\times\Lambda_\rho(t_o),
\end{equation*}
where
\[
    B_\rho^{(\lambda)}(x_o)
    :=
    \{x\in\mathbb{R}^N:\ |x-x_o|<\lambda^{\frac{p-2}{2}}\rho\},
    \qquad
    \Lambda_\rho(t_o)
    :=
    (t_o-\rho^2,t_o+\rho^2).
\]
When $\lambda=1$, the superscript is omitted, and when $z_o=(0,0)$ or when the context is clear, the reference to $z_o$ will also be dropped.

We record an elementary iteration lemma (cf. \cite[Lemma 6.1]{Giusti:book}) that will be repeatedly used to absorb error terms along shrinking radii.
\begin{lemma}\label{lem:tech-classical}
Let $0<\eps<1$, $A,B\ge 0$ and $\alpha\ge 0 $. Then there exists a universal constant  $C = C(\alpha,\eps)$
such that whenever $\phi\colon [R_o,R_1]\to \R$  is a non-negative bounded function satisfying
\begin{equation*}
	\phi(t)
	\le
	\eps \phi(s) + \frac{A}{(s-t)^\alpha}  + B
	\qquad \text{for all $R_o\le t<s\le R_1$,}
\end{equation*}
we have
\begin{equation*}
	\phi(t)
	\le
	C\,  \bigg[\frac{A}{(s-t)^\alpha}  + B\bigg]
\end{equation*}
for any $R_o\le t<s\le R_1$.
\end{lemma}

The following lemma \cite[Chapter I, Lemma 4.1]{DiBe} provides a basic decay criterion for a nonlinear recurrence. 
It asserts that, under a suitable smallness condition on the initial value, the sequence converges to zero at a geometric rate. Such statements are classical tools in iterative schemes, particularly in De Giorgi-type arguments for deriving
$L^\infty$-bounds. 
\begin{lemma}\label{it-lemma}
Let $(\boldsymbol{\mathsf Y}_i)_{i \in \N_0}$ be a sequence of non-negative numbers satisfying
\begin{align*}
	\boldsymbol{\mathsf Y}_{i+1}
	\leq
	C\, \boldsymbol{\mathsf b}^i \boldsymbol{\mathsf Y}_i^{1+\alpha}
	\qquad\mbox{for all $i \in \N_0$}
\end{align*}
for some positive constants $C, \alpha$
and $\boldsymbol{\mathsf b} >1$.
If
\begin{align*}
	\boldsymbol{\mathsf Y}_o
	\le
	C^{-\frac{1}{\alpha}} \boldsymbol{\mathsf b}^{-\frac{1}{\alpha^2}},
\end{align*}
then $\boldsymbol{\mathsf Y}_i \to 0$ as $i \to \infty$.
\end{lemma}

For a map $u\in L^1(0,T;L^1(E,\R^k))$ and given measurable sets $A\subset E$  and $G\subset E_T$ with positive Lebesgue measure, the slicewise mean $\langle u\rangle_{A}\colon (0,T)\to \R^k$ of $u$ on  $A$ is defined
by
\begin{equation*}
	\langle u\rangle_{A}(t)
	:=
	\mint_{A} u(t)\,\dx,
	\quad\mbox{for a.e.~$t\in(0,T)$,}
\end{equation*}
whereas  the mean value $(u)_{G}\in \R^k$ of $u$ on  $G$ is defined by
\begin{equation*}
	(u)_{G}
	:=
	\biint_{G} u\,\dx\dt.
\end{equation*}
Note that if $u\in C^0((0,T);L^2(E,\R^k))$, the slicewise means are defined for any $t\in (0,T)$. 
If  $A$ is a ball $B_\rho^{(\lambda)}(x_o)$,  we write $\langle
u\rangle_{x_o;\rho}^{(\lambda)}(t):=\langle
u\rangle_{B_\rho^{(\lambda)}(x_o)}(t)$. Similarly,   if $G$ is a
cylinder of the form $Q_\rho^{(\lambda)}(z_o)$, we use the shorthand
notation $(u)^{(\lambda)}_{z_o;\rho}:=(u)_{Q_\rho^{(\lambda)}(z_o)}$.

A convenient device repeatedly used below is the elementary quasi–minimality of the average  $
\R^N\ni a\mapsto \int_A |u-a|^r \,\dx$.
For completeness we record it.

\begin{lemma}\label{lem:mean}
Let $r\ge1$, $A\subset\mathbb{R}^{N}$ be a bounded set of positive measure, and $u\in L^{r}(A,\mathbb{R}^{k})$. Then for every $a\in\mathbb{R}^{k}$,
$$
  \mint_{A}\!|u-(u)_{A}|^{r}\,\mathrm{d}x
  \;\le\;
  2^{r}\mint_{A}\!|u-a|^{r}\,\mathrm{d}x.
$$
\end{lemma}

We shall repeatedly interpolate between the $L^q$-energy and
the $L^r$-mass of $u$
by means of the following localized Gagliardo–Nirenberg inequality on balls, stated in the precise form needed below.

\begin{lemma}\label{lem:gag}
Let $1\le p<\infty$. There exists a constant
$c=c(N,p)$ such that for any ball $B_\rho(x_o)\subset\R^N$ with $\varrho>0$ and any function $u \in W^{1,p}(B_\rho(x_o))$, we have
\begin{equation*}
	\mint_{B_\rho(x_o)} \frac{|u|^q}{\rho^q} \,\dx
	\le
	c\bigg[\mint_{B_\varrho(x_o)} \frac{|u|^2}{\rho^2} \,\dx\bigg]
	^{\frac{q-p}{2}}\mint_{B_\rho(x_o)} \bigg[
	\frac{|u|^{p}}{\rho^p} + |Du|^p \bigg] \dx ,
\end{equation*}
where $q=p\frac{N+2}{N}$. 
\end{lemma}

We end this section with the notion of solution.
\begin{definition}\label{def:weak_solution}\upshape
Assume that the Carath\'eodory vector field
$\sfA\colon E_T\times \R^k\times\R^{kN}\to\R^{kN}$ satisfies \eqref{growth-a*}, that
$\sfF\in L^p_{\rm loc} (E_T,\R^{kN})$,
and that $\sff\in L^2_{\rm loc}(E_T,\R^k)$.
We identify a  measurable map
$u\colon E_T\to\R^{N}$ in the class
\begin{equation}\label{func-space}
	u\in C \big([0,T]; L^{2}(E,\R^k)\big) \cap 
	L^p\big(0,T;W^{1,p}(E,\R^k)\big)
\end{equation} 
as a \textit{weak solution} to the  parabolic system \eqref{eq-par-gen} if and only if the identity
\begin{align*}%\label{weak-solution}
	\iint_{E_T}\big[u\cdot\varphi_t - \sfA(x,t,u,Du)\cdot D\varphi\big]\dx\dt
    =
    -\iint_{E_T} \big[|\sfF|^{p-2}\sfF\cdot D\varphi 
    +\sff\cdot\varphi\big]\dx\dt 
\end{align*}
holds, for any  testing function $\varphi\in C_0^\infty(E_T,\R^k)$.
\hfill$\Box$
\end{definition}

\section{Local \texorpdfstring{$L^\infty$}{Linfinity}-control in the sub-critical regime}
We consider weak solutions to the quasilinear parabolic 
$p$-Laplace–type system \eqref{eq-par-gen} under structural conditions \eqref{growth-a*} and with the given source terms
$\sfF\colon E_T\to\R^{kN}$, $\sff\colon E_T\to\R^{k}$. The precise regularity assumptions on $\sfF$ and $\sff$ will be given in the following.

\subsection{Energy and Caccioppoli estimates}
A weighted Caccioppoli estimate is required to run the Moser iteration.
The lemma below furnishes precisely this input: the Lipschitz weight $\Phi(|u|^{2})$ and its primitive $v=\int_{0}^{|u|^{2}}\Phi(s)\,\ds$ legitimize testing with powers of $u$ against cut–offs and deliver the balance needed for the $L^{\sfr}\!\to\!L^{\infty}$ upgrade. This device is specific to Moser’s method; the De Giorgi approach will rely instead on unweighted truncation estimates.
\begin{lemma}\label{lem:energy-est}
Let $1<p\le2$ and $\mu\in[0,1]$. There exists a constant $C=C(p,C_o,C_1)$ with the following property. Suppose $u$ is a weak solution to \eqref{eq-par-gen} 
 in $E_T$ under the structural assumptions \eqref{growth-a*}, %\eqref{growth*}, 
and let $\Phi\colon\mathbb{R}_{\ge0}\to\mathbb{R}_{\ge0}$ be non-negative, bounded, and Lipschitz continuous, with
\begin{equation}
    \label{def:C_Phi}
    C_\Phi:=\sup\Bigl\{\frac{s\,\Phi'(s)}{\Phi(s)}:\ s\ge0,\ \Phi(s)>0\Bigr\}<\infty.
\end{equation}
If $Q_{R,S}(z_o)\Subset E_T$ and $\zeta\in C^\infty\big(Q_{R,S}(z_o),[0,1]\big)$ vanishes on the parabolic boundary of $Q_{R,S}(z_o)$, then
\begin{align}\label{en:Moser}\nonumber
     \sup_{\tau\in (t_o-S,t_o]}&\int_{B_R(x_o)\times\{\tau\}}
    \zeta^pv\,\dx
    +
    C_o \iint_{Q_{R,S}(z_o)}
    \zeta^{p} |Du|^p \Phi\big(|u|^2\big)\,\dx\dt\\\nonumber
    &\le
      C
    \iint_{Q_{R,S}(z_o)}
    \Big[|u|^p|\nabla\zeta|^p+\mu^p\zeta^p\Big]\Phi(|u|^2)
    \,\dx\dt\\  
    &\phantom{\le\,}
    +C\iint_{Q_{R,S}(z_o)}v|\partial_t\zeta^p|
    \,\dx\dt\\\nonumber
        &\phantom{\le\,}
    +
     C\Big( 1+ C_\Phi^\frac{p}{p-1}\Big)
    \iint_{Q_{R,S}(z_o)}|\sfF|^p\Phi\big( |u|^2\big)\zeta^p\,\dx\dt\\\nonumber
    &\phantom{\le\,}
    +
    C\iint_{Q_{R,S}(z_o)}|\sff||u|\Phi\big( |u|^2\big)\zeta^p\,\dx\dt,
\end{align}
where
\begin{equation*}
    v:=\int_{0}^{|u|^2} \Phi (s)\, \ds.
\end{equation*}
\end{lemma}

\begin{proof}
By translation, without loss of generality we may assume that $z_o=(0,0)$.
Fix $\tau\in(-S,0)$ and consider $\delta\in(0,-\tau)$. Let us test the weak form of \eqref{eq-par-gen} with
\[
\varphi(x,t):=u(x,t)\,\Phi\big(|u(x,t)|^2\big)\zeta(x,t)^p\chi(t),
\]
where $\chi$ is given by $\chi\equiv1$ on $(-S,\tau)$, $\chi\equiv0$ on $(\tau+\delta,0)$, and linear on $[\tau,\tau+\delta]$. A standard time–mollification argument justifies this choice of $\varphi$.
For the term containing the time derivative we compute, using integration by parts in $t$ and
\[
v:=\int_{0}^{|u|^2}\Phi(s)\,\ds,
\]
that
\begin{align*}
    \iint_{Q_{R,S}} u\cdot\partial_t\varphi\,\dx\dt
    &=
    -\tfrac12\iint_{Q_{R,S}} \partial_t\big(|u|^2\big)\Phi
    \big(|u|^2\big)\zeta^p\chi\,\dx\dt\\
    &=
    -\tfrac12\iint_{Q_{R,S}} \partial_t v\zeta^p\chi\,\dx\dt\\
    &=
    \tfrac12\iint_{Q_{R,S}} v\partial_t(\zeta^p\chi)\,\dx\dt.
\end{align*}
%Letting $\delta\downarrow 0$ will later produce the time–slice term
% \begin{equation*}
%     \sup_{\tau\in(-S,0]}\int_{B_R\times\{\tau\}}\zeta^p v\,dx.
% \end{equation*}
To estimate the contribution of the diffusion term, we first differentiate the testing function. Using $D_{\alpha}|u|^{2}=2D_{\alpha}u\cdot u$, we obtain
\begin{align*}
    D_\alpha \Big(u\Phi\big(|u|^2\big)\zeta^p\Big)
    &=
    D_\alpha u\Phi\big(|u|^2\big)\zeta^p
    +
    u\Phi'\big(|u|^2\big)D_\alpha |u|^2 \zeta^p
    +
    u\Phi\big(|u|^2\big)D_\alpha \zeta^p.
\end{align*}
Inserting this into the weak formulation of \eqref{eq-par-gen}, and recalling the definition of $v$, we arrive at
\begin{align}\label{initial_identity}
    -\tfrac12 \iint_{Q_{R,S}}v\partial_t(\zeta^p \chi) \,\dx\dt
    +\sfI_1+\sfI_2
    =
    -\sfI_3+\sfI\sfI_1+\sfI\sfI_2
\end{align}
where
\begin{align*}
    \sfI_1
    &=
    \iint_{Q_{R,S}} \sum_{\alpha =1}^N \sfA_\alpha(x,t,u,Du)\cdot  D_\alpha u
    \Phi\big(|u|^2\big)\zeta^p\chi
    \,\dx\dt,\\
    \sfI_2
    &= 2   
    \iint_{Q_{R,S}} \sum_{\alpha =1}^N
    \sfA_\alpha (x,t,u,Du)\cdot u D_\alpha u\cdot u
    \Phi'\big(|u|^2\big)\zeta^p\chi
    \,\dx\dt,\\
    \sfI_3
    &=
    \iint_{Q_{R,S}} 
    \sum_{\alpha =1}^N \sfA_\alpha (x,t,u,Du)
    \cdot u D_\alpha \zeta^p
    \Phi\big(|u|^2\big)\chi
    \,\dx\dt,\\
    \sfI\sfI_1
    &=
    \iint_{Q_{R,S}}
    |\sfF|^{p-2}\sum_{\alpha =1}^N \sfF_\alpha \cdot D_\alpha\big( u\Phi(|u|^2)\zeta^p\big)\chi\,\dx\dt,
\end{align*}
and
\begin{align*}
    \sfI\sfI_2
    &=
    \iint_{Q_{R,S}}
    \sff\cdot u\Phi(|u|^2)\zeta^p\chi\,\dx\dt.
\end{align*}
Here, we recall that $\zeta\in C^\infty\big(Q_{R,S},[0,1]\big)$ vanishes on the parabolic boundary of $Q_{R,S}$, and $\chi$ is the time cut–off specified above.  We begin with the diffusion part on the left-hand side, i.e., $\sfI_1+\sfI_2$. By the ellipticity in \eqref{growth-a*}\(_1\) and the directional monotonicity \eqref{growth-a*}\(_3\) (together with \(\Phi'\ge0\)),
\begin{align*}
    \sfI_1+\sfI_2
    &\ge
    C_o\iint_{Q_{R,S}} (\mu^2+|Du|^2)^{\frac{p-2}{2}}|Du|^2\Phi(|u|^2)\zeta^p\chi\,\dx\dt .
\end{align*}
Using the elementary identity
\[
(\mu^2+|Du|^2)^{\frac{p-2}{2}}|Du|^2
=(\mu^2+|Du|^2)^{\frac p2}-\mu^2(\mu^2+|Du|^2)^{\frac{p-2}{2}}
\]
and the fact that for \(1<p\le2\), we have \(\mu^2(\mu^2+|Du|^2)^{\frac{p-2}{2}}\le \mu^{p}\), we infer
\[
\sfI_1+\sfI_2
\ge
C_o\iint_{Q_{R,S}}(\mu^2+|Du|^2)^{\frac p2}\Phi(|u|^2)\zeta^p\chi\,\dx\dt
-
C_o\iint_{Q_{R,S}}\mu^{p}\Phi(|u|^2)\zeta^p\chi\,\dx\dt.
\]
Next we estimate the cut–off term $\sfI_3$. Using the growth bound \eqref{growth-a*}$_2$, we obtain
\begin{align*}
|\sfI_3|
&\le
\iint_{Q_{R,S}}
|\sfA(x,t,u,Du)||u||\nabla\zeta^{p}|\Phi(|u|^{2})\,\chi\,\dx\dt\\
&\le
C_{1}p\iint_{Q_{R,S}}
(\mu^{2}+|Du|^{2})^{\frac{p-2}{2}}|Du||u||\nabla\zeta|\zeta^{p-1}\Phi(|u|^{2})\chi\,\dx\dt\\
&\le
C_{1}p\iint_{Q_{R,S}}
(\mu^{2}+|Du|^{2})^{\frac{p-1}{2}}|u||\nabla\zeta|\zeta^{p-1}\Phi(|u|^{2})\chi\,\dx\dt,
\end{align*}
where in the last step we used $(\mu^{2}+|Du|^{2})^{\frac{p-2}{2}}|Du|\le(\mu^{2}+|Du|^{2})^{\frac{p-1}{2}}$. Using Young’s inequality with a parameter $\epsilon>0$, we have
\begin{equation*}
    C_1p (\mu^{2}+|Du|^{2})^{\frac{p-1}{2}}|u||\nabla\zeta| \zeta^{p-1}
    \le
    \epsilon\,(\mu^{2}+|Du|^{2})^{\frac{p}{2}}\zeta^{p}
    +
    (C_1p)^{p'}\epsilon^{-\frac1{p-1}}|u|^{p}|\nabla\zeta|^{p}.
\end{equation*}
Substituting this into the bound for $\sfI_3$ obtained above yields
\begin{align*}
    |\sfI_3|
    &\le
    \epsilon \iint_{Q_{R,S}}
    (\mu^{2}+|Du|^{2})^{\frac{p}{2}}\Phi(|u|^{2})\zeta^{p}\chi\,\dx\dt
    \\
    &\phantom{\le\,}
    +(C_{1}p)^{p'}\epsilon^{-\frac1{p-1}}
    \iint_{Q_{R,S}} |u|^{p}|\nabla\zeta|^{p}\Phi(|u|^{2})\chi\,\dx\dt,
\end{align*}
where $\epsilon>0$ will be fixed later so that the first term can be absorbed into the left.
Finally, we control the contributions of the source terms $\sfF$ and $\sff$.
Using the product rule for
$D_{\alpha}\big(u\Phi(|u|^{2})\zeta^{p}\big)$,
the term in the weak formulation involving $-\Div(|\sfF|^{p-2}\sfF)$ against the test function decomposes as
\begin{equation*}
    \sfI\sfI_1
    =
    \iint_{Q_{R,S}}
    \big[\sfJ_1+\sfJ_2+\sfJ_3\big]\chi\,\dx\dt,
\end{equation*}
where 
\begin{align*}
    \sfJ_1
    &:=
    \sum_{\alpha=1}^N |\sfF|^{p-2}\sfF_\alpha\cdot D_{\alpha}u\Phi(|u|^{2})\zeta^{p},\\
    \sfJ_2
    &:=p\sum_{\alpha=1}^N |\sfF|^{p-2}(\sfF_\alpha\cdot u)\Phi(|u|^{2})
    \zeta^{p-1}D_{\alpha}\zeta,\\
    \sfJ_3
    &:=2\sum_{\alpha=1}^N |\sfF|^{p-2}(\sfF_\alpha\cdot u)\,(D_{\alpha}u\cdot u)
    \Phi'(|u|^{2})\zeta^{p}.
\end{align*}
The terms $\sfJ_1$ and $\sfJ_2$ are readily estimated by Young's inequality, while $\sfJ_3$ requires a separate treatment. In fact, let $p'=\frac{p}{p-1}$. Since 
\begin{equation*}
    |u|^{2}\Phi'(|u|^{2})
    \le\;
C_\Phi\,\Phi(|u|^{2})
\end{equation*}
due to~\eqref{def:C_Phi}, we have
\begin{align*}
    |\sfJ_3|
    &=
    2\sum_{\alpha =1}^N |\sfF|^{p-1}\Big( \frac{\sfF_\alpha}{|\sfF|}\cdot u\Big)\big( D_\alpha u\cdot u\big)\Phi^\prime(|u|^{2 }) \zeta^p\\
    &\le
    2|\sfJ|^{p-1}
    |Du||u|^2\Phi^\prime(|u|^{2 }) \zeta^p
    \le
    2C_\Phi|\sfF|^{p-1}\Phi(|u|^{2 })
    |Du|\zeta^p.
\end{align*}
Applying Young's inequality with exponents $p$ and $p'$, and with $\epsilon >0$ gives
\begin{equation*}
    |\sfJ_3|
    \le
    \epsilon|Du|^p\Phi(|u|^{2 })\zeta^p
    +2^{p'} C_\Phi^{p'}\epsilon^{-\frac1{p-1}}|\sfF|^p\Phi(|u|^{2 }) \zeta^{p}.
\end{equation*}
Finally, replacing $|Du|^{p}$ by $(\mu^{2}+|Du|^{2})^{p/2}$  yields the form consistent with the coercive term:
\begin{equation*}
    |\sfJ_3|
    \le
    \varepsilon(\mu^{2}+|Du|^{2})^{\frac{p}{2}}\Phi(|u|^{2})\zeta^{p}
    +
    2^{p'}C_\Phi^{p'}
    \varepsilon^{-\frac1{p-1}}
    |\sfF|^{p}\Phi(|u|^{2})\zeta^{p}.
\end{equation*}
For $\sfJ_{1}$ we use $|\sfF|^{p-2}\sfF_\alpha\cdot D_\alpha u\le |\sfF|^{p-1}|Du|$ and Young’s inequality with exponents $p$ and $p'$ and with a parameter $\varepsilon>0$, obtaining the pointwise bound
\begin{align*}
    |\sfJ_{1}|
    &\le |\sfF|^{p-1}
    |Du|\Phi(|u|^{2})\zeta^{p}\\
    &\le 
    \varepsilon(\mu^{2}+|Du|^{2})^{\frac p2}\Phi(|u|^{2})\zeta^{p}
    + \varepsilon^{-\frac{1}{p-1}}|\sfF|^{p}\Phi(|u|^{2})\zeta^{p}.
\end{align*}
For $\sfJ_{2}$ we observe that
\begin{equation*}
    |\sfJ_{2}|
    \le p
    |\sfF|^{p-1}|u||\nabla\zeta|\,\zeta^{p-1}\Phi(|u|^{2}).
\end{equation*}
Then, Young’s inequality with exponents $p$ and $p'$ gives
\begin{equation*}
    |\sfJ_{2}|
    \le 
    \epsilon |u|^{p}|\nabla\zeta|^{p}\Phi(|u|^{2})
    +
    p^{p'}\epsilon^{-\frac1{p-1}}
    |\sfF|^{p}\,\Phi(|u|^{2})\,\zeta^{p}
   .
\end{equation*}
Combining the bounds for $\sfJ_{1}$, $\sfJ_{2}$, and $\sfJ_{3}$ and substituting into the definition of $\sfI\sfI_{1}$, we obtain, for every $\varepsilon>0$,
\begin{align*}
    |\sfI\sfI_1|
    &\le
    2\epsilon\iint_{Q_{R,S}}
    \big(\mu^2+|Du|^2\big)^\frac{p}2 \Phi\big( |u|^2\big)\zeta^p\chi\,\dx\dt\\
    &\phantom{\le\,}+
    \underbrace{\big(1+p^{p'}+ 2^{p'}C_\Phi^{p'}\big)}_{\le C(p)(1+C_\Phi^{p'})}\epsilon^{-\frac1{p-1}}
    \iint_{Q_{R,S}}|\sfF|^p\Phi\big( |u|^2\big)\zeta^p\chi\,\dx\dt\\
    &\phantom{\le\,}+
    \epsilon\iint_{Q_{R,S}}|u|^p|\nabla\zeta|^p     \Phi\big( |u|^2\big)\chi\, \dx\dt.
\end{align*}
Combining the estimate for $\sfI_2$
with the bound for $\sfI\sfI_1$
obtained above and the trivial bound for $\sfI\sfI_2$, we control the right–hand side of \eqref{initial_identity} by
\begin{align*}
    |\sfI_3| +|\sfI\sfI_1|+|\sfI\sfI_2|
    &\le
    3\epsilon \iint_{Q_{R,S}}
    \big(\mu^2+|Du|^2\big)^\frac{p}2 \Phi\big( |u|^2\big)\zeta^p\chi\,\dx\dt\\
    &\phantom{\le\,}
    +
    \Big[(C_{1}p)^{p'}\epsilon^{-\frac1{p-1}}+\epsilon\Big]
    \iint_{Q_{R,S}} |u|^{p}|\nabla\zeta|^{p}\Phi(|u|^{2})\chi\,\dx\dt\\
    &\phantom{\le\,}
    +C(p)\big(1+C_\Phi^{p'}\big)\epsilon^{-\frac1{p-1}}
    \iint_{Q_{R,S}}|\sfF|^p\Phi\big( |u|^2\big)\zeta^p\chi\,\dx\dt\\
     &\phantom{\le\,}
    +\iint_{Q_{R,S}}
    |\sff| |u|\Phi(|u|^2)\zeta^p\chi\,\dx\dt.
\end{align*}
Combining the lower bound for $\sfI_1+\sfI_2$
with the preceding estimate of $|\sfI_2|+ |\sfI\sfI_1|+|\sfI\sfI_2|$
and rearranging terms in \eqref{initial_identity}, we obtain
\begin{align*}
    -\tfrac12 \iint_{Q_{R,S}}&v\partial_t(\zeta^p \chi) \,\dx\dt
    +
    \big(C_o-3\epsilon\big)\iint_{Q_{R,S}}\underbrace{(\mu^2+|Du|^2)^{\frac p2}}_{\ge |Du|^p}\Phi(|u|^2)\zeta^p\chi\,\dx\dt\\
    &\le 
    C_o\iint_{Q_{R,S}}\mu^{p}
    \Phi(|u|^2)\zeta^p\chi\,\dx\dt\\
    &\phantom{\le\,}
    +
    \Big[(C_{1}p)^{p'}\epsilon^{-\frac1{p-1}}+\epsilon\Big]
    \iint_{Q_{R,S}} |u|^{p}|\nabla\zeta|^{p}\Phi(|u|^{2})\chi\,\dx\dt\\
    &\phantom{\le\,}
    +C(p)\big(1+C_\Phi^{p'}\big) \epsilon^{-\frac1{p-1}}
    \iint_{Q_{R,S}}|\sfF|^p\Phi\big( |u|^2\big)\zeta^p\chi\,\dx\dt\\
     &\phantom{\le\,}
    +\iint_{Q_{R,S}}
    |\sff| |u|\Phi(|u|^2)\zeta^p\chi\,\dx\dt.
\end{align*}
Fix $\epsilon:=\frac16 C_o$.
Since $C_o-3\epsilon =\frac12 C_o$, the previous inequality yields
\begin{align*}
    -\tfrac12 \iint_{Q_{R,S}}&v\zeta^p\partial_t \chi \,\dx\dt
    +
    \tfrac12 C_o\iint_{Q_{R,S}} |Du|^p \Phi(|u|^2)\zeta^p\chi\,\dx\dt\\
    &\le 
    C(p)C_o\Big(\frac{C_1}{C_o}\Big)^{p'}
    \iint_{Q_{R,S}} \big[|u|^{p}|\nabla\zeta|^{p}+\mu^{p}\z^p\big]\Phi(|u|^{2})\chi\,\dx\dt\\
    &\phantom{\le\,}
    +\frac{C(p)}{C_o^{\frac1{p-1}}}\big(1+C_\Phi^{p'}\big)
    \iint_{Q_{R,S}}|\sfF|^p\Phi\big( |u|^2\big)\zeta^p\chi\,\dx\dt\\
     &\phantom{\le\,}
    +\iint_{Q_{R,S}}
    |\sff||u|\Phi(|u|^2)\zeta^p\chi\,\dx\dt
    +
    \tfrac12 \iint_{Q_{R,S}}v|\partial_t\zeta^p|\chi \,\dx\dt
    \\
    &\le C\Bigg[
    \iint_{Q_{R,S}} \big[|u|^{p}|\nabla\zeta|^{p}+\mu^{p}\z^p\big]\Phi(|u|^{2})\chi\,\dx\dt\\
    &\qquad\qquad\qquad
    +
    \big(1+C_\Phi^{p'}\big)
    \iint_{Q_{R,S}}|\sfF|^p\Phi\big( |u|^2\big)\zeta^p\chi\,\dx\dt\Bigg]\\
     &\phantom{\le\,}
     +
    \iint_{Q_{R,S}}
    |\sff||u|\Phi(|u|^2)\zeta^p\chi\,\dx\dt
    %&\phantom{\le\,}
    +
    \tfrac12 \iint_{Q_{R,S}}v|\partial_t\zeta^p|\chi \,
    \dx\dt,
\end{align*}
for a constant $C=C(p,C_{o}/C_{1})$ (in particular, one may take $C\simeq C(p)C_{o}\,(C_{1}/C_{o})^{p'}$, where $p'=\tfrac{p}{p-1}$).  We have used $C_{1}\ge1$ to streamline the constants.
Passing to the limit $\delta\downarrow 0$ in the preceding inequality,  we have $\chi\downarrow \mathbf 1_{(-S,\tau]}$ and
$\partial_t\chi\to -\delta_{t=\tau}$. The first integral on the left-hand side
converges to $ \tfrac12\int_{B_{R}\times\{\tau\}}v\zeta^p\,\dx $,  while the second one can be estimated from below by $\iint_{B_R\times (-S,\tau]} |Du|^p \Phi(|u|^2)\zeta^{p}\,\dx\dt$. In the right-hand side integrals we bound the cut-off function $\chi$ by 1. After that, we take the supremum over $\tau\in (-S,0]$ in the first integral  and let $\tau\uparrow 0$ in the second one. In this way we get the desired energy estimate.
\end{proof}

A De Giorgi–type truncation estimate is needed to control super-level sets without resorting to weighted powers.
The next lemma provides %exactly 
this input: testing the system with the truncated vector $u(|u|-\sfk)_{+}^{p'}\,\zeta^{p}$ yields an intrinsic Caccioppoli inequality in which the gradient appears as $|\nabla|u||^{p}(|u|-\sfk)_{+}^{p'}$. The parameter $\sfk\ge0$ localizes the energy to the set $\{|u|>\sfk\}$, the cut–off $\zeta$ restricts to $Q_{R,S}(z_{o})$, and the nondegeneracy level $\mu$ enters only through lower–order terms. This estimate is the quantitative engine behind the De Giorgi iteration: it converts measure decay of super-level sets into pointwise bounds while accommodating the source terms $\sfF$ and $\sff$ under the structural hypotheses \eqref{growth-a*}. 
Note that it does not follow from Lemma~\ref{lem:energy-est}, since assumption \eqref{def:C_Phi} is not satisfied for the function $\Phi(s)=(s-\sfk)_{+}^{p'}$.

\begin{lemma}\label{en:DeGiorgi-version}
Let $1<p\le 2$  and $\mu\in [0,1]$. There exists a constant $C=C(p,C_o,C_1)$, such that whenever $u$ is a locally bounded weak solution in $E_T$ to \eqref{eq-par-gen}
with \eqref{growth-a*}
and $Q_{R,S}(z_o)\subset E_T$, then for any cut-off function $\zeta\in C^\infty\big(Q_{R,S}(z_o),[0,1]\big)$ that vanishes on the parabolic boundary of $Q_{R,S}(z_o)$ and any $\sfk\ge 0$ we have
\begin{align*}
    &\sup_{\tau\in (t_o-S,t_o]}\int_{B_R(x_o)\times\{\tau\}}
    v\zeta^p\,\dx
    +
    \iint_{Q_{R,S}(z_o)}
     |\nabla |u||^p \big(|u|-\sfk\big)_+^{p'}\zeta^{p}\, \dx\dt\\
    &\qquad\le
    C
    \iint_{Q_{R,S}(z_o)}
    \Big[|u|^p|\nabla\zeta|^p+\mu^p\zeta^p\Big]
    \big(|u|-\sfk\big)_+^{p'}
    \,\dx\dt  
    +
    C\iint_{Q_{R,S}(z_o)}v|\partial_t\zeta^p|
    \,\dx\dt\\
    &\qquad\phantom{\le\,}
    +
    C
    \iint_{Q_{R,S}(z_o)}|\sfF|^p
    |u|^{p'}
    \mathbf{1}_{\{|u|>\sfk\}}
    \zeta^p\,\dx\dt
    +\iint_{Q_{R,S}(z_o)}
   |\sff||u|^{p'+1}\mathbf{1}_{\{|u|>\sfk\}} \zeta^p  \,\dx\dt,
\end{align*}
where
\begin{equation*}
    v:=\int_0^{|u|}s(s-\sfk)_+^{p'}\,\ds.
\end{equation*}
\end{lemma}
\begin{proof}
By translation, we can always assume that $z_o=(0,0)$.
In the weak formulation of \eqref{eq-par-gen} we test with
\begin{equation*}
  \varphi(x,t):=u(x,t)\big(|u(x,t)|-\sfk\big)_{+}^{p'}\zeta(x,t)^{p}
  \chi(t),  
\end{equation*}
where $p'=\frac{p}{p-1}$ is the H\"older conjugate exponent. Here, $\chi$ is the time cut–off defined by $\chi\equiv1$ on $(-S,\tau)$, $\chi\equiv0$ on $(\tau+\delta,0)$, and linear on $[\tau,\tau+\delta]$.
Since $u$ is locally bounded, this choice is admissible; a standard time–mollification (or Steklov averaging) justifies the use of $\varphi$. For the term involving the time derivative in the weak formulation, we compute
\begin{align*}
   -\iint_{Q_{R,S}} u\cdot\partial_t\varphi\,\dx\dt
    &=
     \iint_{Q_{R,S}}\partial_t |u||u|\big(|u|-\sfk\big)_+^{p'}\zeta^p\chi\,\dx\dt\\
    &=
     \iint_{Q_{R,S}}\partial_t \bigg[ \int_0^{|u|}s(s-\sfk)_+^{p'}\,\ds\bigg]\zeta^p\chi\,\dx\dt\\
    &=
    -
     \iint_{Q_{R,S}}v \partial_t(\zeta^p \chi)
    \,\dx\dt.
\end{align*}
To estimate the contribution of the diffusion term in the weak formulation, we first compute the spatial derivative of the testing function:
\begin{align*}
    D_\alpha &\Big(u\big(|u|-\sfk\big)_+^{p'}\zeta^p\Big)\\
    &=
    D_\alpha u\big(|u|-\sfk\big)_+^{p'}\zeta^p
    +
    p'u\big(|u|-\sfk\big)_+^{p'-1}\frac{u\cdot D_\alpha u}{|u|}  \zeta^p
    +
    u\big(|u|-\sfk\big)_+^{p'}D_\alpha \zeta^p.
\end{align*}
Together with the expression for the time derivative term, we arrive at the identity
\begin{equation}\label{initial_identity-*}
    - \iint_{Q_{R,S}} v\partial_t (\zeta^p \chi) \, \dx \dt
    + \sfI_1+ \sfI_2
    = - \sfI_3 + \sfI\sfI_1 + \sfI\sfI_2,
\end{equation}
where we used the following abbreviations:
\begin{align*}
    \sfI_1
    &= \iint_{Q_{R,S}} \sum_{\alpha =1}^N
    \sfA_\alpha(x,t,u,Du) \cdot D_\alpha u 
    (|u|-\sfk)_+^{p'} \zeta^p \chi \, \dx\dt, \\
    \sfI_2
    &= p' \iint_{Q_{R,S}} 
    \sum_{\alpha =1}^N
    \sfA_\alpha(x,t,u,Du) \cdot u  \frac{D_\alpha u \cdot u}{|u|} 
    (|u|-\sfk)_+^{p'-1} \zeta^p \chi \, \dx\dt, \\
    \sfI_3
    &= \iint_{Q_{R,S}} \sum_{\alpha =1}^N
    \sfA_\alpha(x,t,u,Du) \cdot u D_\alpha \zeta^p 
    (|u|-\sfk)_+^{p'} \chi \, \dx\dt, \\
    \sfI\sfI_1
    &= \iint_{Q_{R,S}}
    |\sfF|^{p-2} \sum_{\alpha =1}^N \sfF_\alpha \cdot
    D_\alpha \left(u (|u|-\sfk)_+^{p'} \zeta^p \right) \chi \, \dx\dt,\\
    \sfI\sfI_2
    &=
    \iint_{Q_{R,S}}\sff\cdot u(|u|-\sfk)_+^{p'}\zeta^p\chi\,\dx\dt.
\end{align*}
Next, we consider the integral $\sfI_1$ on the left-hand side. 
Using the coercivity assumption \eqref{growth-a*}$_1$, we obtain
\begin{align*}
    \sfI_1
    &\ge 
    C_o \iint_{Q_{R,S}} \left( \mu^2 + |Du|^2 \right)^{\frac{p-2}{2}} |Du|^2 (|u| - \sfk)_+^{p'} \zeta^p \chi \, \dx\dt \\
    &= 
    C_o \iint_{Q_{R,S}} \left( \mu^2 + |Du|^2 \right)^{\frac{p}{2}} (|u| - \sfk)_+^{p'} \zeta^p \chi \, \dx\dt \\
    &\phantom{=\,} 
    - C_o \iint_{Q_{R,S}} 
    \underbrace{\left( \mu^2 + |Du|^2 \right)^{\frac{p-2}{2}}}_{\le \mu^{p-2}} 
    \mu^2 (|u| - \sfk)_+^{p'} \zeta^p \chi \, \dx\dt \\
    &\ge 
    C_o \iint_{Q_{R,S}} \left( \mu^2 + |Du|^2 \right)^{\frac{p}{2}} (|u| - \sfk)_+^{p'} \zeta^p \chi \, \dx\dt \\
    &\phantom{\le\,} 
    - C_o \iint_{Q_{R,S}} \mu^p (|u| - \sfk)_+^{p'} \zeta^p \chi \, \dx\dt.
\end{align*}
Since $\sfI_2 \ge 0$ due to \eqref{growth-a*}$_3$, among the $\sfI$-terms, it remains to estimate $\sfI_3$ from above. This can be done using the upper bound \eqref{growth-a*}$_2$ for $\sfA$. Indeed, we obtain
\begin{align*}
    |\sfI_3|
    &\le
    C_1p \iint_{Q_{R,S}} \left( \mu^2 + |Du|^2 \right)^{\frac{p-2}{2}} |Du|\,|u|\, |\nabla\zeta| \zeta^{p-1} (|u| - \sfk)_+^{p'} \chi \, \dx\dt \\
    &\le
     C_1 p\iint_{Q_{R,S}} \left( \mu^2 + |Du|^2 \right)^{\frac{p-1}{2}} |u|\, |\nabla\zeta| \zeta^{p-1} (|u| - \sfk)_+^{p'} \chi \, \dx\dt.
\end{align*}
We apply Young's inequality with exponents $p$, $p'$, and a parameter $\epsilon > 0$ to a part of the integrand on the right-hand side to obtain
\begin{align*}
    C_1p\zeta^{p-1} \left( \mu^2 + |Du|^2 \right)^{\frac{p-1}{2}} 
    |u|\, |\nabla \zeta| 
    &\le
    \epsilon\, \zeta^p \left( \mu^2 + |Du|^2 \right)^{\frac{p}{2}} 
    + C(p)C_1^p\epsilon^{1-p} |u|^p |\nabla \zeta|^p.
\end{align*}
Substituting this into the previous estimate, we obtain
\begin{align*}
        |\sfI_3|
        &\le
        \epsilon \iint_{Q_{R,S}}
         \big( \mu^2+|Du|^2\big)^\frac{p}{2} \big(|u|-\sfk\big)_+^{p'}\zeta^{p}\chi\,\dx\dt\\
        &\phantom{\le\, }
        +C(p)C_1^p\epsilon^{1-p}\iint_{Q_{R,S}}|u|^p\big(|u|-\sfk\big)_+^{p'}|\nabla\zeta|^p
        \chi\,\dx\dt.
\end{align*}
It remains to control the terms that arise due to the source terms $\sfF$ and $\sff$. We begin with the contribution from the divergence term. We write
\begin{align*}
    \sfI\sfI_1
    &:=
    \iint_{Q_{R,S}} \big[ \sfJ_1 + \sfJ_2 + \sfJ_3 \big] \chi \, \dx\dt,
\end{align*}
with the abbreviations
\begin{align*}
    \sfJ_1
    &:= \sum_{\alpha =1}^N |\sfF|^{p-2} \sfF_\alpha \cdot D_\alpha u \, (|u| - \sfk)_+^{p'} \zeta^p, \\
    \sfJ_2
    &:= p \sum_{\alpha =1}^N |\sfF|^{p-2} \sfF_\alpha \cdot u \, (|u| - \sfk)_+^{p'} 
    \zeta^{p-1} D_\alpha \zeta, \\
    \sfJ_3
    &:= p' \sum_{\alpha =1}^N |\sfF|^{p-2} \sfF_\alpha \cdot u \, 
    \frac{D_\alpha u \cdot u}{|u|} (|u| - \sfk)_+^{p'-1} \zeta^p.
\end{align*}
The terms $\sfJ_1$ and $\sfJ_2$ can easily be estimated.
In fact,
applying Young’s inequality to the term $\sfJ_1$ gives
\begin{align*}
    |\sfJ_1|
    &\le 
    |\sfF|^{p-1}\,|Du|(|u|-\sfk)_+^{p'} \zeta^p \\
    &\le 
    \epsilon |Du|^p (|u|-\sfk)_+^{p'} \zeta^p
    + 
    \epsilon^{-\frac{1}{p-1}} |\sfF|^p (|u|-\sfk)_+^{p'} \zeta^p \\
    &\le 
    \epsilon (\mu^2+|Du|^2)^{\tfrac{p}{2}} (|u|-\sfk)_+^{p'} \zeta^p
    + 
    \epsilon^{-\frac{1}{p-1}} |\sfF|^p (|u|-\sfk)_+^{p'} \zeta^p .
\end{align*}
A similar application of Young’s inequality to $\sfJ_2$ yields
\begin{align*}
    |\sfJ_2|
    &\le p |\sfF|^{p-1} \zeta^{p-1} |u|\,|\nabla\zeta|(|u|-\sfk)_+^{p'} \\
    &\le \epsilon |u|^p (|u|-\sfk)_+^{p'} |\nabla\zeta|^p
   + C(p) \epsilon^{-\frac{1}{p-1}}|\sfF|^p (|u|-\sfk)_+^{p'} \zeta^p .
\end{align*}
We now consider $\sfJ_3$. By Young’s inequality,
\begin{align*}
    |\sfJ_3|
    &\le 
    p' |\sfF|^{p-1} |u||Du|
    (|u|-\sfk)_+^{p'-1} \zeta^p \\
    &= 
    \Big(|Du| (|u|-\sfk)_+^{p'-1} \zeta \Big)
   \Big(p' |\sfF|^{p-1} |u| \mathbf{1}_{\{|u|>\sfk\}} \zeta^{p-1}\Big) \\
    &\le 
    \epsilon |Du|^p (|u|-\sfk)_+^{p'} \zeta^p
   + C(p)\epsilon^{-\frac{1}{p-1}} |\sfF|^p |u|^{p'} \mathbf{1}_{\{|u|>\sfk\}} \zeta^p \\
    &\le 
    \epsilon (\mu^2+|Du|^2)^{\tfrac{p}{2}} (|u|-\sfk)_+^{p'} \zeta^p
   + C(p)\epsilon^{-\frac1{p-1}} |\sfF|^p |u|^{p'} \mathbf{1}_{\{|u|>\sfk\}} \zeta^p .
\end{align*}
Using the three preceding estimates in $\sfI\sfI_1$, that is, multiplying by $\chi$ and integrating over $Q_{R,S}$, we obtain
\begin{align*}
    |\sfI\sfI_1|
    &\le
     2\epsilon 
     \iint_{Q_{R,S}}
     (\mu^2+|Du|^2)^{\tfrac{p}{2}} (|u|-\sfk)_+^{p'} \zeta^p\chi\,\dx\dt\\
    &\phantom{\le\,}
    +
    \epsilon \iint_{Q_{R,S}}|u|^p (|u|-\sfk)_+^{p'} |\nabla\zeta|^p\chi\,\dx\dt
  \\
   &\phantom{\le\,}+
    C(p)\epsilon^{-\frac1{p-1}} \iint_{Q_{R,S}}|\sfF|^p |u|^{p'} \mathbf{1}_{\{|u|>\sfk\}} \zeta^p
    \chi\,\dx\dt.
\end{align*}
It remains to estimate the term $|\sfI\sfI_2|$. We have
\begin{align*}
    |\sfI\sfI_2|
    &\le \iint_{Q_{R,S}}
   |\sff||u|(|u|-\sfk)_+^{p'} \zeta^p \chi \,\dx\dt \\
    &\le \iint_{Q_{R,S}}
   |\sff||u|^{p'+1}\mathbf{1}_{\{|u|>\sfk\}} \zeta^p \chi \,\dx\dt ,
\end{align*}
since $(|u|-\sfk)_+^{p'} \le |u|^{p'} \mathbf{1}_{\{|u|>\sfk\}}$.
Combining the bounds for for $\sfI_3$, $\sfI\sfI_1$, and  $\sfI\sfI_2$,
we obtain
\begin{align*}
    |\sfI_3|+|\sfI\sfI_1| +|\sfI\sfI_2|
    &\le
    3\epsilon \iint_{Q_{R,S}}
    \big( \mu^2+|Du|^2\big)^\frac{p}{2} \big(|u|-\sfk\big)_+^{p'}\zeta^{p}\chi\,\dx\dt
    \\
    &\phantom{\le\, }
    +\big(C(p)C_1^p\epsilon^{1-p}+\epsilon\big)\iint_{Q_{R,S}}|u|^p\big(|u|-\sfk\big)_+^{p'}|\nabla\zeta|^p\chi\,\dx\dt\\
    &\phantom{\le\,}
    +
    C(p)\epsilon^{-\frac1{p-1}} \iint_{Q_{R,S}}|\sfF|^p |u|^{p'} \mathbf{1}_{\{|u|>\sfk\}} \zeta^p
    \chi\,\dx\dt\\
    &\phantom{\le\,}
    +\iint_{Q_{R,S}}
   |\sff||u|^{p'+1}\mathbf{1}_{\{|u|>\sfk\}} \zeta^p \chi \,\dx\dt.
\end{align*}
We use this to estimate the right-hand side of \eqref{initial_identity-*}, together with the lower bound for $\sfI_1+\sfI_2$. In addition, we absorb the $3\epsilon$-term into the left-hand side. In the time term we also split the derivative of $\zeta^p\chi$. In this way we obtain
\begin{align*}%\label{C_o-3eps}\nonumber
    -\iint_{Q_{R,S}}& v\zeta^p\partial_t  \chi \, \dx \dt
    +
    \big(C_o-3\epsilon\big) \iint_{Q_{R,S}} \left( \mu^2 + |Du|^2 \right)^{\frac{p}{2}} (|u| - \sfk)_+^{p'} \zeta^p \chi \, \dx\dt 
   \\\nonumber
    &\le
    \big(C(p)C_1^p\epsilon^{1-p}+\epsilon\big)\iint_{Q_{R,S}}|u|^p\big(|u|-\sfk\big)_+^{p'}|\nabla\zeta|^p\chi\,\dx\dt\\\nonumber
    &\phantom{\le\,}+
   C_o \iint_{Q_{R,S}} \mu^p (|u| - \sfk)_+^{p'} \zeta^p \chi \, \dx\dt
   \\\nonumber
    &\phantom{\le\,}
    +
    C(p)\epsilon^{-\frac1{p-1}} \iint_{Q_{R,S}}|\sfF|^p |u|^{p'} \mathbf{1}_{\{|u|>\sfk\}} \zeta^p
    \chi\,\dx\dt\\
    &\phantom{\le\,}
    +\iint_{Q_{R,S}}
   |\sff||u|^{p'+1}\mathbf{1}_{\{|u|>\sfk\}} \zeta^p \chi \,\dx\dt
   +
   \iint_{Q_{R,S}} v\chi\partial_t\zeta^p  \, \dx \dt.
\end{align*}
At this stage we fix
$\epsilon:=\tfrac16C_o$.
Then $C_o-3\epsilon=\tfrac12 C_o$, and the previous inequality becomes 
\begin{align*}
    -\iint_{Q_{R,S}}& v\zeta^p\partial_t  \chi \, \dx \dt
    +
    \tfrac12 C_o
    \iint_{Q_{R,S}} \left( \mu^2 + |Du|^2 \right)^{\frac{p}{2}} (|u| - \sfk)_+^{p'} \zeta^p \chi \, \dx\dt 
   \\
    &\le
    C\iint_{Q_{R,S}}\big[|u|^p|\nabla\zeta|^p+\mu^p\zeta^p\big]\big(|u|-\sfk\big)_+^{p'}\chi\,\dx\dt\\
    &\phantom{\le\,}
    +
    C \iint_{Q_{R,S}}|\sfF|^p |u|^{p'} \mathbf{1}_{\{|u|>\sfk\}} \zeta^p
    \chi\,\dx\dt\\
    &\phantom{\le\,}
    +\iint_{Q_{R,S}}
   |\sff||u|^{p'+1}\mathbf{1}_{\{|u|>\sfk\}} \zeta^p \chi \,\dx\dt
   +
   \iint_{Q_{R,S}} v\chi\partial_t\zeta^p  \, \dx \dt
\end{align*}
for a constant $C=C(p,C_{o}/C_{1})$. In particular, one may take $C\simeq C(p)C_{o}\,(C_{1}/C_{o})^{p'}$. Here $p'=\tfrac{p}{p-1}\ge 2$, and we have used $C_1\ge 1$ to simplify the constants.  Arrived at this point, we let $\delta \downarrow 0$ in the preceding inequality. In the limit we have $\chi \downarrow \boldsymbol{\chi}_{(-S,\tau]}$. The first integral on the left-hand side converges to
$$
\tfrac12 \int_{B_R\times\{\tau\}} v \zeta^p \,\dx ,
$$
while the second integral is bounded below by
$$
\iint_{B_R\times(-S,\tau]} |\nabla |u||^p (|u|-\sfk)_+^{p'} \zeta^p \,\dx\dt .
$$
On the right-hand side we simply bound $\chi$ by $1$. In the obtained inequality we take the supremum over $\tau \in (-S,0)$ in the first term, and let $\tau \uparrow 0$ in the second. Hence we obtain the desired energy estimate.
\end{proof}

\subsection{Moser's iteration: from \texorpdfstring{$L^{\sfr}$}{Lr} to \texorpdfstring{$L^\infty$}{Linfinity}}

Building on the local energy  estimate from the previous
subsection, one can apply the Moser iteration scheme to upgrade an {\it a
priori} \(L^r\)-integrability assumption to an \(L^\infty\)-bound. The
mechanism is the familiar nonlinear iteration on shrinking cylinders,
whose effectiveness relies on the positivity of the parameter
\(\boldsymbol\lambda_r\). The following lemma provides the resulting
estimate for weak solutions to \eqref{eq-par-gen}.

\begin{lemma}\label{lm:sup-est-qual}
Let $p\in (1,\frac{2N}{N+2}]$, $q>\frac{N+p}{p}$, $\sfF\in L^{qp}_{\rm loc}\big(E_T,\R^{kN}
\big)$, $\sff\in L^{qp'}_{\rm loc}\big(E_T,\R^{k}\big)$, and
$\sfr>2$ such that
\begin{equation*}
    \boldsymbol\lambda_\sfr:=N(p-2)+p\sfr>0.
\end{equation*}
Then any local, weak solution $u$
to \eqref{eq-par-gen} in $E_T$ with \eqref{growth-a*} 
satisfying
$$
    u\in L^\sfr_{\rm loc}(E_T,\R^k),
$$
is locally bounded in $E_T$. 
\end{lemma}

\begin{remark}\upshape
It is possible to extract an explicit bound of $u$ out of Moser's iteration. See \cite[Proposition 4.1]{BDGL-25} in this connection. However, such form of bound is insufficient for our later application. 
\end{remark}

\begin{proof}
The proof proceeds in two steps. In the first step, towards establishing the local $L^\infty$-bound for $u$, we refine the energy inequality from Lemma~\ref{lem:energy-est} under the additional assumption that $|u|\in L_{\rm loc}^{2+2\alpha}(E_T)$ for some $\alpha \geq 0$. This is achieved by deriving a quantitative reverse H\"older inequality for $|u|$. In the second step, we apply an iteration scheme to this reverse H\"older inequality. Starting from the assumption $u\in L^\sfr_{\rm loc}(E_T)$, we gradually improve the integrability of $|u|$, while carefully controlling the constants that arise at each stage. Passing to the limit in this iteration ultimately yields the desired local $L^\infty$-bound. Throughout the proof, we always assume that $\rho\le1$.

{\it Step 1.}
Fix $\alpha \geq 0$ and a cylinder $Q_{\rho}(z_o)\equiv B_\rho(x_o)\times(t_o-\rho^p,t_o]\Subset E_T$. 
Assume in addition that
\begin{equation*}%\label{u-1+2alpha}
    |u|\in L^{2+2\alpha}_{\rm loc}(E_T).
\end{equation*}
By translation, we can always assume that $z_o=(0,0)$. In the energy inequality \eqref{en:Moser} we employ the testing function
\begin{equation*}%\label{def:phi-alpha-l-k}
    \Phi_{\alpha, \ell,k}(s)=
    \left\{
    \begin{array}{cl}
        k^{2\alpha}, &\mbox{if $0\le s\le k^2$,}\\[5pt]
        s^\alpha,&\mbox{if $k^2<s<\ell^2$,}\\[5pt]
        \ell^{2\alpha},&\mbox{if $s\ge\ell^{2}$,}
    \end{array}
    \right.
\end{equation*}
and consider the  auxiliary function 
\begin{align*}
    v
    &=\int_0^{|u|^{2}} \Phi_{\alpha,\ell,k }(s)\,\ds.
\end{align*}
The function $\Phi_{\al, \ell, k}$ is identical to $\phi_{\al, k,\ell}$ defined in \cite[Eqn. (4.2)]{BDGL-25} taking $m=1$ there.
By the same consideration as in \cite{BDGL-25}, one checks the following algebraic estimates of $v$, namely
\begin{align*}
     \tfrac{1}{1+\alpha}\Phi_{\alpha ,\ell,k}
    \big(|u|^2\big)^{\frac{1+\alpha}{\alpha}}-1 \le v \le |u|^{2 +2\alpha}+1.
\end{align*}
In addition, by the chain rule for Sobolev functions, also Kato’s inequality, we obtain
\begin{align*}
    &\Big|\nabla \big[\Phi_{\alpha,\ell,k}\big(|u|^2\big)\big]^\frac{p+2\alpha}{2p\alpha}\Big|^p \\
    &\qquad=
    \Big|\tfrac{p+2\alpha}{2p\alpha } \big[\Phi_{\alpha,\ell ,k}
    \big(|u|^2\big)\big]^{\frac{p+2\alpha}{2p\alpha }-1} \Phi_{\alpha,\ell,k}'\big(|u|^2\big)\nabla |u|^2\Big|^p\\
    &\qquad\le
     \big(\tfrac{p+2\alpha}{p\alpha } \big)^p
     \Phi_{\alpha,\ell,k }
     \big(|u|^2\big) |\nabla|u||^p 
  \Big[\big[\Phi_{\alpha,\ell,k}
  \big(|u|^2\big)\big]^{\frac{1}{2\alpha}-1} 
    \Phi_{\alpha,\ell,k }'\big(|u|^2\big)|u|\Big]^p \\
    &\qquad\le
     \big(\tfrac{p+2\alpha}{p\alpha } \big)^p
     \Phi_{\alpha,\ell,k }
     \big(|u|^2\big) |Du|^p 
  \Big[\big[\Phi_{\alpha,\ell,k}
  \big(|u|^2\big)\big]^{\frac{1}{2\alpha}-1} 
    \Phi_{\alpha,\ell,k }'\big(|u|^2\big)|u|\Big]^p.
\end{align*}
If $k<|u|<\ell$, then 
\begin{align*}
     \big[\Phi_{\alpha,\ell,k}\big(|u|^2\big)\big]^{\frac{1}{2\alpha}-1} 
     \Phi_{\alpha,\ell,k}'\big(|u|^2\big)|u| 
     &=
    \alpha
     |u|^{2\alpha(\frac{1}{2\alpha}-1)} 
     |u|^{2(\alpha -1)}|u|
     =
     \alpha.
\end{align*}
For $|u|> \ell$ or $|u|<k$ we have 
$\Phi_{\alpha,\ell,k }'\big(|u|^2)=0$. 
In either case it follows that
\begin{align*}
   \Big|\nabla \big[\Phi_{\alpha,\ell,k}\big(|u|^2\big)\big]^\frac{p+2\alpha}{2p\alpha}\Big|^p
     &\le
     (p+2\alpha)^p
    |Du|^p\Phi_{\alpha,\ell,k }\big(|u|^2\big).
\end{align*}
Finally, we note that
\begin{align*}
    \Phi_{\alpha,\ell,k}\big(|u|^2\big)
    +
    |u|^2 \Phi_{\alpha,\ell,k}^{\prime}\big(|u|^2\big)
    &\le
    (1+\alpha) \Phi_{\alpha,\ell,k}\big(|u|^2\big).
\end{align*}
Inserting the preparatory bounds into \eqref{en:Moser} yields
\begin{align*}
     \tfrac{1}{1+\alpha}
     \sup_{\tau\in (-\rho^p,0]}&\int_{B_\rho\times\{\tau\}}
    \zeta^p   \Phi_{\alpha ,\ell,k}
    \big(|u|^2\big)^{\frac{1+\alpha}{\alpha}}\,\dx\\
    &\phantom{\le\,} +
    \tfrac{C_o}{(p+2\alpha)^p}
    \iint_{Q_{\rho}}
    \zeta^{p}  \Big|\nabla \big[\Phi_{\alpha,\ell,k}\big(|u|^2\big)\big]^\frac{p+2\alpha}{2p\alpha}\Big|^p\,\dx\dt\\
    &\le
      C
    \iint_{Q_{\rho}}
    \big[|u|^p|\nabla\zeta|^p +\mu^p\zeta^p\big]\Phi_{\alpha ,\ell,k}
    \big(|u|^2\big) \,\dx\dt  \\
    &\phantom{\le\,}
    +C\iint_{Q_{\rho}} \big[  |u|^{2 +2\alpha}+1\big]|\partial_t\zeta^p|
    \,\dx\dt +
    \sup_{\tau\in (-\rho^p,0]}\int_{B_\rho\times\{\tau\}}
    \zeta^p \,\dx\\
        &\phantom{\le\,}
    +
     C\Big( 1+ \alpha^\frac{p}{p-1}\Big)
    \iint_{Q_{\rho}}|\sfF|^p\Phi_{\alpha ,\ell,k}
    \big(|u|^2\big)\zeta^p\,\dx\dt\\
    &\phantom{\le\,}
    +
     C\iint_{Q_{\rho}}|\sff||u|\Phi_{\alpha ,\ell,k}
    \big(|u|^2\big)\zeta^p\,\dx\dt ,
\end{align*}
where $C=C(p,C_o,C_1)$. 
On the right-hand side, the first term is estimated by using $\mu^p\zeta^p\le 1$ and
\[
    \big(1+|u|^{p } \big)\Phi_{\alpha,\ell,k}\big(
    |u|^{2}\big)
     \le
     2\big(|u|^{2+2\alpha}+1\big).
\]
Hence, together with the second term they are bounded by
\[
C
    (1+\|\nabla\z\|^p_{\infty}+\|\pl_t\z\|_{\infty})
    \iint_{Q_{\rho}}
    \big(|u|^{2+2\alpha}+1\big) \,\dx\dt ,
\]
whereas the third term is simply bounded by $1$ using $\rho\le 1$. The forth term is estimated by H\"older’s inequality with exponents $q$ and $\frac{q}{q-1}$:
\begin{align*}
    \iint_{Q_{\rho}} |\sfF|^p
    \Phi_{\alpha,\ell,k}\big(|u|^2\big)\zeta^p
    \,\dx\dt
    &\le
    \bigg[ \iint_{Q_{\rho}} | \sfF|^{qp}\,\dx\dt\bigg]^{\frac1q}\sfT
    ,
\end{align*}
where we defined
\begin{equation*}
    \sfT:= \bigg[\iint_{Q_{\rho}}
    \big( \Phi_{\alpha,\ell,k}\big(|u|^2\big)\zeta^p\big)^{\frac{q}{q-1}}\,\dx\dt\Big]^{\frac{q-1}q}.
\end{equation*}
The fifth term is again estimated by using Hölder’s inequality, now with exponents
$qp'$, $p$, and $\frac{qp'}{q-1}$, followed by Young’s inequality with exponents $p$ and $p'$. This gives
\begin{align*}
    &\iint_{Q_{\rho}} |\sff||u|
    \Phi_{\alpha,\ell,k}\big(|u|^2\big)
    \zeta^p\,\dx\dt\\
    &\quad
    =
    \iint_{Q_{\rho}} |\sff|\big[|u|
    \Phi_{\alpha,\ell,k}\big(|u|^2\big)^\frac1p\zeta\big] 
    \big[ \Phi_{\alpha,\ell,k}\big(|u|^2\big)\zeta^p\big]^\frac1{p'}\,\dx\dt\\
    &\quad
    \le
    \bigg[
    \iint_{Q_{\rho}}
    |u|^p\Phi_{\alpha,\ell,k}\big(|u|^2\big)\zeta^p\,\dx\dt
    \bigg]^\frac1p\\
    &\quad\qquad\qquad\cdot
    \bigg[\iint_{Q_{\rho}} |\sff|^{qp'}\,\dx\dt\bigg]^{\frac1{qp'}}
    \bigg[\iint_{Q_{\rho}}
    \big( \zeta^p\Phi_{\alpha,\ell,k}\big(|u|^2\big)\big)^{\frac{q}{q-1}}\,\dx\dt\bigg]^{\frac{q-1}{qp'}}
    \\
    &\quad\le
    \rho^{-p}\iint_{Q_{\rho}}
    |u|^p\Phi_{\alpha,\ell,k}\big(|u|^2\big)\zeta^p\,\dx\dt
    +
    \bigg[ \iint_{Q_{\rho}} |\sff|^{qp'}\,\dx\dt\bigg]^{\frac1{q}}
     \sfT.
\end{align*}
The integrals of $\sfF$ and $\sff$ can be extended to $Q_1$ without loss of generality. Their contributions will not be traced in what follows but simply absorbed into a lump constant $C$.

On the left-hand side of the energy estimate, we absorb the cut-off function $\zeta$
into the gradient in the second integral. For this we use the following estimate
\begin{align*}
    \Big|\nabla \big[\Phi_{\alpha,\ell,k}&\big(|u|^2\big)^\frac{p+2\alpha}{2p\alpha }\zeta\big]\Big|^p \\
    &\le
    2^{p-1} \Big|\nabla \Phi_{\alpha,\ell,k}\big(|u|^2\big)^\frac{p+2\alpha}{2p\alpha }\Big|^p \zeta^p
    +
    2^{p-1} \Phi_{\alpha,\ell,k}\big(|u|^2\big)^\frac{p+2\alpha}{2\alpha }|\nabla\zeta |^p\\
    &\le
    2^{p-1} \Big|\nabla \Phi_{\alpha,\ell,k}\big(|u|^2\big)^\frac{p+2\alpha}{2\alpha }\Big|^p \zeta^p
    +
    2^{p-1} \big( |u|^p+1\big)\Phi_{\alpha,\ell,k}\big(|u|^2\big)|\nabla\zeta |^p.
\end{align*}
Collecting all the preceding estimates, absorbing the cut-off function into the gradient and shifting the coefficients containing $\al$ to the right-hand side, we arrive at 
\begin{align*}
    \sup_{\tau\in [-\rho^p,0]}&\int_{B_\rho\times\{\tau\}} \zeta^p \,\Phi_{\alpha ,\ell,k}
    \big(|u|^2\big)^{\frac{1+\alpha}{\alpha}}\,\dx
    + 
    \iint_{Q_{\rho}}\Big|\nabla \big[\zeta\Phi_{\alpha,\ell,k}\big(|u|^2\big)^\frac{p+2\alpha}{2p\alpha }\big]\Big|^p  \,\dx\dt \\
    &\le 
    C\, (1+\al)^p(\rho^{-p}+\|\nabla\z\|^p_{\infty}+\|\pl_t\z\|_{\infty})
    \iint_{Q_{\rho}}\big[ |u|^{2+2\alpha}
    +1\big]\,\dx\dt\\
    &\phantom{\le\,}
    +C(1+\alpha)^{p'+p}\sfT,
\end{align*}
From now on $C$ also depends on $\|\sff\|_{qp',Q_1}$ and $\|\sfF\|_{qp,Q_1}$. The left-hand side will be denoted as $\sfL$ for short. In what follows we aim to show $\sfT\le \varep\sfL+\text{[other terms]}$ and absorb $\sfL$ to the left in the energy estimate.
To this end, we denote
\begin{equation}\label{def:m-q}
    \sfm=p\frac{2+2\alpha}{p+2\alpha},
    \qquad
    \sfq=p\frac{N+\sfm}{N},
\end{equation}
invoke the parabolic embedding (Proposition~4.1 in the Preliminaries of \cite{DBGV-book}) with
$$
    (v,p,m,q)=
     \Big( 
    \zeta
    \Phi_{\alpha,\ell,k}
    \big(|u|^2\big)^{\frac{p+2\alpha}{2p\alpha }},p,\sfm,\sfq\Big),
$$
and obtain that
\begin{align}\label{Sob:0}
    \iint_{Q_{\rho}}&
    \Big(
    \zeta
    \Phi_{\alpha,\ell,k}
    \big(|u|^2\big)^{\frac{p+2\alpha}{2p\alpha }}
    \Big)^{\sfq}\,\dx\dt
    \\\nonumber
    &\le
    C\iint_{Q_{\rho}}
    \Big|\nabla\Big[\zeta
    \Phi_{\alpha,\ell,k}
    \big(|u|^2\big)^{\frac{p+2\alpha}{2p\alpha }}\Big]\Big|^p\,\dx\dt\\\nonumber
    &\qquad\qquad 
    \cdot\bigg[\sup_{\tau\in(-\rho^p,0]} 
    \int_{B_\rho\times\{\tau\}} 
    \Big[\zeta
    \Phi_{\alpha,\ell,k}\big(|u|^2\big)^{\frac{p+2\alpha}{2p\alpha }}
    \Big]^{p\frac{2+2\alpha}{p+2\alpha}}\,\dx\bigg]^{\frac{p}N}\\ \nonumber
        &\le
    C\iint_{Q_{\rho}}
    \Big|\nabla\Big[\zeta
    \Phi_{\alpha,\ell,k}
    \big(|u|^2\big)^{\frac{p+2\alpha}{2p\alpha }}\Big]\Big|^p\,\dx\dt\\\nonumber
    &\qquad\qquad 
    \cdot\bigg[\sup_{\tau\in( -\rho^p,0]} 
    \int_{B_\rho\times\{\tau\}} 
    \zeta^p
    \Phi_{\alpha,\ell,k}
    \big(|u|^2\big)^{\frac{1+\alpha}{\alpha }}
    \,\dx\bigg]^{\frac{p}N}\\ \nonumber
     &=C
    \underbrace{
    \iint_{Q_{\rho}}
    \Big|\nabla\Big[\zeta
    \Phi_{\alpha,\ell,k}\big(|u|^2\big)^{\frac{p+2\alpha}{2p\alpha }}\Big]\Big|^p\,\dx\dt}_{\le\sfL}\\ \nonumber
    &\qquad\qquad 
    \cdot\bigg[
    \underbrace{
    \sup_{\tau\in (-\rho^p,0]} 
    \int_{B_\rho\times\{\tau\}} 
    \zeta^p
    \Phi_{\alpha,\ell,k}
    \big(|u|^2\big)^{\frac{1+\alpha}{\alpha }}
    \,\dx}_{\le\sfL}\bigg]^{\frac{p}N}\\ \nonumber
    &\le 
    C \sfL^{\frac{N+p}{N}} .
\end{align}
Noticing that $\frac{q}{q-1}<\frac{N+p}{N}<\sfq\frac{p+2\alpha}{2p\alpha}$, we estimate $\sfT$ by first applying
the interpolation inequality of the form $\|\cdot\|_{\frac{q}{q-1}}\le \|\cdot\|^{\theta}_{\frac{N+p}{N}}  \|\cdot\|^{1-\theta}_{1}$ 
with $\theta=\frac{N+p}{pq}$. After that we apply Young's inequality to obtain for every $\varep>0$ that
\begin{align}\label{T-L}
\sfT&\le \varep  \bigg[\iint_{Q_{\rho}}
    \big( \Phi_{\alpha,\ell,k}\big(|u|^2\big)\zeta^p\big)^{\frac{N+p}{N}}\,\dx\dt\Big]^{\frac{N}{N+p}}\\\nonumber
    &\qquad+\varep^{-\frac{N+p}{qp-(N+p)}}\iint_{Q_{\rho,\theta}} \zeta^p\Phi_{\alpha,\ell,k}\big(|u|^2\big)\,\dx\dt.
\end{align}
Using again that $\frac{q}{q-1}<\frac{N+p}{N}<\sfq\frac{p+2\alpha}{2p\alpha}$, we first apply H\"older's inequality to estimate the first integral in \eqref{T-L} (noting again $\rho\le1$) and then apply \eqref{Sob:0} to get
\begin{align*}
\bigg[\iint_{Q_{\rho}}
    &\big( \Phi_{\alpha,\ell,k}\big(|u|^2\big)\zeta^p\big)^{\frac{N+p}{N}}\,\dx\dt\bigg]^{\frac{N}{N+p}}\\
    &\le \bigg[\iint_{Q_{\rho}}
    \big(
    \zeta^p
    \Phi_{\alpha,\ell,k}\big(|u|^2\big)
    \big)^{\sfq\frac{p+2\alpha}{2p\alpha }}\,\dx\dt\bigg]^{\frac{2p\alpha }{\sfq(p+2\alpha)}}\\ \nonumber
    &=
    \bigg[\iint_{Q_{\rho}}
    \Big(
    \zeta^{\frac{p+2\alpha}{2\alpha }}
    \Phi_{\alpha,\ell,k}\big(|u|^2
    \big)^{\frac{p+2\alpha}{2p\alpha }}
    \Big)^{\sfq}\,\dx\dt\bigg]^{\frac{2p\alpha }{\sfq(p+2\alpha)}}\\ \nonumber
    &\le
    \bigg[ \iint_{Q_{\rho}}\Big(
    \zeta
    \Phi_{\alpha,\ell,k}
    \big(|u|^2\big)^{\frac{p+2\alpha}{2p\alpha }}
    \Big)^{\sfq}\,\dx\dt \bigg]^{\frac{2p\alpha }{\sfq(p+2\alpha)}}\\\nonumber
    &\le C \sfL^{\frac{N+p}{N}\frac{2p\alpha }{\sfq(p+2\alpha)}}.
\end{align*}
With the above estimate at hand, we plug \eqref{T-L} into the energy estimate to get
\begin{align*}
    \sfL
    &\le 
    C(1+\al)^p(\rho^{-p}+\|\nabla\z\|^p_{\infty}+\|\pl_t\z\|_{\infty})
    \iint_{Q_{\rho}}\big[ |u|^{2+2\alpha}
    +1\big]\,\dx\dt\\
    &\phantom{\le\,}
    +C(1+\alpha)^{p'+p}\varep\sfL^{\frac{N+p}{N}\frac{2p\alpha }{\sfq(p+2\alpha)}}\\
    &\phantom{\le\,}
    +C\varep^{-\frac{N+p}{qp-(N+p)}}(1+\alpha)^{p'+p}\iint_{Q_{\rho}} \zeta^p\Phi_{\alpha,\ell,k}\big(|u|^2\big)\,\dx\dt.
\end{align*}
In the above estimate, we further use the fact that $\sfL^{\frac{N+p}{N}\frac{2p\alpha }{\sfq(p+2\alpha)}}\le \sfL +1$ since the power is in $(0,1)$ and then choose $\varep$ to satisfy $C(1+\alpha)^{p'+p}\varep=\frac12$ to absorb $\sfL$ to the left-hand side. 
Moreover, using this choice of $\varep$ and the fact that $\Phi_{\alpha,\ell,k}\big(|u|^2\big)
\le |u|^{2 +2\alpha}+1$, the previous inequality eventually gives that
\begin{align*}
    \sfL
    &\le 
    C(1+\al)^{\frac{p'+p}{1-\frac{N+p}{qp}}}(\rho^{-p}+\|\nabla\z\|^p_{\infty}+\|\pl_t\z\|_{\infty})
    \iint_{Q_{\rho}}\big[ |u|^{2+2\alpha}
    +1\big]\,\dx\dt.
\end{align*}
Plug this into \eqref{Sob:0} to get
\begin{align*}
\bigg[\iint_{Q_{\rho}}&
    \Big(
    \zeta
    \Phi_{\alpha,\ell,k}
    \big(|u|^2\big)^{\frac{p+2\alpha}{2p\alpha }}
    \Big)^{\sfq}\,\dx\dt\bigg]^{\frac{N}{N+p}}\\
    &\le     C(1+\al)^{\frac{p'+p}{1-\frac{N+p}{qp}}}(\rho^{-p}+\|\nabla\z\|^p_{\infty}+\|\pl_t\z\|_{\infty})
    \iint_{Q_{\rho}}\big[ |u|^{2+2\alpha}
    +1\big]\,\dx\dt.
\end{align*}
Note that the right-hand side is finite and independent of $k$ and $\ell$. Moreover,
$$
    \lim_{\ell\to\infty}
    \lim_{k\downarrow 0}
    \Phi_{\alpha,\ell,k}
    \big(|u|^2\big)^{\sfq\,\frac{p+2\alpha}{2p\alpha}}
    = |
    u|^{\frac{\sfq}{p}(p+2\alpha)} .
$$
Since $0\le \zeta\le 1$ and $\phi_{\alpha,\ell,k}(|u|^2)\uparrow |u|^{2\alpha}$ as $k\downarrow 0$, $\ell\to\infty$, we may pass to the limit in the previous estimate by monotone convergence to obtain
\begin{align}\label{Moser-energy:0-new}
\nonumber
    \bigg(\iint_{Q_{\rho}}&
    \zeta^{\sfq}\,|u|^{\frac{\sfq}{p}(p+2\alpha)}\,\dx\dt\bigg)^{\!\frac{N}{N+p}}\\
    &\le
    C(1+\alpha)^{\frac{p'+p}{\,1-\frac{N+p}{qp}\,}}(\rho^{-p}+\|\nabla\z\|^p_{\infty}+\|\pl_t\z\|_{\infty})
    \iint_{Q_{\rho}}
    \big(|u|^{2+2\alpha}+1\big)\,\dx\dt,
\end{align}
where $C=C(N,p,C_o,C_1)$, and $\sfq$ is defined in \eqref{def:m-q}.

By a standard covering argument, the estimate on 
$Q_{\frac12\rho}$ propagates to compactly contained cylinders. Consequently,
\begin{equation*}
     |u|\in L^{2+2\alpha}_{\rm loc} (E_T)
     \quad\Longrightarrow\quad
     |u|\in L^{(p+2\alpha)\frac{\bq}{p}}_{\rm loc} (E_T).
\end{equation*}
In order to have 
\[
(p+2\alpha)\frac{\bq}{p}>2+2\alpha,
\]
we need
\[
N(p-2)+(2+2\alpha)p>0,
\]
and this is guaranteed by the existence of $\sfr>2$ such that $\boldsymbol\lm_{\sfr}>0$; indeed, it is enough to let $\sfr\equiv2+2\alpha$.

{\em Step 2. Setup of the iteration.}
Fix $Q_o:=Q_{\rho}\subset E_T$ and define a shrinking family of  cylinders
$ Q_i:=Q_{\rho_i}=B_{\rho_i}\times (-\rho_i^{\,p},0]$, $i\in\mathbb N_0$,
and radii
$$
    \rho_i=\frac{\rho}2 +\frac{\rho}{2^{i+1}},\quad \mbox{for $i\in\N_0$.}
$$
Then $(\rho_i)_{i\in\N_0}$
is decreasing, $\rho_o=\rho$, and $\rho_\infty
=\frac12 \rho $. For each $i$ choose $\zeta=\zeta_i\in W^{1,\infty}(Q_i,[0,1])$ such that $\zeta\equiv 1$ on $Q_{i+1}$, $\zeta=0$ on $\partial_{\rm par} Q_i$, and
\begin{equation*}
    |\nabla\zeta|\le\frac{2^{i+2}}{\rho},
    \quad\mbox{and}\quad
    |\partial_t\zeta|\le\frac{2^{i+4}}{\rho^p}.
\end{equation*}
Apply the above choice of $\zeta=\zeta_i$ in \eqref{Moser-energy:0-new} and take integral averages
yield that
\begin{align}\label{Iter-start-new}\nonumber
    \bigg[ & \biint_{Q_{i+1}} |u|^{(p+2\alpha)\frac{\sfq}{p}}\,\dx\dt\bigg]^{\frac N{N+p}}\\\nonumber
    &\le 
    C\,(1+\alpha)^{\frac{p'+p}{1-\frac{N+p}{qp}}}
    \bigg[
    1+\rho^p\Big(\frac{2^{p(i+2)}}{\rho^p}
    +
    \frac{2^{i+4}}{\rho^p}\Big)\bigg] \biint_{Q_{i}}\big(|u|^{2+2 \alpha}+1\big)\,\dx\dt\\
    &\le
    C\,2^{pi} (1+\alpha)^{\frac{p'+p}{1-\frac{N+p}{qp}}}\biint_{Q_{i}}
    \big(|u|^{2 +2 \alpha}+1\big)\,\dx\dt.
\end{align}
Let $\kappa:=1+\tfrac{p}{N}$. A direct calculation shows
$$
    (p+2\alpha)\,\frac{\sfq}{p}
    =
    2\alpha\,\kappa \;+\; p\,\frac{N+2}{N}.
$$
Therefore, it is natural to choose $\alpha_o$ by $2\alpha_o+2=\sfr$, and to define $(\alpha_i)_{i\in\mathbb N_0}$ recursively by
$$
    2\alpha_{i+1}+2
    =
    2\kappa\,\alpha_i \;+\; p\,\frac{N+2}{N}.
$$
Equivalently, with $s_i:=2+2\alpha_i$ one has the linear recurrence
$$
    s_{i+1}
    =\kappa s_i+(p-2),\qquad s_o=\sfr.
$$
Hence,
\begin{align*}
    s_i
    &=
    \kappa^i \sfr+\frac{p-2}{\kappa-1}(\kappa^i-1)\\
    &=
    \kappa^i\Big(\sfr+\frac{N}{p}(p-2)\Big)-
    \frac{N}{p}(p-2)\\
    &=
    \frac1p 
    \boldsymbol{\lambda}_\sfr
    \kappa^i
    -\frac1p \boldsymbol{\lambda}_2+2,
\end{align*}
so that $s_i\to\infty$ as $i\to\infty$. In particular, the iteration \eqref{Iter-start-new} starts with the $L^\sfr$-integral of $|u|$ on $Q_o$ on the right-hand side, which is finite by assumption, and produces a strictly improving sequence of integrability exponents. 
Inserting the recursion 
$
    2\alpha_{i+1}+2
    =
    2\kappa\,\alpha_i \;+\; p\,\frac{N+2}{N}
$
into \eqref{Iter-start-new} yields
\begin{align*}
    \bigg[
    \biint_{Q_{i+1}}  |u|^{2+2\alpha_{i+1}}\,\dx\dt
    \bigg]^{\frac N{N+p}}\le 
    C\,2^{pi} (1+\alpha_i)^{\frac{p'+p}{1-\frac{N+p}{qp}}}\biint_{Q_{i}}
    \big(|u|^{2 +2 \alpha_i}+1\big)\,\dx\dt.
\end{align*}
In order to control $1+\alpha_i$ in terms of $\kappa^i$, note that with $\boldsymbol\lambda_\sfr=N(p-2)+p\sfr$ and $\boldsymbol\lambda_2=N(p-2)+2p$ we have
\begin{align*}
    s_i=2+2\alpha_i
    &=2+\frac{1}{p}\boldsymbol\lambda_\sfr\kappa^i-\frac{1}{p}\boldsymbol\lambda_2\\
    &=\frac{1}{p}\,N(2-p)+\frac{1}{p}\boldsymbol\lambda_\sfr\,\kappa^i\\
    &\le 
    \frac{1}{p}\,[\,N(2-p)+\boldsymbol\lambda_\sfr\,]\,\kappa^i
=\sfr\,\kappa^i.
\end{align*}
Consequently,
$$
1+\alpha_i \;\le\; \frac{\sfr}{2}\,\kappa^i.
$$
Substituting this inequality 
into the preceding estimate, we obtain the {\it recursive inequality}
\begin{align}\label{rec-Hoelder-new}
    \bigg[
    \biint_{Q_{i+1}}  |u|^{2+2\alpha_{i+1}}\,\dx\dt
    \bigg]^{\frac N{N+p}}\le 
    C\,\Big(2^{p} \kappa^{\frac{p'+p}{1-\frac{N+p}{qp}}}\Big)^i\biint_{Q_{i}}
    \big(|u|^{2 +2 \alpha_i}+1\big)\,\dx\dt,
\end{align}
where now $C$ takes into account also the term $\left(\frac\sfr2\right)^{{\frac{p'+p}{1-\frac{N+p}{qp}}}}$.
In order to complete the iteration, we add $1$ on the left side integral and introduce 
\begin{align*}
    \sfY_i=\bigg[\biint_{Q_i}\big(|u|^{s_i}+1\big)\,\dx\dt\bigg]^{\frac1{s_i}},\qquad \sfb=2^{p} \kappa^{\frac{p'p}{1-\frac{N+p}{qp}}}.
\end{align*}
In this notation, \eqref{rec-Hoelder-new} takes the form
$$
    \sfY_{i+1}^{s_{i+1}} 
    \le
    \big(
    C \sfb^i \sfY_i^{s_i}
    \big)^{\kappa}, \quad \forall \, i \in \mathbb{N}_0,
$$
and the iteration yields
\begin{equation*}
    \sfY_i
    \le
    \prod_{j=1}^i C^{\frac{\kappa^{i-j+1}}{s_i}}
    \prod_{j=1}^i \sfb^{j\frac{\kappa^{i-j+1}}{s_i}} 
    \sfY_o^{\frac{s_o \kappa^i}{s_i}},\qquad\forall\,i\in\mathbb{N}_0.
\end{equation*}
Unwinding the products gives
$$
    \sfY_i
    \le
    C^{\frac{1}{s_i}\sum_{j=1}^{i}\kappa^{\,i-j+1}}
    \sfb^{\frac{1}{s_i}\sum_{j=1}^{i} j\kappa^{i-j+1}}\;
    \sfY_o^{\frac{s_o\kappa^{i}}{s_i}}.
$$
Recall that $s_o=2\alpha_o+2=\sfr$. To compute and bound the exponents, note that
\begin{align*}
    \frac{\kappa^{i-j+1}}{s_i}
    &=
    \frac{\kappa^{i-j+1}}{2 \alpha_i+2}
    =
    \frac{\kappa^{i-j+1}}{\frac{\boldsymbol \lambda_\sfr}{p} \kappa^i-\frac{\boldsymbol \lambda_2}{p}+2}\\
    & =
    \frac{\kappa^{i-j+1}}{\frac{\boldsymbol \lambda_\sfr}{p} \kappa^i+\frac{N(2-p)}{p}} 
    \le \frac{p}{\boldsymbol \lambda_\sfr} \kappa^{-j+1},
\end{align*}
and likewise
\begin{align*}
    \frac{s_o \kappa^i}{s_i}
    &=
    \frac{\sfr}{\frac{\boldsymbol \lambda_\sfr}{p} \kappa^i+\frac{N(2-p)}{p}} \kappa^i 
    \le 
    \frac{\sfr}{\frac{\boldsymbol \lambda_\sfr}{p}} 
    =
    \frac{p \sfr}{\boldsymbol \lambda_\sfr}.
\end{align*}
Here we used $\boldsymbol{\lambda}_\sfr= N(p-2)+\sfr p$ and $\boldsymbol{\lambda}_2= N(p-2)+2p$, so that $2p-\boldsymbol{\lambda}_2=N(2-p)$.  In particular, we obtain the bound
\begin{align*}
     \sum_{j=1}^i \frac{\kappa^{i-j+1}}{s_i} 
     &\le 
     \frac{p}{\boldsymbol\lambda_\sfr} 
     \sum_{j=1}^i \kappa^{-j+1} 
     \le 
     \frac{p}{\boldsymbol\lambda_\sfr} \frac{1}{1-1 /\kappa}
     =
     \frac{N+p}{\boldsymbol\lambda_\sfr},
\end{align*}
and, similarly,
\begin{align*}
    \sum_{j=1}^i j \frac{\kappa^{i-j+1}}{s_i}
    &\le 
    \frac{p}{\boldsymbol \lambda_\sfr} \sum_{j=1}^i j \kappa^{-j+1} 
    \le 
    \frac{p}{\boldsymbol \lambda_\sfr}\Big(\frac{\kappa}{\kappa-1}\Big)^2
    =
    \frac{p}{\boldsymbol \lambda_\sfr}\Big(\frac{N+p}{p}\Big)^2.
\end{align*}
Using these estimates to control the exponents in the product representation of $\sfY_i$ yields
\begin{equation*}
  \limsup_{i \rightarrow \infty} \sfY_i 
  \le 
  C^{\frac{N+p}{\boldsymbol\lambda_\sfr} } 
  \sfb^{\frac{p}{\boldsymbol\lambda_\sfr}(\frac{N+p}{p})^2} 
    \sfY_o^{\frac{p\sfr}{\boldsymbol \lambda_\sfr}}.  
\end{equation*}
Moreover, straightforward computations give that
\begin{align*}
   \limsup_{i\rightarrow\infty}\sfY_i\ge \sup_{Q_{\frac12\rho}}|u|, \qquad \sfY_o^{\frac{p\sfr}{\boldsymbol\lambda_\sfr}}
    &=
    \bigg[
    \biint_{Q_\rho}\big(|u|^\sfr+1\big)\,\dx\dt
    \bigg]^{\frac{p}{\boldsymbol\lambda_\sfr}}.
\end{align*}
Therefore, we have
\begin{equation*}
    \sup_{Q_{\frac12\rho}}|u|
    \le
    \Bigg[
    C 
    \biint_{Q_{\rho}}\big(|u|^\sfr+1\big)
    \,\dx\dt
    \Bigg]^{\frac{p}{\boldsymbol\lambda_\sfr}},
\end{equation*}
where $C$ depends on $N,p,C_o,C_1,q,\sfr$ and also on $\|\sff\|_{qp',Q_1}$ and $\|\sfF\|_{qp,Q_1}$. As a result of a standard covering argument, we conclude that $u$ is locally bounded in $E_T$.
\end{proof}

\begin{remark}
\upshape It is straightforward to verify that if $\sfr=\frac{N(2-p)}{p}$, that is, if $\boldsymbol\lambda_\sfr=0$, then $\forall i\in{\mathbb N}_0$ we have $s_i=\sfr$, and the iteration process described above cannot converge. This does not exclude the possibility that $u$ is bounded; however, such a conclusion cannot be obtained by our method, and therefore remains an open problem.
\end{remark}

\begin{remark}[Qualitative $L^\infty$-estimate for $\frac{2N}{N+2}<p\le 2$]\upshape
In this range we assume
$$
u\in C^0\big([0,T];L^2(E,\R^k)\big)\cap L^p\big(0,T;W^{1,p}(E,\R^k)\big).
$$
By the parabolic Sobolev embedding,
$$
u\in L^\sfr_{\rm loc}(E_T,\R^k)\quad\text{with}\quad \sfr=p\,\frac{N+2}{N}>2.
$$
In particular, $u\in L^2_{\rm loc}(E_T,\R^k)$. Moreover,
$$
\boldsymbol\lambda_2:=N(p-2)+2p>0\qquad\text{whenever}
\quad p>\tfrac{2N}{N+2}.
$$
Hence the structural condition $\boldsymbol\lambda_\sfr>0$ is satisfied, \emph{taking $\sfr=2$}. Consequently, the argument developed for the sub-critical case $1<p\le \tfrac{2N}{N+2}$ applies verbatim (with this choice of $\sfr$) and yields the local boundedness of $u$ without any extra assumption.
\end{remark}

\subsection{De Giorgi's iteration: proof of Theorem~\ref{thm:sup-quant}}\label{sec:degiorgi}
Throughout the proof the center point 
$z_o$ of all cylinders is fixed, without loss of generality we may assume it to coincide with the origin $(0,0)$, and we suppress it in the notation. Since by Lemma~\ref{lm:sup-est-qual} weak solutions are locally bounded, the energy estimate in Lemma \ref{en:DeGiorgi-version} is at our disposal.
In fact, Lemma \ref{en:DeGiorgi-version} gives
\begin{align*}
    &\sup_{t\in (-\theta,0]}\int_{B_\rho\times\{t\}}
    v\zeta^p\,\dx
    +
    \iint_{Q_{\rho,\theta}}
     \big|\nabla \big[ \big(|u|-\sfk\big)_+^{p'}\zeta\big]\big|^{p}\, \dx\dt\\
    &\quad\le
    C
    \iint_{Q_{\rho,\theta}}
    \Big[|u|^p|\nabla\zeta|^p +\mu^p\zeta^p\Big]
    \big(|u|-\sfk\big)_+^{p'}
    \dx\dt 
    +   
    C\iint_{Q_{\rho,\theta}}
    v|\partial_t\zeta^p|
    \dx\dt\\
    &\quad\phantom{\le\,}
    +
     C
    \iint_{Q_{\rho,\theta}}|\sfF|^p
    |u|^{p'}
    \mathbf{1}_{\{|u|>\sfk\}}
    \zeta^p\,\dx\dt
    +
    \iint_{Q_{\rho,\theta}}|\sff|
    |u|^{p'+1}
    \mathbf{1}_{\{|u|>\sfk\}}
    \zeta^p\,\dx\dt,
\end{align*}
where $C=C(p,C_o,C_1)$, $p'=\frac{p}{p-1}$, and
$$
    \tfrac1{2+p'} \big(|u|-\sfk\big)_+^{p'+2}
    \le
    v=\int_0^{|u|}s(s-\sfk)_+^{p'}
    \,\ds
    \le
    \tfrac1{2+p'}|u|^{p'+2}\mathbf{1}_{\{ |u|>\sfk\}}.
$$
This energy estimate holds  for every cylinder $Q_{\rho,\theta}\Subset E_T$, for every level $\sfk\ge 0$, and for every cut-off  $\zeta\in W^{1,\infty}(Q_{\rho,\theta}, [0,1])$ vanishing on the parabolic boundary $\partial_{\rm par}Q_{\rho,\theta}$. Fix a cylinder $Q_o:=Q_{\rho,\theta}\Subset E_T$ and choose $0<\tau<1$. Define a nested family of cylinders
$Q_n:=Q_{\rho_n,\theta_n}$ with
\begin{equation*}
    \rho_n:=\tau\rho
    +\frac{1-\tau}{2^{\,n}}\rho,
    \qquad
    \theta_n:=\tau\theta
    +\frac{1-\tau}{2^{\,n}}\theta,\qquad n\in\mathbb N_0,
\end{equation*}
so that $Q_{n+1}\Subset Q_n$, $\rho_n\downarrow\tau\rho$, and $\theta_n\downarrow\tau\theta$.
For each $n$ pick $\zeta_n\in W^{1,\infty}(Q_n,[0,1])$ with
$\zeta_n\equiv1$ on $Q_{n+1}$,
$\zeta_n=0$ on $\partial_{\rm par}Q_n$,  
\begin{equation*}
    |\nabla\zeta_n|\le \frac{2^{n+2}}{(1-\tau)\rho},
    \quad\mbox{and}\quad
    |\partial_t\zeta_n|
    \le\frac{2^{n+2}}{(1-\tau)\theta}.
\end{equation*}
The truncation levels are taken increasing to $\sfk>0$ by
\begin{equation*}
      \sfk_n:=\sfk-\frac{\sfk}{2^n},
\end{equation*}
so that $\sfk_{n+1}-\sfk_n=\tfrac{\sfk}{2^{n+1}}$ and $\sfk_n\uparrow \sfk$. With these preparations, we insert the pairs $(\zeta_n,\sfk_{n+1})$ in the above energy inequality on the cylinders $Q_n$ and obtain
\begin{align*}
    \sup_{t\in (-\theta_n,0]}\int_{B_n\times\{t\}}&
    \big(|u|-\sfk_{n+1}\big)^{p'+2}\zeta_n^p\,\dx
    +
    \iint_{Q_{n}}
     \big|\nabla \big[ \big(|u|-\sfk_{n+1}\big)_+^{p'}\zeta_n\big]\big|^{p}\, \dx\dt\\
    % &\le
    % C\bigg[ \frac{2^{(n+1)p}}{(1-\tau)^p\varrho^p}\sfI_1
    %  +
    % \frac{2^{n+1}}{(1-\tau)\theta}\sfI_2
    % \bigg]+ C\big[\sfI_3+\sfI_4+\sfI_5
    % \big]\\
    &\le
    C2^{(n+1)p} \bigg[ \frac{\sfI_1}{(1-\tau)^p\varrho^p}
     +
    \frac{\sfI_2}{(1-\tau)\theta}\bigg]
    +
    C\big[\sfI_3+\sfI_4+\sfI_5
    \big],
\end{align*}
where we abbreviated
\begin{align*}
    \sfI_1
    &:=
    \iint_{Q_{n}}
    |u|^{p'+p}\mathbf{1}_{\{|u|>\sfk_{n+1}\}}
    \dx\dt,\\
    \sfI_2
    &:=
    \iint_{Q_{n}}|u|^{p'+2}
    \mathbf{1}_{\{|u|>\sfk_{n+1}\}}
    \dx\dt,\\
     \sfI_3
    &:=
     \mu^p\iint_{Q_{n}}
    |u|^{p'}
    \mathbf{1}_{\{|u|>\sfk_{n+1}\}}
    \,\dx\dt,\\
    \sfI_4
    &:=
     \iint_{Q_{n}}|\sfF|^p
    |u|^{p'}
    \mathbf{1}_{\{|u|>\sfk_{n+1}\}}
    \,\dx\dt,\\
    \sfI_5
    &:=
     \iint_{Q_{n}}|\sff|
    |u|^{p'+1}
    \mathbf{1}_{\{|u|>\sfk_{n+1}\}}
    \,\dx\dt.
\end{align*}
Our next aim is to dominate each term on the right-hand side of the energy inequality by a truncation integral $\|(|u|-\sfk_n)_{+}\|_{a,Q_n}$ with the same exponent $a$.
For this purpose, we first establish two estimates that will be used repeatedly: For $a\ge b>0$ we have
\begin{align}\label{energy-cascade-1}
    \iint_{Q_n} |u|^{b}
    \mathbf 1_{\{|u|>\sfk_{n+1}\}}\,\dx\dt
    \le
    \frac{2^{(n+1)a}}{\sfk^{a-b}}
    \iint_{Q_n}
    &\big(|u|-\sfk_n\big)_{+}^{a}\,\dx\dt
\end{align}
and
\begin{align}\label{energy-cascade-2}
    \big|Q_n\cap \{|u|>\sfk_{n+1}\}\big|
    \le
    \Big( \frac{2^{n+1}}{\sfk}\Big)^{a}
    \iint_{Q_n}
    \big(|u|-\sfk_n\big)_+^{a}\,\dx\dt.
\end{align}
To show the first estimate, we use the relation between consecutive levels,
\begin{equation*}%\label{k_n-k_n+1}
    \sfk_n=\sfk_{n+1} \frac{2^{n+1}-2}{2^{n+1}-1},
\end{equation*}
and compute
\begin{align*}
    \iint_{Q_n}
    &\big(|u|-\sfk_n\big)_{+}^{a}\,\dx\dt
    \ge 
    \iint_{Q_n}\big(|u|-\sfk_n\big)_{+}^{b}
    \big(|u|-\sfk_n\big)_+^{a-b} 
    \mathbf 1_{\{|u|>\sfk_{n+1}\}}\,\dx\dt \\
    &\ge
    \big( \sfk_{n+1}-\sfk_n\big)^{a-b}
    \iint_{Q_n}
    \bigg[|u|-\sfk_{n+1}
    \frac{2^{n+1}-2}{2^{n+1}-1}\bigg]^{b} \mathbf 1_{\{|u|>\sfk_{n+1}\}}\,\dx\dt\\
    &=
    \Big( \frac{\sfk}{2^{n+1}}\Big)^{a-b}
    \iint_{Q_n} |u|^{b}
    \bigg[1-
    \frac{\sfk_{n+1}}{|u|}\frac{2^{n+1}-2}{2^{n+1}-1}\bigg]^{b} 
    \mathbf 1_{\{|u|>\sfk_{n+1}\}}\,\dx\dt\\
    &\ge
    \Big( \frac{\sfk}{2^{n+1}}\Big)^{a-b}
    \iint_{Q_n} |u|^{b}
    \bigg[1-
    \frac{2^{n+1}-2}{2^{n+1}-1}\bigg]^{b} 
    \mathbf 1_{\{|u|>\sfk_{n+1}\}}\,\dx\dt\\
    &\ge
    \frac{\sfk^{a-b}}{2^{(n+1)a}}
    \iint_{Q_n} |u|^{b}
    \mathbf 1_{\{|u|>\sfk_{n+1}\}}\,\dx\dt.
\end{align*}
To prove the second estimate, we compute
\begin{align*}
    \iint_{Q_n}
    \big(|u|-\sfk_n\big)_+^{a}\,\dx\dt
    &
    \ge 
    \iint_{Q_n}\big(|u|-\sfk_n\big)_{+}^{a} 
    \mathbf 1_{\{|u|>\sfk_{n+1}\}}
    \,\dx\dt \\
    &
    \ge
    \big(\sfk_{n+1}-\sfk_n\big)^{a}\big|Q_n\cap \{|u|>\sfk_{n+1}\}\big| \\
    & 
    =
    \Big( \frac{\sfk}{2^{n+1}}\Big)^{a}
    \big|Q_n\cap \{|u|>\sfk_{n+1}\}\big|.
\end{align*}

Since $1<p\le\frac{2N}{N+2}$, the condition $\boldsymbol{\lambda}_\sfr=N(p-2)+p\sfr>0$ forces 
$\sfr>\frac{N(2-p)}{p}\ge2$. To proceed, the decisive point is the comparison of $\sfr$ with $p'+2$. 

In the {\bf case} $\boldsymbol{2\,\le\frac{N(2-p)}{p}}<\sfr<p'+2$ 
we will dominate each term on the right-hand side of the energy inequality in terms of the $L^{p'+2}$-integral of the truncated function $(|u|-\sfk_n)_+$. 
Using \eqref{energy-cascade-1} with $(a,b)=(p'+2, p'+p)$ we estimate the term $\sfI_1$ as 
\begin{equation}\label{est:|u|^p'+p}
    \sfI_1
    \le
    \frac{2^{(n+1)(p'+2)}}{\sfk^{2-p}}
     \iint_{Q_n}
    \big(|u|-\sfk_n\big)_{+}^{p'+2}\,\dx\dt
    =
    2^{(n+1)(p'+2)}\sfk^{p'+p}\widetilde{\sfY}_n,
\end{equation}
where we set
\begin{equation*}
    \widetilde{\sfY}_n:= 
    \frac{1}{\sfk^{p'+2}}\iint_{Q_n}
    \big(|u|-\sfk_n\big)_{+}^{p'+2}\,\dx\dt.
\end{equation*}
Similarly, the $\sfI_2$-term is estimated by applying \eqref{energy-cascade-1} with $a=b=p'+2$ as
\begin{equation}\label{est:|u|^p'+2}
    \sfI_2
    \le
    2^{(n+1)(p'+2)}
     \iint_{Q_n}
    \big(|u|-\sfk_n\big)_{+}^{p'+2}\,\dx\dt
    =
    2^{(n+1)(p'+2)}\sfk^{p'+2} \widetilde{\sfY}_n.
\end{equation}
Applying \eqref{energy-cascade-2} with $a=p'+2$ we have
\begin{align}\label{est:meas-p'+2}\nonumber
    \big|Q_n\cap \{|u|>\sfk_{n+1}\}\big|
    &\le
    \frac{2^{(n+1)(p'+2)}}{\sfk^{p'+2}}
    \iint_{Q_n}
    \big(|u|-\sfk_n\big)_+^{p'+2}\,\dx\dt\\
    &=2^{(n+1)(p'+2)}\widetilde{\sfY}_n.
\end{align}
We now estimate the non–degeneracy term $\sfI_3$. By Hölder’s inequality, together with \eqref{est:|u|^p'+p} and \eqref{est:meas-p'+2}, we obtain
\begin{align}\label{est:mu-}\nonumber
     \sfI_3
     &\le
     \mu^p\bigg[\iint_{Q_{n}}|u|^{p'+p}
     \mathbf{1}_{\{|u|>\sfk_{n+1}\}}
    \,\dx\dt\bigg]^\frac1p 
    \big|Q_n\cap \{|u|>\sfk_{n+1}\}\big|^{1-\frac{1}{p}}
    \\\nonumber
    &\le
    \mu^p
    \big[ 2^{(n+1)(p'+2)}\sfk^{p'+p}\widetilde{\sfY}_n\big]^\frac1p
    \big[
    2^{(n+1)(p'+2)}\widetilde{\boldsymbol Y}_n
    \big]^{1-\frac{1}{p}}\\
    &=
    \mu^p 2^{(n+1)(p'+2)}\sfk^{p'}\widetilde{\sfY}_n.
\end{align}
In the same vein we estimate $\sfI_4$ (the 
$\sfF$-term). Using the local boundedness of $u$, H\"older's inequality  with exponents $q$ and $\frac{q}{q-1}$ and the measure bound \eqref{est:meas-p'+2}, we obtain
\begin{align}\label{est:F}\nonumber
    \sfI_4
    &\le
    \|u\|_{\infty, Q_o}^{p'}
    \bigg[
    \iint_{Q_o}|\sfF|^{qp}\,\dx\dt
    \bigg]^\frac{1}{q}
    \big|Q_n\cap \{|u|>\sfk_{n+1}\}\big|^{1-\frac1q}
    \\%\nonumber
    % &\le 
    % \|u\|_{\infty, Q_o}^{p'}
    % \|\sfF\|_{qp, Q_o}^{p}
    % \frac{2^{(n+1)(p'+2)(1-\frac{1}{q})}}{\sfk^{(p'+2)(1-\frac1q)}}
    % \bigg[
    % \iint_{Q_n}
    % \big(|u|-\sfk_n\big)_+^{p'+2}\,\dx\dt    
    % \bigg]^{1-\frac{1}{q}}
    % \\
    &\le
    \|u\|_{\infty, Q_o}^{p'}
    \|\sfF\|_{qp, Q_o}^{p}
    2^{(n+1)(p'+2)(1-\frac{1}{q})}
    \widetilde{\sfY}_n^{1-\frac1q}.
\end{align}
We now handle the $\sff$-term. 
By the local boundedness of $u$, Hölder’s inequality with exponents $qp'$, $(1-\frac{1}{qp'})^{-1}$  gives
\begin{align}\label{est:f}\nonumber
     \sfI_5
    &\le
    \|u\|_{\infty ,Q_o}^{p'+1}
    \bigg[
    \iint_{Q_{n}}|\sff|^{qp'}\dx\dt
    \bigg]^\frac{1}{qp'}
    \big|Q_n\cap\{|u|>\sfk_{n+1}\}\big|^{1-\frac{1}{qp'}}\\\nonumber
    &\le
    \|u\|_{\infty ,Q_o}^{p'+1}
    \|\sff\|_{qp' ,Q_o}
    \big|Q_n\cap\{|u|>\sfk_{n+1}\}\big|^{1-\frac{1}{q}+(\frac{1}{q}-
    \frac{1}{qp'}) }\\\nonumber
    &\le
     |Q_o|^\frac{1}{qp}\|u\|_{\infty ,Q_o}^{p'+1}
    \|\sff\|_{qp' ,Q_o}
    \big|Q_n\cap\{|u|>\sfk_{n+1}\}\big|^{1-\frac{1}{q}}
    \\
    &\le
     |Q_o|^\frac{1}{qp}\|u\|_{\infty ,Q_o}^{p'+1}
    \|\sff\|_{qp' ,Q_o}
    2^{(n+1)(p'+2)(1-\frac{1}{q})}
    \widetilde{\sfY}_n^{1-\frac1q}.
\end{align}
Here, we also used \eqref{energy-cascade-2} with $a=p'+2$ in the last line.
Collecting the bounds 
\eqref{est:|u|^p'+p}, \eqref{est:|u|^p'+2}, \eqref{est:mu-}, \eqref{est:F}, and \eqref{est:f}, and inserting them into the energy inequality, we obtain for each $n\in\N_0$
\begin{align*}
    &\sup_{t\in (-\theta_n,0]}\int_{B_n\times\{t\}}
    \big(|u|-\sfk_{n+1}\big)_+^{p'+2}\zeta_n^p\,\dx
    +
    \iint_{Q_{n}}
     \big|\nabla \big[ \big(|u|-\sfk_{n+1}\big)_+^{p'}\zeta_n\big]\big|^{p}\, \dx\dt\\
    &\qquad\le
    C\,2^{(n+1)(p'+p+2)}\sfk^{p'+2}
    \bigg[
    \frac{1}{(1-\tau)^p\varrho^p}\frac{1}{\sfk^{2-p}}
    +
    \frac{1}{(1-\tau)\theta}+\frac{\mu^p}{\sfk^2}
    \bigg]
    \widetilde{\sfY}_n
    \\
    &\qquad\phantom{\le\,}+
    C\Big[ \|u\|_{\infty, Q_o}^{p'}
    \|\sfF\|_{qp, Q_o}^{p}+
     |Q_o|^\frac{1}{qp}\|u\|_{\infty ,Q_o}^{p'+1}
    \|\sff\|_{qp' ,Q_o}\Big]
    2^{(n+1)(p'+2)(1-\frac{1}{q})}
    \widetilde{\sfY}_n^{1-\frac{1}{q}}\\
    &\qquad\le
    \frac{C2^{(n+1)(p'+p+2)}\sfk^{p'+2}}{(1-\tau)^p\theta}
    \bigg[
    \frac{\theta}{\varrho^p\sfk^{2-p}}\Big(1+\frac{\mu^p\rho^p}{\sfk^p}\Big)
    +
   1
    \bigg]
    \widetilde{\sfY}_n
    \\
    &\qquad\phantom{\le\,}+
    C 2^{(n+1)(p'+2)(1-\frac{1}{q})} \sfR
    \widetilde{\sfY}_n^{1-\frac{1}{q}},
\end{align*}
where
\begin{equation}\label{def:R}
    \sfR:= \|u\|_{\infty, Q_o}^{p'}
    \|\sfF\|_{qp, Q_o}^{p}+
     |Q_o|^\frac{1}{qp}\|u\|_{\infty ,Q_o}^{p'+1}
    \|\sff\|_{qp' ,Q_o}.
\end{equation}
We now set the first largeness requirements on $\sfk$: indeed, we require
\begin{equation*}%\label{est:k-mu}
    \frac{\mu^p\rho^p}{\sfk^{p}}\le 1
    \quad\Longleftrightarrow\quad
    \sfk\ge \mu\rho,
\end{equation*}
and 
\begin{equation*}%\label{est:k-theta}
    \frac{\theta}{\varrho^p\sfk^{2-p}}\le 1
    \quad\Longleftrightarrow\quad
    \sfk\ge \Big(\frac{\theta}{\varrho^p}\Big)^\frac{1}{2-p}.
\end{equation*}
That is, taking
\begin{equation}\label{first-choice-k}
    \sfk\ge \max\bigg\{ \Big(\frac{\theta}{\varrho^p}\Big)^\frac{1}{2-p},\,\mu\rho\bigg\}
\end{equation}
the bracketed factors in front of $\widetilde{\sfY}_n$ are bounded by $3$; we obtain
\begin{align}\label{est:en-prel}\nonumber
    \sup_{t\in (-\theta_n,0]}&\int_{B_n\times\{t\}}
    \big(|u|-\sfk_{n+1}\big)^{p'+2}\zeta_n^p\,\dx
    +
    \iint_{Q_{n}}
     \big|\nabla \big[ \big(|u|-\sfk_{n+1}\big)_+^{p'}\zeta_n\big]\big|^{p}\, \dx\dt\\\nonumber
    &\le
     \frac{C2^{(n+1)(p'+p+2)}\sfk^{p'+2}}{(1-\tau)^p\theta}
    \widetilde{\sfY}_n
    +
     C 2^{(n+1)(p'+2)(1-\frac{1}{q})}\sfR
   \widetilde{\sfY}_n^{1-\frac{1}{q}}\\
    &\le
    C2^{(n+1)(p'+p+2)}
    \bigg[
    \frac{\sfk^{p'+2}\widetilde{\sfY}_n}{(1-\tau)^p\theta}
    +\sfR
    \widetilde{\sfY}_n^{1-\frac{1}{q}}
    \bigg].
\end{align}
Next, we invoke the parabolic Sobolev embedding (Proposition 3.1 in Chapter I of \cite{DiBe}) with
\begin{equation*}
    (v,p,m,q) =\Big(\big(|u|-\sfk_{n+1}\big)_{+}^{p'} \zeta_n^{p'}, p, \sfm, \sfq\Big),
\end{equation*}
where
\begin{equation}\label{def:bq}
    \sfm:=\frac{p'+2}{p'},\qquad\sfq:= p\frac{N+\sfm}{N}.
\end{equation}
This gives
\begin{align}\label{est:q-int}%\nonumber
    \iint_{Q_n}&\big[\big(|u|-\sfk_{n+1}\big)_+^{p'} \zeta_n^{p'}\big]^{\sfq}\,\dx\dt\nonumber\\
    &\le
    C_{\rm Sob}^{\sfq}
    \bigg[\iint_{Q_n}\big|\nabla\big[
    \big(|u|-\sfk_{n+1}\big)_+^{p'}\zeta_n^{p'}
    \big]\big|^p\,\dx\dt\bigg]\nonumber\\
    &\qquad\,\,\,\cdot 
    \bigg[\sup _{t\in (-\theta_n,0]}  
    \int_{B_n\times\{t\}}
    \big[\big(|u|-\sfk_{n+1}\big)_+^{p'}\zeta^{p'}_n\big]^{\sfm  }\,\dx\bigg]^{\frac{p}{N}} \\
    &\le
    C\Big[\mbox{right-hand side of \eqref{est:en-prel}}\Big]^{1+\frac{p}{N}}.\nonumber
\end{align} 
Observe that the hypothesis $\boldsymbol\lambda_\sfr=N(p-2)+p\sfr>0$ together with the small $\sfr$-assumption $\sfr<p'+2$, implies $\sfq>\sfm$. Indeed,
\begin{align*}
    \sfq
    &=
    \frac{Np'p+p(p'+2)}{Np'}
    =
    \frac{N(p+p')+p(p'+2)}{Np'}\\
    &= 
     \frac{N(p-2)+p(p'+2)+N(p'+2)}{Np'}
     >
     \frac{N(p-2)+p\sfr+N(p'+2)}{Np'}\\
     &=
     \frac{\boldsymbol\lambda_\sfr}{Np'}+\frac{p'+2}{p'}
     >\sfm.
\end{align*}
Consequently, Hölder’s inequality yields the integrability gain
\begin{align*}
      \iint_{Q_{n+1}}&
    \big(|u|-\sfk_{n+1}\big)_{+}^{p'+2}\,\dx\dt 
    \le 
    \iint_{Q_n}
    \big[\big(|u|-\sfk_{n+1}\big)_{+}^{p'} \zeta_n^{p'}\big]^{\sfm}\,\dx\dt \\
    &\le
    \bigg[\iint_{Q_n}
    \big[\big(|u|-\sfk_{n+1}\big)_+^{p'}\zeta_n^{p'}\big]^{\sfq}
    \,\dx\dt\bigg]^{\frac{\sfm}{\sfq }}
    \big|Q_n\cap \{|u|>\sfk_{n+1}\} \big|^{1- \frac{\bom}{\sfq }}.
\end{align*}
Let us continue to estimate the first term by \eqref{est:q-int} and the second term by \eqref{est:meas-p'+2}. After a simple algebraic manipulation we get
\begin{align*}
    \iint_{Q_{n+1}}&
    \big(|u|-\sfk_{n+1}\big)_{+}^{p'+2}\,\dx\dt \\
    &\le
    C \Big[\mbox{right-hand side of \eqref{est:en-prel}}\Big]^{\frac{\sfm}{\sfq }(1+\frac{p}N)}
    \Big[ 
     2^{(n+1)(p'+2)}\widetilde{\sfY}_n
    \Big]^{1- \frac{\sfm}{\sfq }}\\
    &\le
    C 2^{(n+1)(p'+p+2)(1+\frac{p}{N}\frac{\bom}{\sfq })}\widetilde{\sfY}_n^{1- \frac{\sfm}{\sfq}} 
    \bigg[
    \frac{\sfk^{p'+2}\widetilde{\sfY}_n}{(1-\tau)^p\theta}
    +
    \sfR
    \widetilde{\sfY}_n^{1-\frac{1}{q}}
    \bigg]^{\frac{\sfm}{\sfq}(1+\frac{p}N)}
   \\
    &\le
    C 2^{(n+1)(p'+p+2)(1+\frac{p}{N}\frac{\sfm}{\sfq})}
    \sfk^{(p'+2)\frac{\sfm}{\sfq }\frac{N+p}{N}}
    \widetilde{\sfY}_n^{1+ \frac{p}{N}\frac{\sfm}{\sfq}(1-\frac{N+p}{qp})}\\
    &\phantom{\le\,\,\, }
    \cdot
   \bigg[
\frac{\widetilde{\sfY}_n^\frac1q}{(1-\tau)^p\theta}
    +\frac{\sfR}{\sfk^{p'+2}}
    \bigg]^{\frac{\sfm}{\sfq }\frac{N+p}{N}}.
\end{align*}
Dividing both sides by $\sfk^{p'+2}$ changes the power of $\sfk$ on the right. Using $\sfq=p\,\frac{N+\sfm}{N}$ and $\sfm=\frac{p'+2}{p'}$, we simplify the power as
\begin{align*}
    &(p'+2)\frac{\sfm}{\sfq }\frac{N+p}{N}
    -(p'+2)
    =
    (p'+2)
    \bigg[
    \frac{\sfm(N+p)}{p( N+\sfm)}-1
    \bigg]\\
    &\qquad\qquad
    =
    \frac{p'+2}{p( N+\sfm)}
    \bigg[
    \sfm(N+p)
    -
    p \big( N+\sfm\big)
    \bigg]\\
    &\qquad\qquad
    =
    \frac{\sfm}{p\frac{N+\sfm}{N}}\big(p'+2-pp'\big)
    =
    \frac{\sfm}{\sfq}(2-p),
\end{align*}
where we used the identity $pp'=p+p'$. Hence, dividing by $\sfk^{p'+2}$
and collecting constants, we arrive at the recursive bound
\begin{align*}
    \widetilde{\sfY}_{n+1}
    &\le
    C \sfb^{n+1}
    \sfk^{\frac{\sfm}{\sfq }(2-p)}
    \widetilde{\sfY}_n^{1+ \frac{p}{N}\frac{\sfm}{\sfq}(1-\frac{N+p}{qp})}
   \bigg[
   \frac{\widetilde{\sfY}_o}{(1-\tau)^{qp}\theta^q}
    +\frac{
    \sfR^q}{\sfk^{q(p'+2)}}
    \bigg]^{\frac{p}{N}\frac{\sfm}{\sfq }\frac{N+p}{qp}},
\end{align*}
where 
\begin{equation*}
    \sfb:=2^{(p'+p+2)(1+\frac{p}{N}\frac{\sfm}{\sfq })}.
\end{equation*}
Passing to mean values in $\widetilde{\sfY}_{n+1}$, $\widetilde{\sfY}_{n}$, and $\widetilde{\sfY}_{o}$
(and collecting the resulting volume factors) yields the recursive inequality
\begin{align*}%\label{est:rec-prelim}\nonumber
    \sfY_{n+1}
    &\le
    C \sfb^{n+1}
    \sfk^{\frac{p'+2}{\sfq p'}(2-p)}|Q_o|^{\frac{p}{N}\frac{\sfm}{\sfq }}
    \sfY_n^{1+ \frac{p}{N}\frac{\sfm}{\sfq }(1-\frac{N+p}{qp})}\\
    &\phantom{\le\,\,\, }\cdot
   \bigg[
   \frac{\sfY_o}{(1-\tau)^{qp}\theta^q}
    +\frac{
    \sfR^q}{|Q_o|\sfk^{q(p'+2)}}
    \bigg]^{\frac{p}{N}\frac{\sfm}{\sfq }\frac{N+p}{qp}},
\end{align*}
where
\begin{equation*}
    \sfY_n:=\frac{1}{\sfk^{p'+2}}\biint_{Q_n}
    \big(|u|-\sfk_n\big)_{+}^{p'+2}\,\dx\dt.
\end{equation*}
Here $\sfR$ abbreviates the contribution arising from the source terms in the previous step.

With a judicious choice of the truncation level 
$\sfk$, the right–hand side can be streamlined. Indeed, the bracket splits into three contributions -- coming from $\sfY_o$, the $\sfF$–term, and the $\sff$–term, namely
\begin{equation*}
    \Big[\dots\Big]
    =
    \frac{\sfY_o}{(1-\tau)^{qp}\theta^q} 
    +
    \frac{\|u\|_{\infty, Q_o}^{qp'}
    \|\sfF\|_{qp, Q_o}^{qp}}{|Q_o|\sfk^{q(p'+2)}}
    +
    \frac{
     |Q_o|^\frac{1}{p}\|u\|_{\infty ,Q_o}^{q(p'+1)}
    \|\sff\|_{qp' ,Q_o}^q }{|Q_o|\sfk^{q(p'+2)}}.
\end{equation*}
We now fix the truncation level $\sfk$
so large that
\begin{equation}\label{k-conditions}
\begin{aligned}
    %\sfY_o=\frac{1}{\sfk^{p'+2}}\biint_{Q_o}|u|^{p'+2}\dx\dt
    \frac{\sfY_o}{(1-\tau)^{qp}\theta^q}
    &\le
    \frac{\theta^{\frac{N+p}{p}-q}(1-\tau)^{N+p-qp}}{|Q_o|\sfk^{\frac{N}{p}(2-p)}},\\
        \frac{\|u\|_{\infty, Q_o}^{qp'}
    \|\sfF\|_{qp, Q_o}^{qp}}{|Q_o|\sfk^{q(p'+2)}}
    &\le        
    \frac{\theta^{\frac{N+p}{p}-q}
   (1-\tau)^{N+p-qp}}{|Q_o|\sfk^{\frac{N}{p}(2-p)}},\\
    \frac{ |Q_o|^\frac{1}{p}\|u\|_{\infty ,Q_o}^{q(p'+1)}
    \|\sff\|_{qp' ,Q_o}^q}{|Q_o|\sfk^{q(p'+2)}}&\le\frac{\theta^{\frac{N+p}{p}-q}
   (1-\tau)^{N+p-qp}}{|Q_o|\sfk^{\frac{N}{p}(2-p)}}
    .
\end{aligned}
\end{equation}
Recalling $|Q_o|=|B_1|\rho^N\theta$
and $\boldsymbol{\lambda}_{s}= N(p-2)+ps$, solving the first condition for $\sfk$ yields
\begin{align}\label{est:k-|u|^p'+2}\nonumber
    \sfk^{\frac{\boldsymbol\lambda_{p'+2}}{p}}
    %&=
    %\sfk^{\frac{N}{p}(p-2)+ p'+2}\\\nonumber
    &\ge
    \frac{|Q_o|}{\theta^\frac{N+p}{p}(1-\tau)^{N+p}}\biint_{Q_o}|u|^{p'+2}\,\dx\dt\\
    &=
    \frac{|B_1|}{(1-\tau)^{N+p}}\Big(\frac{\varrho^p}{\theta}\Big)^\frac{N}{p}\biint_{Q_o}|u|^{p'+2}\,\dx\dt.
\end{align}
Similarly, the second condition is ensured provided
\begin{align}\label{est:k-F}\nonumber
    \sfk^\frac{\boldsymbol\lambda_{q(p'+2)}}{p}
    %&=
    %\sfk^{\frac{N}{p}(p-2)+ q(p'+2)}\\\nonumber
    &\ge
    \|u\|_{\infty, Q_o}^{qp'}
     \theta^{q-\frac{N+p}{p}}|Q_o|
    \biint_{Q_o} |\sfF|^{qp}\,\dx\dt\\
    &=
    |B_1| \|u\|_{\infty, Q_o}^{qp'}
    \Big(\frac{\theta}{\varrho^p}\Big)^{q-\frac{N}{p}}
    \biint_{Q_o} |\varrho\, \sfF|^{qp}\,\dx\dt.
\end{align}
The third one is satisfied provided
\begin{align}\label{est:k-f}\nonumber
    \sfk^\frac{\boldsymbol
    \lambda_{q(p'+2)}}{p}
    %=
    %\sfk^{\frac{N}{p}(p-2)+2q}
    &\ge
    \theta^{q-\frac{N+p}{p}}
    |Q_o|^\frac1p \|u\|_{\infty ,Q_o}^{q(p'+1)}\bigg[
    \iint_{Q_o}|\sff|^{qp'}\,\dx\dt
     \bigg]^\frac{1}{p'}\\%\nonumber
    % &=
    % \theta^{q-\frac{N+p}{p}}
    % |Q_o|\rho^{-qp} \|u\|_{\infty ,Q_o}^q\bigg[
    % \biint_{Q_o}|\rho^p \sff|^{qp'}\,\dx\dt
    % \bigg]^\frac{1}{p'}\\\nonumber
    % &=
    % |B_1|\theta^{q-\frac{N}{p}}
    % \rho^{N-qp} \|u\|_{\infty ,Q_o}^q\bigg[
    % \biint_{Q_o}|\rho^p \sff|^{qp'}\,\dx\dt
    % \bigg]^\frac{1}{p'}\\
    &=
    |B_1|\|u\|_{\infty ,Q_o}^{q(p'+1)}
    \Big(\frac{\theta}{\rho^p}\Big)^{q-\frac{N}{p}}
     \bigg[
    \biint_{Q_o}|\rho^p \sff|^{qp'}\,\dx\dt
    \bigg]^\frac{1}{p'}.
\end{align}
The first condition \eqref{est:k-|u|^p'+2} is certainly satisfied if we choose 
$\sfk$ so that
\begin{equation}\label{est:k-|u|^p'+2-*}
    \sfk^{\frac{\boldsymbol\lambda_{p'+2}}{p}}
    \ge
    \frac{|B_1|\|u\|_{\infty, Q_o}^{p'+2-\sfr}}{(1-\tau)^{N+p}}\Big(\frac{\varrho^p}{\theta}\Big)^\frac{N}{p}\biint_{Q_o}|u|^\sfr\,\dx\dt.
\end{equation}
With $\sfk$
fixed so that \eqref{first-choice-k}, \eqref{est:k-|u|^p'+2-*}, \eqref{est:k-F}, and \eqref{est:k-f} all hold, the recursion collapses to
\begin{align*}%\label{est:it-final}
    \sfY_{n+1}
    &\le
    C \sfb\Bigg[
    \frac{
    \sfk^{\frac{N}{p}(2-p)}
    |Q_o|}{(1-\tau)^{N+p}\theta^\frac{N+p}{p}}
    \Bigg]^{\frac{p}{N}\frac{\sfm}{\sfq }(1-\frac{N+p}{qp})}
    \sfb^{n}
    \sfY_n^{1+ \frac{p}{N}\frac{\sfm}{\sfq }(1-\frac{N+p}{qp})}.
\end{align*}
With
\begin{equation*}
    \alpha:=
    \frac{p}{N}\frac{\sfm}{\sfq}\Big(1-\frac{N+p}{qp}\Big),
    \quad
    {\sf C}
    :=
     C \sfb\Bigg[
    \frac{
    \sfk^{\frac{N}{p}(2-p)}
    |Q_o|}{(1-\tau)^{N+p}\theta^\frac{N+p}{p}}
    \Bigg]^\alpha,
\end{equation*}
the recursion reads
\begin{equation*}
    \sfY_{n+1}\le {\sf C}\sfb^n\sfY_n^{1+\alpha}.
\end{equation*}
By Lemma \ref{it-lemma}, $\sfY_{n}\to0$ provided 
\begin{align*}
    \sfY_o
    &\le
    {\sf  C}^{-\frac1\alpha}
    \sfb^{-\frac{1}{\alpha^2}}.
\end{align*}
This is equivalent to
\begin{align*}
    \frac1{\sfk^{p'+2}}\biint_{Q_o} |u|^{p'+2}
    \,\dx\dt
    &\le
    \frac{1}{(C\sfb)^{\frac{1}{\alpha}}
    \sfb^{\frac{1}{\alpha^2}}}
    \frac{(1-\tau)^{N+p}\theta^\frac{N+p}{p}}{
    \sfk^{\frac{N}{p}(2-p)}
    |Q_o|}.
\end{align*}
Rearranging and recalling $\boldsymbol\lambda_{p'+2}=N(p-2)+p(p'+2)$
yields the  condition
\begin{align*}
    \sfk^\frac{\boldsymbol\lambda_{p'+2}}{p}
    =
    \sfk^{\frac{N}{p}(p-2)+ p'+2}
    &\ge
    \frac{(C\sfb^{1+\frac{1}{\alpha}})^\frac{1}{\alpha}
    |Q_o|}{(1-\tau)^{N+p}\theta^\frac{N+p}{p}}
    \biint_{Q_o} |u|^{p'+2}
    \,\dx\dt\\
    &=
    \frac{C|B_1|}{(1-\tau)^{N+p}}\Big(\frac{\varrho^p}{\theta}\Big)^\frac{N}{p}
    \biint_{Q_o} |u|^{p'+2}
    \,\dx\dt.
\end{align*}
This is certainly ensured if we choose $\sfk$ so that
\begin{equation}\label{est:k-r}
    \sfk^{\frac{\boldsymbol\lambda_{p'+2}}{p}}
    \ge
    \frac{C|B_1|\|u\|_{\infty,Q_o}^{p'+2-\sfr}}{(1-\tau)^{N+p}}
    \Big(\frac{\varrho^{p}}{\theta}\Big)^{\frac{N}{p}}
    \biint_{Q_o}|u|^{\sfr}\,\dx\dt .
\end{equation}
Indeed, since $|u|^{p'+2}\le \|u\|_{\infty,Q_o}^{p'+2-\sfr}|u|^{\sfr}$ on $Q_o$, the condition \eqref{est:k-r} implies 
\eqref{est:k-|u|^p'+2-*}. In particular, because $C\ge 1$, \eqref{est:k-r} is strictly stronger than \eqref{est:k-|u|^p'+2-*} and thus also guarantees the corresponding requirement used in the iteration.
By Lemma \ref{it-lemma}, the recursion yields $\sfY_n\to \sfY_\infty =0$ 
as $n\to\infty$.  Since $Q_n\downarrow Q_{\tau\rho,\tau\theta}$ and $\sfk_n\uparrow\sfk$, this implies
$$
    \iint_{Q_{\tau\rho,\tau\theta}}(|u|-\sfk)_+^{p'+2}
    \,\dx\dt=0,
$$
and hence $(|u|-\sfk)_+=0$ a.e.~in the limit cylinder $Q_{\tau\rho,\tau\theta}$. In other words,
\begin{equation}\label{est:sup-1}
    \sup_{Q_{\tau\varrho,\tau\theta}}
    |u|
    \le 
    \sfk,
\end{equation}
provided $\sfk$ has been chosen large enough to satisfy \eqref{first-choice-k}, \eqref{est:k-r}, \eqref{est:k-F}, and \eqref{est:k-f}.
With the preparatory bounds in place, we now fix the truncation level by
\begin{align}\label{def:k}\nonumber
    \sfk
    &:=
    \max\Bigg\{
    \|u\|_{\infty, Q_o}^{1-\frac{\boldsymbol\lambda_{\sfr}}{\boldsymbol\lambda_{p'+2}}}
    \bigg[
    \frac{C|B_1|}{(1-\tau)^{N+p}}\Big(\frac{\varrho^p}{\theta}\Big)^\frac{N}{p}
    \biint_{Q_o} |u|^{\sfr} 
    \,\dx\dt
    \bigg]^\frac{p}{\boldsymbol\lambda_{p'+2}},\\\nonumber
    &\qquad\qquad\,\,\,
    \|u\|_{\infty, Q_o}^{1-\frac{\boldsymbol\lambda_{2q}}{\boldsymbol\lambda_{q(p'+2)}}}\bigg[
       |B_1| 
    \Big(\frac{\theta}{\varrho^p}\Big)^{q-\frac{N}{p}}
    \biint_{Q_o} |\varrho\, \sfF|^{qp}\,\dx\dt
    \bigg]^\frac{p}{\boldsymbol\lambda_{q(p'+2)}},
    \\\nonumber
    &\qquad\qquad\,\,\,
   \|u\|_{\infty, Q_o}^{1-\frac{\boldsymbol\lambda_{q}}{\boldsymbol\lambda_{q(p'+2)}}} \bigg[
       |B_1|^{p'} 
    \Big(\frac{\theta}{\varrho^p}\Big)^{p'(q-\frac{N}{p})}
    \biint_{Q_o} |\rho^p \sff|^{qp'}\,\dx\dt
    \bigg]^\frac{p}{p'\boldsymbol\lambda_{q(p'+2)}},\\
     &\qquad\qquad\,\,\,
    \Big(\frac{\theta}{\varrho^p}\Big)^\frac{1}{2-p},\,\mu\rho
    \Bigg\}.
\end{align}
With $\sfk$ chosen as in \eqref{def:k}, the conditions \eqref{first-choice-k}, \eqref{est:k-r}, \eqref{est:k-F}, and  \eqref{est:k-f} are all satisfied, so the iteration applies and, by Lemma \ref{it-lemma}, $\sfY_n\to0$. Consequently, the conclusion \eqref{est:sup-1} follows. This completes the De Giorgi iteration in the case $2<\sfr<p'+2$.

The bound \eqref{est:sup-1} with \eqref{def:k} can be refined via an interpolation. To set up the interpolation, we analyze the first three entries in the maximum more closely. For the first entry, Young’s inequality with exponents
$\big(1-\frac{\boldsymbol \lambda_{\sfr}}{\boldsymbol \lambda_{p'+2}}\big)^{-1}$ and $\frac{\boldsymbol \lambda_{p'+2}}{\boldsymbol \lambda_{\sfr}}$ yields
\begin{align}\label{def:k-first-entry}
    \big[\mbox{first entry}\big]
    % &=
    % \|u\|_{\infty, Q_o}^{1-\frac{\boldsymbol \lambda_{r}}{\boldsymbol \lambda_{p'+2}}}
    % \bigg[\frac{C|B_1|}{(1-\tau)^{N+p}}\Big(\frac{\varrho^p}{\theta}\Big)^\frac{N}{p}
    % \biint_{Q_o} |u|^{r}
    % \,\dx\dt
    % \bigg]^\frac{p}{\boldsymbol\lambda_{p'+2}}\\
    &\le
    \tfrac16
    \|u\|_{\infty, Q_o}
    +
    C\bigg[\frac{|B_1|}{(1-\tau)^{N+p}}\Big(\frac{\varrho^p}{\theta}\Big)^\frac{N}{p}
    \biint_{Q_o} |u|^{\sfr}
    \,\dx\dt
    \bigg]^\frac{p}{\boldsymbol\lambda_{\sfr}},
\end{align}
where the constant $C$ in \eqref{def:k-first-entry} denotes $6^{\frac{\boldsymbol\lambda_{p'+2}}{\boldsymbol\lambda_{\sfr}}} C^{\frac p{\boldsymbol\lambda_{\sfr}}}$, with $C$ taken from \eqref{def:k}.
For the second entry (the \(\sfF\)-term), Young’s inequality with exponents
$\big(1-\frac{\boldsymbol \lambda_{2q}}{\boldsymbol \lambda_{q(p'+2)}}\big)^{-1}$ and
$\frac{\boldsymbol \lambda_{q(p'+2)}}{\boldsymbol \lambda_{2q}}$ gives
\begin{align}\label{2nd-entry}
    \big[\mbox{second entry}\big]
    % &=
    % \|u\|_{\infty, Q_o}^{1-\frac{\boldsymbol \lambda_{2q}}{\boldsymbol \lambda_{q(p'+2)}}}
    % \bigg[|B_1|\Big(\frac{\theta}{\varrho^p}\Big)^{q-\frac{N}{p}}
    % \biint_{Q_o} |\varrho \sfF|^{qp}\,\dx\dt
    % \bigg]^\frac{p}{\boldsymbol\lambda_{q(p'+2)}}\\
    &\le
    \tfrac16
    \|u\|_{\infty, Q_o}
    +
    C\bigg[|B_1|\Big(\frac{\theta}{\varrho^p}\Big)^{q-\frac{N}{p}}
    \biint_{Q_o} |\varrho\, \sfF|^{qp}\,\dx\dt
    \bigg]^\frac{p}{\boldsymbol\lambda_{2q}}.
\end{align}
For the third entry (the \(\sff\)-term),  Young’s inequality with exponents
$\big(1-\frac{\boldsymbol \lambda_{q}}{\boldsymbol \lambda_{q(p'+2)}}\big)^{-1}$
and $\frac{\boldsymbol \lambda_{q(p'+2)}}{\boldsymbol \lambda_{q}}$ implies
\begin{align}\label{3rd-entry}
    \big[\mbox{third entry}\big]
    % &=
    % \|u\|_{\infty, Q_o}^{1-\frac{\boldsymbol \lambda_{q}}{\boldsymbol \lambda_{q(p'+2)}}}
    % \bigg[
    %  |B_1|^{p'}
    % \Big(\frac{\theta}{\rho^p}\Big)^{p'(q-\frac{N}{p})}
    % \biint_{Q_o}|\rho^p \sff|^{qp'}\,\dx\dt
    % \bigg]^\frac{p}{p'\boldsymbol\lambda_{q(p'+2)}}\\
    &\le
    \tfrac16
    \|u\|_{\infty, Q_o}
    +
    C \bigg[
     |B_1|^{p'}
    \Big(\frac{\theta}{\rho^p}\Big)^{p'(q-\frac{N}{p})}
    \biint_{Q_o}|\rho^p \sff|^{qp'}\,\dx\dt
    \bigg]^\frac{p}{p'\boldsymbol\lambda_{q}}.
\end{align}
Therefore, for any $\tau\in [\frac12, 1)$ we obtain
\begin{align}\label{est:sup-interpol}\nonumber
    \sup_{Q_{\tau\varrho,\tau\theta}}|u|
    &\le 
    \tfrac12 \sup_{Q_{\varrho,\theta}}|u|
    +
     \frac{C}{(1-\tau)^{\frac{p}{\boldsymbol\lambda_\sfr}(N+p)}}
    \bigg[
     \Big(\frac{\varrho^p}{\theta}\Big)^\frac{N}{p}
    \biint_{Q_{\varrho,\theta}} |u|^{\sfr}
    \,\dx\dt
    \bigg]^\frac{p}{\boldsymbol\lambda_\sfr}
    \\\nonumber
    &
    \phantom{\le\,}
    +C 
     \bigg[
    \Big(\frac{\theta}{\varrho^p}\Big)^{q-\frac{N}{p}}
    \biint_{Q_{\varrho,\theta}} |\varrho \,\sfF|^{qp}\,\dx\dt
    \bigg]^\frac{p}{\boldsymbol\lambda_{2q}}\\
     &
    \phantom{\le\,}
    +C \bigg[
    \Big(\frac{\theta}{\rho^p}\Big)^{p'(q-\frac{N}{p})}
    \biint_{Q_{\varrho,\theta}}|\rho^p \sff|^{qp'}\,\dx\dt
    \bigg]^\frac{p}{p'\boldsymbol\lambda_{q}}
    +
    \Big(\frac{\theta}{\varrho^p}\Big)^\frac{1}{2-p}
    +
    \mu\rho.
\end{align}
Fix $\tau\in [\frac12, 1)$ and let $\tau\le \tau_1 <\tau_2\le 1$. Applying \eqref{est:sup-interpol} with $(\rho,\theta)$ replaced by $(\tau_2\varrho,\tau_2\theta)$ and $\tau$ replaced by $\frac{\tau_1}{\tau_2}$, we obtain
\begin{align*}
    \sfM (\tau_1)
    &\le
    \tfrac12 
    \sfM (\tau_2)
    +
    \frac{C}{(\tau_2-\tau_1)^{\frac{p}{\boldsymbol\lambda_\sfr}(N+p)}}
    \bigg[
     \Big(\frac{\varrho^p}{\theta}\Big)^\frac{N}{p}
    \biint_{Q_{\varrho,\theta}} |u|^{\sfr}
    \,\dx\dt
    \bigg]^\frac{p}{\boldsymbol\lambda_\sfr}\\
     &
    \phantom{\le\,}
    +C 
     \bigg[
    \Big(\frac{\theta}{\varrho^p}\Big)^{q-\frac{N}{p}}
    \biint_{Q_{\varrho,\theta}} |\varrho \,\sfF|^{qp}\,\dx\dt
    \bigg]^\frac{p}{\boldsymbol\lambda_{2q}}\\
    &\phantom{\le\,}+
    C \bigg[
    \Big(\frac{\theta}{\rho^p}\Big)^{p'(q-\frac{N}{p})}
    \biint_{Q_o}|\rho^p \sff|^{qp'}\,\dx\dt
    \bigg]^\frac{p}{p'\boldsymbol\lambda_{q}}
    +
    C\Big(\frac{\theta}{\varrho^p}\Big)^\frac{1}{2-p}
    +
    C\mu\rho
   ,
\end{align*}
where
\begin{equation*}
    [\tau, 1]\ni s\mapsto \sfM (s)
    :=
    \sup_{Q_{s\varrho,s\theta}}|u|.
\end{equation*}
Invoking Lemma \ref{lem:tech-classical}, we obtain for every $\tau\in [\frac12, 1)$,
\begin{align*}
    \sup_{Q_{\tau\varrho,\tau\theta}}
    |u|
    &\le
    C
    \bigg[
    \frac{1}{(1-\tau)^{N+p}}
     \Big(\frac{\varrho^p}{\theta}\Big)^\frac{N}{p}
    \biint_{Q_{\varrho,\theta}} |u|^{\sfr}
    \,\dx\dt
    \bigg]^\frac{p}{\boldsymbol\lambda_\sfr}\\
     &
    \phantom{\le\,}
    +C 
     \bigg[
    \Big(\frac{\theta}{\varrho^p}\Big)^{q-\frac{N}{p}}
    \biint_{Q_{\varrho,\theta}} |\varrho \,\sfF|^{qp}\,\dx\dt
    \bigg]^\frac{p}{\boldsymbol\lambda_{2q}}
    \\
    &\phantom{\le\,}+
    C \bigg[
    \Big(\frac{\theta}{\rho^p}\Big)^{p'(q-\frac{N}{p})}
    \biint_{Q_o}|\rho^p \sff|^{qp'}\,\dx\dt
    \bigg]^\frac{p}{p'\boldsymbol\lambda_{q}}
    +
    C\Big(\frac{\theta}{\varrho^p}\Big)^\frac{1}{2-p}
    +
    C\mu\rho,
\end{align*}
where $\boldsymbol\lambda_{s}=N(p-2)+ps$, and $C=C(N,p,C_o,C_1,q,\sfr)$.
This proves the claimed quantitative $L^\infty$-estimate in the  case $2<\sfr<p'+2$

{\bf The case} $\boldsymbol{p'+2\le \sfr}$. 
The guiding idea is again to control the individual terms on the right-hand side of the energy inequality by means of the 
$L^\sfr$-integral of the truncated function 
$(|u|-\sfk_n)_+$. More precisely, using \eqref{energy-cascade-1} with $(a,b)=(p'+p,\sfr)$, we get 
\begin{equation}\label{est:|u|^p'+p*}
    \sfI_1
    \le
    \frac{2^{(n+1)r}}{\sfk^{\sfr-(p'+p)}}
     \iint_{Q_n}
    \big(|u|-\sfk_n\big)_{+}^{\sfr}\,\dx\dt
    =
    2^{(n+1)\sfr}\sfk^{p'+p}\widetilde{\sfY}_n,
\end{equation}
where we set
\begin{equation*}
    \widetilde{\sfY}_n:= 
    \frac{1}{\sfk^{\sfr}}\iint_{Q_n}
    \big(|u|-\sfk_n\big)_{+}^{\sfr}\,\dx\dt.
\end{equation*}
Using \eqref{energy-cascade-1} with $(a,b)=(p'+2,\sfr)$ we get 
\begin{equation}\label{est:|u|^p'+2*}
    \sfI_2
    \le
    \frac{2^{(n+1)\sfr}}{\sfk^{\sfr-(p'+2)}}
     \iint_{Q_n}
    \big(|u|-\sfk_n\big)_{+}^{\sfr}\,\dx\dt
    =
    2^{(n+1)\sfr}\sfk^{p'+2} \widetilde{\sfY}_n.
\end{equation}
In addition, applying \eqref{energy-cascade-2} with $a=\sfr$ we have
\begin{equation}\label{est:meas-r}
    \big|Q_n\cap \{|u|>\sfk_{n+1}\}\big|
    \le
    \frac{2^{(n+1)\sfr}}{\sfk^{\sfr}}
    \iint_{Q_n}
    \big(|u|-\sfk_n\big)_+^{\sfr}\,\dx\dt
    =2^{(n+1)\sfr}\widetilde{\sfY}_n.
\end{equation}
For the $\mu$-term $\sfI_3$ we estimate it like \eqref{est:mu-} using Hölder’s inequality with exponents $p$ and $p'$, together with \eqref{est:|u|^p'+p*} and \eqref{est:meas-r} and obtain 
\begin{align*}%\label{mu-term-r>}
%\nonumber
     \sfI_3
     &\le
    \mu^p 2^{(n+1)\sfr}\sfk^{p'}\widetilde{\sfY}_n.
\end{align*}
In the same vein we estimate $\sfI_4$ (the 
$\sfF$-term) like \eqref{est:F}, using the local boundedness of $u$, H\"older's inequality  with exponents $q$ and $\frac{q}{q-1}$ and the measure bound \eqref{est:meas-r}, and obtain
\begin{align*}%\label{est:F*}%\nonumber
    \sfI_4
    \le
    \|u\|_{\infty, Q_o}^{p'}
    \|\sfF\|_{qp, Q_o}^{p}
    2^{(n+1)\sfr(1-\frac{1}{q})}
    \widetilde{\sfY}_n^{1-\frac1q}.
\end{align*}
We estimate the $\sff$-Term $\sfI_5$ like \eqref{est:f}, using the local boundedness of $u$, Hölder’s inequality with exponents $qp'$ and $(1-\frac1{qp'})^{-1}$,
and the measure estimate \eqref{est:meas-r},
to obtain
\begin{align*}%\label{est:f*}
     \sfI_5
    \le
     |Q_o|^\frac{1}{qp}\|u\|_{\infty ,Q_o}^{p'+1}
    \|\sff\|_{qp' ,Q_o}
    2^{(n+1)\sfr(1-\frac{1}{q})}
    \widetilde{\sfY}_n^{1-\frac1q}.
\end{align*}
Using \eqref{est:|u|^p'+p*}, \eqref{est:|u|^p'+2*}, \eqref{est:meas-r}, together with the bounds for the $\mu$, $\sfF$, and $\sff$-terms,
we infer
\begin{align*}
    &\sup_{t\in (-\theta_n,0]}
    \int_{B_n\times\{t\}}
    \big(|u|-\sfk_{n+1}\big)^{p'+2}\zeta_n^p\,\dx
    +
    \iint_{Q_{n}}
     \big|\nabla \big[ \big(|u|-\sfk_{n+1}\big)_+^{p'}\zeta_n\big]\big|^{p}\, \dx\dt\\
    &\quad\le
    C2^{(n+1)(\sfr+p)}\sfk^{p'+2}
    \bigg[
    \frac{1}{(1-\tau)^p\varrho^p}\frac{1}{\sfk^{2-p}}
    +
    \frac{1}{(1-\tau)\theta}+\frac{\mu^p}{\sf\sfk^2}
    \bigg]
    \widetilde{\sfY}_n
    \\
    &\quad\phantom{\le\,}+
    C 2^{(n+1)\sfr(1-\frac{1}{q})}
    \Big[ 
    \|u\|_{\infty, Q_o}^{p'}
    \|\sfF\|_{qp, Q_o}^{p}
    +
    |Q_o|^\frac{1}{qp}\|u\|_{\infty ,Q_o}^{p'+1}
    \|\sff\|_{qp' ,Q_o}
    \Big]
    \widetilde{\sfY}_n^{1-\frac{1}{q}}\\
    &\quad\le
    \frac{C2^{(n+1)(\sfr+p)}\sfk^{p'+2}}{(1-\tau)^p\theta}
    \bigg[
    \frac{\theta}{\varrho^p\sfk^{2-p}}
    \Big(
    1+\frac{\mu^p\rho^p}{\sfk^p}
    \Big)
    +
   1
    \bigg]
    \widetilde{\sfY}_n
    +
    C
    2^{(n+1)\sfr(1-\frac{1}{q})}\sfR
    \widetilde{\sfY}_n
^{1-\frac{1}{q}},
\end{align*}
where $\sfR$ was defined in~\eqref{def:R}.
Exactly as in the case $2<\sfr<p'+2$, we fix the truncation level $\sfk$
so that \eqref{first-choice-k} holds. With this choice the right-hand side of the energy inequality simplifies, and we obtain
\begin{align}\label{est:en-prel-r}\nonumber
    \sup_{t\in (-\theta_n,0]}\int_{B_n\times\{t\}}&
    \big(|u|-\sfk_{n+1}\big)^{p'+2}\zeta_n^p\,\dx
    +
    \iint_{Q_{n}}
     \big|\nabla \big[ \big(|u|-\sfk_{n+1}\big)_+^{p'}\zeta_n\big]\big|^{p}\, \dx\dt\\
    &\le
    C2^{(n+1)(\sfr+p)}
    \bigg[
    \frac{\sfk^{p'+2}\widetilde{\sfY}_n}{(1-\tau)^p\theta}
    +
    \sfR
    \widetilde{\sfY}_n^{1-\frac{1}{q}}
    \bigg].
\end{align}
This is the same energy estimate as \eqref{est:en-prel} except an immaterial change of exponent of $2$.
Next, we apply the  parabolic Sobolev embedding (Proposition 3.1 in Chapter I of \cite{DiBe}) with 
\[
    (v,q,p,m)=\Big(\big(|u|-k_{n+1}\big)_{+}^{p'} \zeta_n^{p'},  p,  \sfm, \sfq\Big),
\]
where $\sfm$ and $\sfq$ are given in \eqref{def:bq}.
The embedding yields
\begin{align}\label{est:q-int*}
    \iint_{Q_n}&\big[\big(|u|-\sfk_{n+1}\big)_+^{p'} \zeta_n^{p'}\big]^{\sfq}\,\dx\dt\\\nonumber
    &\le
    C_{\rm Sob}^{\sfq}
    \bigg[\iint_{Q_n}\big|\nabla\big[
    \big(|u|-\sfk_{n+1}\big)_+^{p'}\zeta_n^{p'}
    \big]\big|^p\,\dx\dt\bigg]\\\nonumber
    &\qquad\,\,\,\cdot 
    \bigg[\sup _{t\in (-\theta_n,0]}  
    \int_{B_n\times\{t\}}
    \big[\big(|u|-\sfk_{n+1}\big)_+^{p'}\zeta^{p'}_n\big]^{\sfm  }\,\dx\bigg]^{\frac{p}{N}} \\\nonumber
    &\le
    C\Big[\mbox{right-hand side of \eqref{est:en-prel-r}}\Big]^{1+\frac{p}{N}}
    .
\end{align}
Noting that $\sfq<p'+2\le \sfr$, we first raise the exponent from $\sfr$ to $p'\sfr$
by Hölder’s inequality, then lower it to 
$p'\sfq$
at the cost of a factor $\|u\|_{\infty ,Q_o}^{\sfr-\sfq}$, and finally use the measure estimate \eqref{est:meas-r} together with \eqref{est:q-int*}. This gives
\begin{align*}
    &\iint_{Q_{n+1}}
    \big[\big(|u|-\sfk_{n+1}\big)_{+}^{\sfr}\,\dx\dt 
    \le 
    \iint_{Q_n}
    \big[\big(|u|-\sfk_{n+1}\big)_{+} \zeta_n\big]^\sfr\,\dx\dt \\
    &\quad\le
    \bigg[\iint_{Q_n}
    \big[\big(|u|-\sfk_{n+1}\big)_+\zeta_n\big]^{p'\sfr}
    \,\dx\dt\bigg]^{\frac{1}{p'}}
    \big|Q_n\cap \{|u|>\sfk_{n+1}\} \big|^{1- \frac{1}{p'}}\\
    &\quad\le
    \|u\|_{\infty, Q_o}^{\sfr-\sfq}
    \bigg[\iint_{Q_n}
    \big[\big(|u|-\sfk_{n+1}\big)_+^{p'}\zeta_n^{p'}
    \big]^{\sfq}
    \,\dx\dt\bigg]^{\frac{1}{p'}}
    \big|Q_n\cap \{|u|>\sfk_{n+1}\} \big|^{ \frac{1}{p}}\\
    &\quad\le
    C 
    \|u\|_{\infty, Q_o}^{\sfr-\sfq}
    \Big[\mbox{right-hand side of \eqref{est:en-prel-r}}\Big]^{\frac{1}{p'}(1+\frac{p}N)}
    \Big[ 
     2^{(n+1)\sfr}\widetilde{\sfY}_n
    \Big]^{\frac{1}{p}}\\
    &\quad\le
    C \|u\|_{\infty, Q_o}^{\sfr-\sfq}
    2^{(n+1)(\sfr+p)(1+\frac{p}{N})}\widetilde{\sfY}_n^{ \frac{1}{p}} 
    \bigg[
    \frac{\sfk^{p'+2}\widetilde{\sfY}_n}{(1-\tau)^p\theta}
    +
    \sfR
    \widetilde{\sfY}_n^{1-\frac{1}{q}}
    \bigg]^{\frac{1}{p'}(1+\frac{p}N)}
   \\
    &\quad\le
    C \|u\|_{\infty, Q_o}^{\sfr-\sfq}2^{(n+1)(\sfr+p)(1+\frac{p}{N})}
    \sfk^{\sfm\frac{N+p}{N}}
    \widetilde{\sfY}_n^{1+ \frac{p}{N}\frac{1}{p'}(1-\frac{N+p}{qp})}
   \bigg[
   \frac{\widetilde{\sfY}_n^\frac1q}{(1-\tau)^p\theta}
    +\frac{
    \sfR}{\sfk^{p'+2}}
    \bigg]^{\frac{1}{p'}\frac{N+p}{N}}\\
     &\quad\le
    C \|u\|_{\infty, Q_o}^{\sfr-\sfq}\sfb^{n+1}
    \sfk^{\sfm\frac{N+p}{N}}
    \widetilde{\sfY}_n^{1+ \frac{p}{N}\frac{1}{p'}(1-\frac{N+p}{qp})}
   \bigg[
   \frac{\widetilde{\sfY}_o}{(1-\tau)^{qp}\theta^q}
    +\frac{
    \sfR^q}{\sfk^{q(p'+2)}}
    \bigg]^{\frac{p}{N}\frac{1}{p'}\frac{N+p}{qp}},
\end{align*}
where
\begin{equation*}
    \sfb:= 2^{(\sfr+p)(1+\frac{p}{N})}.
\end{equation*}
Dividing both sides by $\sfk^\sfr$ gives
\begin{align*}
 \widetilde{\sfY}_{n+1}
     &\le
    C \|u\|_{\infty, Q_o}^{\sfr-\sfq}\sfb^{n+1}
    \sfk^{\sfm\frac{N+p}{N}-\sfr}
    \widetilde{\sfY}_n^{1+ \frac{p}{N}\frac{1}{p'}(1-\frac{N+p}{qp})}\\
    &\phantom{\le\,\,\, }\cdot
   \bigg[
   \frac{\widetilde{\sfY}_o}{(1-\tau)^{qp}\theta^q}
    +\frac{\sfR^q
    }{\sfk^{q(p'+2)}}
    \bigg]^{\frac{p}{N}\frac{1}{p'}\frac{N+p}{qp}}.
\end{align*}
Passing to mean values in 
$\widetilde{\sfY}_n$, $\widetilde{\sfY}_{n+1}$ and $\widetilde{\sfY}_o$, we obtain the recursive inequality
\begin{align*}
     \sfY_{n+1}
     &\le
    C \|u\|_{\infty, Q_o}^{\sfr-\sfq}\sfb^{n+1}
    |Q_o|^{\frac{p}{N}\frac{1}{p'}}\sfk^{\sfm\frac{N+p}{N}-\sfr}
    \sfY_n^{1+ \frac{p}{N}\frac{1}{p'}(1-\frac{N+p}{qp})}\\
    &\phantom{\le\,\,\, }\cdot
   \bigg[
   \frac{\sfY_o}{(1-\tau)^{qp}\theta^q}
    +\frac{
    \sfR^q}{|Q_o|\sfk^{q(p'+2)}}
    \bigg]^{\frac{p}{N}\frac{1}{p'}\frac{N+p}{qp}},
\end{align*}
where
\begin{equation*}
    \sfY_n
    :=
    \frac{1}{\sfk^\sfr}\biint_{Q_n}(|u|-\sfk_n)_+^\sfr\,\dx\dt.
\end{equation*}
As in the case $\sfr<p'+2$, a judicious choice of the truncation level $\sfk$
allows us to streamline the right–hand side of the recursive inequality. The bracket splits into three contributions -- coming from $\sfY_o$, the $\sfF$–term, and the $\sff$–term:
\begin{equation*}
    \Big[\dots\Big]
    =
    \frac{\sfY_o}{(1-\tau)^{qp}\theta^q} 
    +
    \frac{\|u\|_{\infty, Q_o}^{qp'}
    \|\sfF\|_{qp, Q_o}^{qp}}{|Q_o|\sfk^{q(p'+2)}}
    +
    \frac{
     |Q_o|^\frac{1}{p}\|u\|_{\infty ,Q_o}^{q(p'+1)}
    \|\sff\|_{qp' ,Q_o}^q}{|Q_o|\sfk^{q(p'+2)}}
\end{equation*}
We now choose $\sfk$ sufficiently large to satisfy the three conditions in \eqref{k-conditions}.
Consequently, the bracket is controlled by 
\begin{equation*}
    \frac{\theta^{\frac{N+p}{p}-q}
   (1-\tau)^{N+p-qp}}{|Q_o|\sfk^{\frac{N}{p}(2-p)}}.
\end{equation*}
Using this control, the recursion simplifies to
\begin{align*}
     \sfY_{n+1}
     &\le
    C \|u\|_{\infty, Q_o}^{\sfr-\sfq}
    \bigg[
    \frac{|Q_o|}{(1-\tau)^{N+p}\theta^\frac{N+p}{p}}
    \bigg]^{\frac{p}{N}\frac{1}{p'}(1-\frac{N+p}{qp})}\\
     &\phantom{\le\,\,\, }\cdot
    \sfk^{\sfm\frac{N+p}{N}+\frac{N}{p}(p-2)\frac{p}{N}\frac{1}{p'}\frac{N+p}{qp}-\sfr}
    \sfb^{n+1}\sfY_n^{1+ \frac{p}{N}\frac{1}{p'}(1-\frac{N+p}{qp})}.
\end{align*}
Here, the exponent of 
$\sfk$ can be simplified. In fact, by \eqref{def:bq},
\begin{equation*}
    \mbox{exponent of $\sfk$}
    =
    \sfq -\sfr +
    \frac{N}{p}(2-p)\frac{p}{N}\frac{1}{p'}\Big(1-\frac{N+p}{qp}\Big).
\end{equation*}
Consequently,
\begin{align*}
     \sfY_{n+1}
     &\le
    C \sfb\|u\|_{\infty, Q_o}^{\sfr-\sfq}
    \bigg[
    \frac{|Q_o|\sfk^{\frac{N}{p}(2-p)}}{(1-\tau)^{N+p}\theta^\frac{N+p}{p}}
    \bigg]^{\frac{p}{N}\frac{1}{p'}(1-\frac{N+p}{qp})}
    \sfk^{\sfq -\sfr}\sfb^{n}\sfY_n^{1+ \frac{p}{N}\frac{1}{p'}(1-\frac{N+p}{qp})}.
\end{align*}
Before proceeding with the iteration, we revisit the constraints \eqref{k-conditions} imposed on $\sfk$. Solving them for 
$\sfk$, the first condition in \eqref{k-conditions} yields an analogue of \eqref{est:k-|u|^p'+2}, that is 
\begin{equation*}
    \sfk^\frac{\boldsymbol{\lambda}_\sfr}{p}
    \ge \frac{|B_1|}{(1-\tau)^{N+p}}\Big(
    \frac{\rho^p}{\theta} \Big)^\frac{N}{p}
    \biint_{Q_o}|u|^\sfr\,\dx\dt,
\end{equation*}
whereas the second and third conditions give exactly \eqref{est:k-F} and \eqref{est:k-f} respectively.
Set
\begin{equation*}
    \alpha:=
    \frac{p}{N}\frac{1}{p'}\Big(1-\frac{N+p}{qp}\Big),
    \quad
    {\sf C}
    :=
     C \sfb\|u\|_{\infty, Q_o}^{\sfr-\sfq}\sfk^{\sfq -\sfr}
    \bigg[
    \frac{|Q_o|\sfk^{\frac{N}{p}(2-p)}}{(1-\tau)^{N+p}\theta^\frac{N+p}{p}}
    \bigg]^{\alpha}
    .
\end{equation*}
Notice that $\alpha>0$ since $q>\frac{N+p}{p}$.
To apply Lemma \ref{it-lemma} with $(\sfY_n, \sfb, {\sf C},\alpha)$ we need
\begin{align*}
    \sfY_o
    &\le
    {\sf C}^{-\frac1\alpha}
    \sfb^{-\frac{1}{\alpha^2}}.
\end{align*}
Equivalently,
\begin{align*}
    \frac1{\sfk^{\sfr}}\biint_{Q_o} |u|^{\sfr}
    \,\dx\dt
    &\le
     \frac{\|u\|_{\infty, Q_o}^\frac{\bq-\sfr}{\alpha}}{(C\sfb)^{\frac{1}{\alpha}}
    \sfb^{\frac{1}{\alpha^2}}}
    \frac{(1-\tau)^{N+p}\theta^\frac{N+p}{p}}{
    |Q_o|}\sfk^{\frac{N}{p}(p-2) + \frac{\sfr-\sfq}{\alpha}}.
\end{align*}
Since $\|u\|_{\infty,Q_o}>0$ (otherwise the claim is trivial), we may solve the preceding condition for $\sfk$. This gives another condition on $\sfk$, that is
\begin{equation*}
    \sfk^{\frac{\boldsymbol\lambda_\sfr}{p}+\frac{\sfr-\sfq}{\alpha}}
    % =
    % \sfk^{\frac{N}{p}(p-2)+r+\frac{r-\sfq}{\alpha}}
    \ge
    \frac{\big(C\sfb^{1+\frac{1}{\alpha}}\big)^{\frac{1}{\alpha}}
    |Q_o|\|u\|_{\infty,Q_o}^{\frac{\sfr-\sfq}{\alpha}}}
    {(1-\tau)^{N+p}\,\theta^{\frac{N+p}{p}}}
    \biint_{Q_o}|u|^{\sfr}\,\dx\dt,
\end{equation*}
or equivalently,
\begin{equation*}
    \sfk^{\frac{\boldsymbol\lambda_\sfr}{p}+\frac{\sfr-\sfq}{\alpha}}
    \ge
    \frac{C|B_1|\|u\|_{\infty,Q_o}^{\frac{\sfr-\sfq}{\alpha}}}{(1-\tau)^{N+p}}
    \Big(\frac{\varrho^{p}}{\theta}\Big)^{\frac{N}{p}}
    \biint_{Q_o}|u|^{\sfr}\,\dx\dt.
\end{equation*}
By Lemma \ref{it-lemma} we have $\sfY_n\to 0$ as $n\to\infty$. In particular,
$$
    \sup_{Q_{\tau\varrho,\tau\theta}} |u|\le\sfk,
$$
provided $\sfk$ is chosen large as required above. Hence we fix $\sfk$ by
\begin{align*}
    \sfk
    &=\max
    \Bigg\{
    \bigg[
    \frac{|B_1|}{(1-\tau)^{N+p}}
     \Big(\frac{\varrho^p}{\theta}\Big)^\frac{N}{p}
    \biint_{Q_{o}} |u|^{\sfr}
    \,\dx\dt\bigg]^{\frac{p}{\boldsymbol\lambda_\sfr}},
    \\   
    &\qquad\qquad\,\,\, 
    \|u\|_{\infty, Q_o}^{1-\frac{\frac{\boldsymbol \lambda_\sfr}{p}}{\frac{\boldsymbol \lambda_\sfr}{p}+\frac{\sfr-\sfq}{\alpha}}}
    \bigg[
    \frac{C|B_1|}{(1-\tau)^{N+p}}\Big(\frac{\varrho^p}{\theta}\Big)^\frac{N}{p}
    \biint_{Q_o} |u|^{\sfr}
    \,\dx\dt
    \bigg]^{\frac{1}{\frac{\boldsymbol \lambda_\sfr}{p}+\frac{\sfr-\sfq}{\alpha}}},\\
    &  \qquad\qquad\,\,\, 
    \|u\|_{\infty, Q_o}^{1- \frac{\boldsymbol\lambda_{2q}}{\boldsymbol\lambda_{q(p'+2)}}}
      \bigg[
       |B_1|
    \Big(\frac{\theta}{\varrho^p}\Big)^{q-\frac{N}{p}}
    \biint_{Q_o} |\varrho\, \sfF|^{qp}\,\dx\dt
    \bigg]^\frac{p}{\boldsymbol\lambda_{q(p'+2)}}
    , \\
    &  \qquad\qquad\,\,\,
    \|u\|_{\infty, Q_o}^{1- \frac{\boldsymbol\lambda_{q}}{\boldsymbol\lambda_{q(p'+2)}}}
      \bigg[
       |B_1|^{p'}
    \Big(\frac{\theta}{\varrho^p}\Big)^{p'(q-\frac{N}{p})}
    \biint_{Q_o} |\varrho^p\sff|^{qp'}\,\dx\dt
    \bigg]^\frac{p}{p'\boldsymbol\lambda_{q(p'+2)}}
    , \\
    & \qquad\qquad\,\,\,
    \Big(\frac{\theta}{\varrho^p}\Big)^\frac{1}{2-p}
    ,\mu\rho\Bigg\}.
\end{align*}

The bound can be refined via interpolation and the argument is essentially identical to the case $2<\sfr<p'+2$. 
Indeed, apply Young’s inequality with conjugate exponents
$$
    s=
    \frac{\frac{\boldsymbol\lambda_\sfr}{p}+\frac{\sfr-\sfq}{\alpha}}{\frac{\sfr-\sfq}{\alpha}},
    \qquad
    s'=\frac{\frac{\boldsymbol\lambda_\sfr}{p}+\frac{\sfr-\sfq}{\alpha}}{\frac{\boldsymbol\lambda_\sfr}{p}},
$$
to the second entry. This yields
$$
    \big[\text{second entry}\big]
    \le
    \tfrac16\|u\|_{\infty,Q_o}
    +
    C\left[
    \frac{|B_1|}{(1-\tau)^{N+p}}
    \Big(\frac{\varrho^p}{\theta}\Big)^{\frac{N}{p}}
    \biint_{Q_o}|u|^{\sfr}\,\dx\dt
    \right]^{\frac{p}{\boldsymbol\lambda_\sfr}},
$$
The estimates of the third and forth entries are identical to \eqref{2nd-entry} and \eqref{3rd-entry} respectively.
Hence, for any $\tau\in [\frac12, 1)$ there holds
\begin{align*}
    \sup_{Q_{\tau\varrho,
    \tau\theta}} |u|
    &\le
    \tfrac12 \sup_{Q_{\varrho,
    \theta}} |u|
    +
    \frac{C}{(1-\tau)^{\frac{p}{\boldsymbol\lambda_r}(N+p)}}
    \left[
    \Big(\frac{\varrho^p}{\theta}\Big)^{\frac{N}{p}}
    \biint_{Q_o}|u|^{\sfr}\,\dx\dt
    \right]^{\frac{p}{\boldsymbol\lambda_\sfr}}\\
    &\phantom{\le\,}
    +C\Bigg[
    \Big(\frac{\theta}{\varrho^p}\Big)^{q-\frac{N}{p}}
    \biint_{Q_o} |\varrho\, \sfF|^{qp}\,\dx\dt
    \Bigg]^\frac{p}{\boldsymbol\lambda_{2q}}\\
    &\phantom{\le\,}
    +C\left[
    \Big(\frac{\theta}{\varrho^p}\Big)^{p'(q-\frac{N}{p})}
    \biint_{Q_o}|\varrho^{p}\sff|^{q p'}\,\dx\dt
    \right]^{\frac{p}{p'\boldsymbol\lambda_q}}
     +C\Big(\frac{\theta}{\varrho^p}\Big)^\frac{1}{2-p}
    +C\mu\rho,
\end{align*}
where $\boldsymbol\lambda_s=N(p-2)+ps$ and $C=C(N,p,C_o,C_1,q,\sfr)$. The remainder of the argument is identical to the case $2<\sfr<p'+2$ (up to harmless changes of constants). In particular, repeating the interpolation yields the same quantitative local $L^\infty$-bound, so the proof in the regime $\sfr\ge p'+2$ is complete.
%\end{proof}

\begin{remark}[Quantative $L^\infty$-estimate: the case $\frac{2N}{N+2}<p\le 2$]\upshape
In our notion of weak solution we assume
$$
u\in C^0\big([0,T];L^2(E,\mathbb{R}^k)\big)\cap L^p\big(0,T;W^{1,p}(E,\mathbb{R}^k)\big).
$$
By the parabolic Sobolev embedding this yields
$$
u\in L^\sfr_{\mathrm{loc}}(E_T,\mathbb{R}^k)\quad\text{with}\quad \sfr=\frac{p(N+2)}{N}>2.
$$
Moreover,
$$
\boldsymbol{\lambda}_2:=N(p-2)+2p=(N+2)p-2N>0.
$$
Hence the structural condition $\boldsymbol\lambda_\sfr>0$ is met already with $\sfr=2$, and the quantitative scheme developed for the sub-critical range $1<p\le \tfrac{2N}{N+2}$ (reverse Hölder + interpolation/iteration) applies verbatim. In particular, for any %intrinsic 
cylinder $Q_o:=Q_{\rho,\theta}\Subset E_T$, and for any $\tau\in[\frac12,1)$ we obtain
\begin{align*}
    \sup_{Q_{\tau\varrho,
    \tau\theta}} |u|
    &\le
    C
    \left[
    \frac{1}{(1-\tau)^{N+p}}
    \Big(\frac{\varrho^p}{\theta}\Big)^{\frac{N}{p}}
    \biint_{Q_o}|u|^{2}\,\dx\dt
    \right]^{\frac{p}{(N+2)p-2N}}\\
    &\phantom{\le\,}
    +C\left[
    \Big(\frac{\theta}{\varrho^p}\Big)^{q-\frac{N}{p}}
    \biint_{Q_o} |\varrho\, \sfF|^{qp}\,\dx\dt
    \right]^\frac{p}{\boldsymbol\lambda_{2q}}\\
    &\phantom{\le\,}
    +C\left[
    \Big(\frac{\theta}{\varrho^p}\Big)^{p'(q-\frac{N}{p})}
    \biint_{Q_o}|\varrho^{p}\sff|^{q p'}\,\dx\dt
    \right]^{\frac{p}{p'\boldsymbol\lambda_q}}
     +C\Big(\frac{\theta}{\varrho^p}\Big)^\frac{1}{2-p}
    +C\mu\rho,
\end{align*}
where $q>\tfrac{N+p}{p}$ as in the assumptions. 
The dependence of the constant is of the same form as in the sub-critical case, namely $C=C(N,p,C_o,C_1,q)$.
\end{remark}

\section{Reverse Hölder-type inequality}\label{S:4}
From here on, i.e. both in \S~\ref{S:4} and in \S~\ref{S:5}, we assume $\mu=0$ in the structure conditions \eqref{growth-a*}. 
%Moreover, we let $Q^*_{\rho,\theta}(z_o):=B_\rho(x_o)\times(t_o-\theta,t_o+\theta)$.
For $\lambda>0$ and $z_o=(x_o,t_o)$, recall that we set
\begin{equation}\label{def-Q}
	Q_\rho^{(\lambda)}(z_o)
	:=
	B_\rho^{(\lambda)}(x_o)\times\Lambda_\rho(t_o),
\end{equation}
where
\begin{equation*}
    B_\rho^{(\lambda)}(x_o)
    :=
    \big\{x\in\mathbb{R}^N\colon\ |x-x_o|<\lambda^{\frac{p-2}{2}}\rho\big\},
    \quad
    \Lambda_\rho(t_o)
    :=
    (t_o-\rho^2,t_o+\rho^2).
\end{equation*}
The parameter $\lm$ encodes the dimension of solutions and will restore the homogeneity of forthcoming estimates.
\subsection{Energy inequality}\label{sec:energy}

The proof of the following energy inequality is standard and can for instance be inferred from \cite[Lemma~3.2]{Kinnunen-Lewis:1}.

\begin{lemma}\label{lm:energy-without-zeta}
Let $p>1$. Then, there exists a constant $C=C(N, p,C_o,C_1)$ such that whenever $u$ is  a weak solution to \eqref{eq-par-gen} in $E_T$ in the sense of Definition \ref{def:weak_solution} with \eqref{growth-a*}, $Q_\rho^{(\lambda)}(z_o)\subseteq E_T$, $r\in[\tfrac12\rho, \rho)$,  and $a\in\R^k$,
then the following energy estimate holds
\begin{align}\label{est:energy*}\nonumber
    \sup_{t\in \Lambda_r(t_o)}&
    \mint_{B^{(\lm)}_{r}(x_o)} 
     \frac{| u(t)-a|^2}{\rho^2}
    \,\dx
    +
    \biint_{Q_r^{(\lambda)}(z_o)}|Du|^p\,\dx\dt\\
    &
    \le
    C
    \biint_{Q_\rho^{(\lambda)}(z_o)}
    \bigg[ \frac{|u-a|^p}{\lm^{\frac{p-2}{2}p}(\rho-r)^p}
    + 
    \frac{|u-a|^2}{\rho^2-r^2}
    +
    |\sfF|^p+|\lm^{\frac{p-2}{2}}\rho\,\sff|^{p'}\bigg]\,\dx\dt.
\end{align}
\end{lemma}

\subsection{Gluing lemma}
We record a radius–selection estimate controlling the temporal oscillation of spatial means. Writing the time slice-mean
$$\langle u\rangle^{(\lm)}_{x_o;r}(t):=\mint_{B^{(\lm)}_r(x_o)}u(x,t)\,\dx,
$$
the lemma shows that for a suitable $\hat\rho\in[\frac12 \rho,\rho)$ the variation of $\langle u\rangle^{(\lm)}_{x_o;\hat\rho}(t)$ is bounded by the space–time integral of the natural flux terms. 

\begin{lemma}\label{lem:glue}
Let $p>1$. There exists a constant $C=C(N,C_1)>0$ with the following property. If $u$ is a weak solution to 
\eqref{eq-par-gen} in $E_T$ in the sense of Definition~\ref{def:weak_solution}, satisfying \eqref{growth-a*}, and if $Q_\rho^{(\lambda)}(z_o) \subseteq E_T$, then there exists $\hat\rho\in[\tfrac{\rho}{2},\rho)$ such that for all $t_1,t_2\in\Lambda_\rho(t_o)$,
$$
	\big|\langle u\rangle^{(\lm)}_{x_o;\hat\rho}(t_2)-
	\langle u\rangle^{(\lm)}_{x_o;\hat\rho}(t_1)\big|
	\le
	C\lm^{\frac{2-p}{2}}\rho\,
	\biint_{Q_\rho^{(\lambda)}(z_o)}\!\Big[|Du|^{p-1}+|\sfF|^{p-1}+|\lambda^{\frac{p-2}{2}}\rho\,\sff|\Big]\,\dx\dt.
$$
\end{lemma}

\begin{proof}
It is a straightforward adaptation of the argument used in \cite[Lemma~3.1]{Kinnunen-Lewis:1}. We refrain from going into further details.
\end{proof}

\subsection{Parabolic Sobolev–Poincaré-type inequalities}%\label{sec:HI-poin}
A parabolic Poincaré-type inequality on space–time cylinders is recorded next. Combining the slice-wise Poincaré estimate with the radius–selection bound from Lemma~\ref{lem:glue} controls the deviation from the space–time mean by the spatial gradient and 
%with the factor $\theta/\rho^{2}$, 
by the temporal oscillation encoded in the flux terms. 

\begin{lemma}\label{lem:poin}
Let $p>1$. There exists a constant $C=C(N,p,C_1)>0$ with the following property. If $u$ is a weak solution to \eqref{eq-par-gen} in $E_T$ in the sense of Definition \ref{def:weak_solution}, satisfying \eqref{growth-a*}, and if $Q_\rho^{(\lambda)}\equiv Q_\rho^{(\lambda)}(z_o)\subseteq E_T$, then for any $\sigma\in [1,p]$ one has the Poincar\'e-type inequality
\begin{align*}
	&\biint_{Q_\rho^{(\lambda)}} 
	\frac{\big|u-(u)_{z_o;\rho}^{(\lambda)}\big|^{\sigma}}{(\lm^{\frac{p-2}{2}}\rho)^\sigma}\, \dx\dt\\  
    &\quad\le
	C\biint_{Q_\rho^{(\lambda)}}
	|Du|^\sigma\,\dx\dt
     +C\bigg[\lm^{2-p}
	\biint_{Q_\rho^{(\lambda)}} \Big[|Du|^{p-1} + |\sfF|^{p-1}+|\lambda^{\frac{p-2}{2}}\rho\,\sff|\Big]\dx\dt
	\bigg]^\sigma.
\end{align*}
\end{lemma}

\begin{proof}
Throughout the proof we suppress the center $z_o$ from the notation. Let $\hat\rho\in[\frac12 \rho,\rho)$ be the radius provided by Lemma~\ref{lem:glue}; by the quasi–minimality Lemma~\ref{lem:mean}, we get
\begin{align*}
     \biint_{Q_\rho^{(\lambda)}}&
     \frac{\big|u-(u)^{(\lm)}_{\rho}\big|^\sigma }{(\lm^{\frac{p-2}{2}}\rho)^\sigma }\,\dx\dt
     \le
     2^\sigma\biint_{Q_\rho^{(\lambda)}}
     \frac{\big|u-
     (u)^{(\lm)}_{\hat\rho}
     \big|^\sigma }
     {(\lm^{\frac{p-2}{2}}\rho)^{\sigma }}\,\dx\dt 
     \le 
     8[\sfI+\sfI\sfI],
\end{align*}
where
\begin{equation*}
	\sfI
    :=
    \biint_{Q_\rho^{(\lambda)}}
	\frac{\big|u-\langle u\rangle^{(\lm)}_{\hat\rho}(t)\big|^\sigma }
    {\rho^{\sigma }}\,\dx\dt,
    \qquad
    \sfI\sfI
    :=
    \bint_{\Lambda_{\rho}}
    \frac{\big|\langle u\rangle^{(\lm)}_{\hat\rho}(t)-
    (u)^{(\lm)}_{\hat\rho}\big|^\sigma }
    {\rho^{\sigma }}\,\dt,
\end{equation*}
and moreover, 
%$\Lambda_{\theta}$ as defined above,          
\begin{equation*}
    \langle u\rangle^{(\lm)}_{\hat\rho}(t):=\mint_{B^{(\lm)}_{\hat\rho}}u(x,t)\,
    \dx
    \quad\mbox{and}\quad
    (u)^{(\lm)}_{\hat\rho}
    :=\biint_{Q_{\hat\rho}^{(\lambda)}}u\,
    \dx\dt.
\end{equation*}
Next, we estimate $\sfI$ and $\sfI\sfI$. Since $\hat\rho\in[\frac12\rho,\rho)$, Lemma~\ref{lem:mean} and, subsequently, Poincaré’s inequality on the ball $B_\rho^{(\lm)}$ (applied slice-wise) yield
\begin{align*}
	\sfI
	\le
	C\biint_{Q_\rho^{(\lambda)}}
	\frac{|u-\langle u\rangle^{(\lm)}_{\rho}(t)|^\sigma }{(\lm^{\frac{p-2}{2}}\rho)^{\sigma }}\,\dx\dt 
	\le
	C\biint_{Q_\rho^{(\lambda)}}
	|Du|^\sigma \,\dx\dt  ,
\end{align*}
with $C=C(N,p)$. The dependence of the constant in Poincaré’s inequality on $\sigma\in[1,p]$ is continuous and can be absorbed into $C$.
As for $\sfI\sfI$, by Lemma~\ref{lem:glue}, we have
\begin{align*}
    \sfI\sfI
    &\le
    \frac{C}{(\lm^{\frac{p-2}{2}}\rho)^{\sigma }} 
    \sup_{t,\tau\in \Lambda_{\rho}} 
    \big|\langle u\rangle^{(\lm)}_{\hat\rho}(t)-
    \langle u\rangle^{(\lm)}_{\hat\rho}(\tau)\big|^\sigma  \\
    &\le
    C\bigg[\lm^{2-p}
	\biint_{Q_\rho^{(\lambda)}} \Big[|Du|^{p-1} + |\sfF|^{p-1}+|\lambda^{\frac{p-2}{2}}\rho\,\sff|\Big]\dx\dt
	\bigg]^\sigma   ,
\end{align*}
with $C=C(N,C_1)$. Combining the estimates for $\sfI$ and $\sfI\sfI$ yields the claim.
\end{proof}

The next lemma records a Poincaré-type inequality on intrinsic cylinders. 
%The spatial radius $\lambda^{\frac{p-2}{2}}\rho$ reflects the degeneracy parameter $\lambda$. 
The estimate follows from Lemma~\ref{lem:poin}  together with Hölder’s and Young’s inequalities. The parameter $\varepsilon$ balances the mean-oscillation term against the energy and source terms. 

\begin{corollary}\label{cor:poin}
Let $1<p<2$. There exists a constant $C=C(N,p,C_1)>0$ with the following property. If $u$ is a weak solution to \eqref{eq-par-gen} in $E_T$ in the sense of Definition~\ref{def:weak_solution}, satisfying \eqref{growth-a*}, and if $Q_{2\rho}^{(\lambda)}(z_o)\subseteq E_T$, then for any $\sigma\in[p-1,p]$ and any $\varepsilon\in(0,1]$,
\begin{align*}
	\biint_{Q_{\rho}^{(\lambda)}(z_o)}&
	\frac{\big|u-(u)_{z_o;\rho}^{(\lambda)}\big|^{\sigma}}{(\lambda^{\frac{p-2}{2}}\rho)^\sigma}\, \dx\dt \\
	&\le
	\epsilon\lambda^\sigma +
    \frac{C}{\epsilon^{\frac{2-p}{p-1}}} \biint_{Q_{\rho}^{(\lambda)}(z_o)} \Big[|Du|^{\sigma} + |\sfF|^{\sigma}+|\lm^{\frac{p-2}2}\rho\,\sff|^\frac{\sigma}{p-1}\Big]\,\dx\dt,
\end{align*}
where we recall that
\begin{equation*}
    (u)_{z_o;\rho}^{(\lambda)}
    :=\biint_{Q_{\rho}^{(\lambda)}(z_o)}u\,\dx\dt.
\end{equation*}
\end{corollary}

\begin{proof}
We suppress $z_o$ in the notation; apply Lemma~\ref{lem:poin}
and Hölder's inequality to get 
\begin{align*}
	&\biint_{Q_{\rho}^{(\lambda)}} 
	\frac{\big|u-(u)_{\rho}^{(\lambda)}\big|^{\sigma}}{(\lambda^{\frac{p-2}{2}}\rho)^\sigma}\, \dx\dt \\
	&\,\,\le
	C\biint_{Q_{\rho}^{(\lambda)}}
	|Du|^\sigma\,\dx\dt +
    C\lambda^{\sigma(2-p)}
    \bigg[
	\biint_{Q_{\rho}^{(\lambda)}} \Big[|Du|^{p-1} + |\sfF|^{p-1} +|\lambda^{\frac{p-2}{2}}\rho\,\sff|\Big]\dx\dt
	\bigg]^\sigma \\
	&\,\,\le
	C\biint_{Q_{\rho}^{(\lambda)}}
	|Du|^\sigma\,\dx\dt +
    C\lambda^{\sigma(2-p)}\bigg[
	\biint_{Q_{\rho}^{(\lambda)}} \Big[|Du|^{\sigma} + |\sfF|^{\sigma}
    +|\lambda^{\frac{p-2}{2}}\rho\,\sff|^\frac{\sigma}{p-1}
    \Big]\dx\dt
	\bigg]^{p-1}.
\end{align*}
%In the case $p=2$ this is precisely the desired bound. If $1<p<2$, 
Apply Young’s inequality with exponents $\frac{1}{2-p}$ and $\frac{1}{p-1}$ to the second term on the right-hand side and conclude the proof.
\end{proof}

\subsection{Estimates on (sub-)intrinsic cylinders}\label{sec:est-intrinsic-cyl}
Fix $\sfr>2$ such that~\eqref{Eq:lm>0} holds. 
Throughout this subsection, a cylinder $Q_{2\rho}^{(\lambda)}(z_o)\subseteq E_T$ is said to satisfy a \emph{sub-intrinsic coupling} if, for some $\sfK\ge1$,
\begin{align}\label{sub-intr}
	\biint_{Q_{2\rho}^{(\lambda)}(z_o)}
	\frac{\big|u-(u)_{z_o;2\rho}^{(\lambda)}\big|^\sfr}{(2\rho)^\sfr} \,\dx\dt 
    \le
    \sfK\lambda^{\frac{p\sfr}2}.
\end{align}
Likewise, $Q_{2\rho}^{(\lambda)}(z_o)$ is called \emph{super-intrinsic} if, for some $\sfK\ge1$,
\begin{align}\label{super-intr}
	\lambda^{\frac{p\sfr}2}
	\le
	\sfK\biint_{Q_\rho^{(\lambda)}(z_o)}
	\frac{\big|u-(u)_{z_o;\rho}^{(\lambda)}\big|^\sfr}{\rho^\sfr} \,\dx\dt .
\end{align}
The next result provides a quantitative $L^\infty$-estimate on these cylinders.

\begin{corollary}\label{cor:sup-qualitativ}
Under the assumptions of Theorem~\ref{thm:sup-quant}, for any $\tau\in(0,1)$ and any cylinder $Q_{\rho}^{(\lambda)}(z_o)\subseteq E_T$,
\begin{align*}
    \sup_{Q_{\tau\rho}^{(\lambda)}(z_o)} |u|
    &\le
    C\lambda^{\frac{p-2}{2}}\rho
    \Bigg[
    \bigg[
    \frac{\lambda^{(p-2)\frac{N}{p}}}{(1-\tau)^{N+p}}
    \biint_{Q_{\rho}^{(\lambda)}(z_o)} 
    \frac{|u|^\sfr}{(\lambda^\frac{p-2}2\rho)^\sfr} \,\dx\dt
    \bigg]^\frac{p}{\boldsymbol\lambda_\sfr}\\
     &\qquad\qquad+
    \bigg[ 
    \biint_{Q_{\rho}^{(\lambda)}(z_o)} 
    \Big[| \sfF|^{qp} + \big|\lambda^{\frac{p-2}{2}}\rho\, \sff\big|^{qp'}\Big] \,\dx\dt
    \bigg]^\frac{1}{pq} +
    \lambda %+\mu 
    \Bigg],
\end{align*}
where $\boldsymbol{\lambda}_{\sfr}=N(p-2)+p\sfr$ and $C$ depends only on the structural data.
\end{corollary}

\begin{proof} 
Set $z_o=0$; apply Theorem~\ref{thm:sup-quant} with $(\rho,\theta)$ replaced by $\big(\lambda^{\frac{p-2}{2}}\rho,\ \rho^{2}\big)$ to obtain
\begin{align*}
    \sup_{Q_{\tau\rho}^{(\lambda)}} |u|
    &\le
    C
    \bigg[
    \frac{1}{(1-\tau)^{N+p}}
     \bigg(\frac{(\lambda^{\frac{p-2}{2}}\varrho)^p}{\rho^2}\bigg)^\frac{N}{p}
    \biint_{Q_{\rho}^{(\lambda)}} 
    |u|^\sfr \,\dx\dt
    \bigg]^\frac{p}{\boldsymbol\lambda_\sfr}\\
    &\phantom{\le\,} +
    C \bigg[
    \bigg(\frac{\rho^2}{(\lambda^{\frac{p-2}{2}}\varrho)^p}\bigg)^{q-\frac{N}{p}}
    \biint_{Q_{\rho}^{(\lambda)}} |\lambda^{\frac{p-2}{2}}\varrho \,\sfF|^{qp}\,\dx\dt
    \bigg]^\frac{p}{\boldsymbol\lambda_{2q}} \\
    &\phantom{\le\,}+
    C \bigg[
    \bigg(\frac{\rho^2}{(\lambda^{\frac{p-2}{2}}\varrho)^p}\bigg)^{p'(q-\frac{N}{p})}
    \biint_{Q_{\rho}^{(\lambda)}}|(\lambda^{\frac{p-2}{2}}\rho)^p \sff|^{qp'}\,\dx\dt
    \bigg]^\frac{p}{p'\boldsymbol\lambda_{q}}\\
    &
    \phantom{\le\,}
    +
    C\bigg[\frac{\rho^2}{(\lambda^{\frac{p-2}{2}}\varrho)^p}\bigg]^\frac{1}{2-p} \\
    % +
    % C\mu\lambda^{\frac{p-2}{2}}\varrho \\
    &=
    C\lambda^{\frac{p-2}{2}}\rho
    \Bigg[ 
    \bigg[
    \frac{\lambda^{(p-2)\frac{N}{p}}}{(1-\tau)^{N+p}}
    \biint_{Q_{\rho}^{(\lambda)}} 
    \frac{|u|^\sfr}{(\lambda^{\frac{p-2}{2}}\rho)^\sfr} \,\dx\dt
    \bigg]^\frac{p}{\boldsymbol\lambda_\sfr}\\
     &\qquad\qquad\qquad +
    \lambda^{1-pq\frac{p}{\boldsymbol\lambda_{2q}}}
    \bigg[ 
    \biint_{Q_{\rho}^{(\lambda)}} | \sfF|^{qp}\,\dx\dt
    \bigg]^\frac{p}{\boldsymbol\lambda_{2q}} \\
    &\qquad\qquad\qquad +
    \lambda^{1-pq\frac{p}{p'\boldsymbol\lambda_{q}}}
    \bigg[
    \biint_{Q_{\rho}^{(\lambda)}(z_o)}|\lambda^{\frac{p-2}{2}}\rho \,\sff|^{qp'}\,\dx\dt
    \bigg]^\frac{p}{p'\boldsymbol\lambda_{q}}
    +
    \lambda
    % +\mu
    \Bigg] .
\end{align*}  
Applying Young’s inequality to the second term with conjugate exponents $\big(1-pq\,\tfrac{p}{\boldsymbol{\lambda}_{2q}}\big)^{-1}$ and $\big(pq\,\tfrac{p}{\boldsymbol{\lambda}_{2q}}\big)^{-1}$, and to the third term with $\big(1-pq\,\tfrac{p}{p'\boldsymbol{\lambda}_{q}}\big)^{-1}$ and $\big(pq\,\tfrac{p}{p'\boldsymbol{\lambda}_{q}}\big)^{-1}$, one arrives at the desired estimate.
\end{proof}

The $u$-dependence of $\sfA$ is not allowed in the following result as we have to guarantee that $u+const.$ is also a solution.

\begin{lemma}[Intrinsic $L^\infty$
estimate on sub-intrinsic cylinders]\label{lem:intr-sup}
Let $p\in\bigl(1,\tfrac{2N}{N+2}\bigr]$ and $\sfK\ge1$. If $u$ is a weak solution to \eqref{eq-par-gen} in $E_T$ in the sense of Definition~\ref{def:weak_solution}, satisfying \eqref{growth-a*}, and if $Q_{2\rho}^{(\lambda)}(z_o)\subseteq E_T$ is a parabolic cylinder fulfilling \eqref{sub-intr},
then %for every $\tau\in[1,2)$,
\begin{align*}
    \sup_{Q_{\rho}^{(\lambda)}(z_o)}&
    \frac{\big|u-(u)_{z_o;2\rho}^{(\lambda)}\big|}{\lambda^{\frac{p-2}{2}}\rho}
    \le
    %\frac{C\,\lambda}{(2-\tau)^{\frac{p(N+p)}{\boldsymbol{\lambda}_\sfr}}}
    C\,\lm
    +
    C\bigg[
    \biint_{Q_{2\rho}^{(\lambda)}(z_o)}
\Big[|\sfF|^{qp}+|\lambda^{\frac{p-2}{2}}\rho\,\sff|^{qp'}\Big]\,\dx\dt
\bigg]^{\frac{1}{pq}},
\end{align*}
with a constant $C=C(N,p,C_o,C_1,q,\sfr,\sfK)$,
where
\begin{equation*}
    (u)_{z_o;2\rho}^{(\lambda)}
    :=\biint_{Q_{2\rho}^{(\lambda)}(z_o)}u\, \dx\dt.
\end{equation*}
\end{lemma}

\begin{proof}
Omit $z_o$ from the notation. 
First observe that $u-(u)_{2\rho}^{(\lambda)}$ is again a weak solution to \eqref{eq-par-gen}. Applying Corollary~\ref{cor:sup-qualitativ} to $u-(u)_{2\rho}^{(\lambda)}$ on $Q_{2\rho}^{(\lambda)}$ and taking $\tau=1/2$ yield
\begin{align*}
    \sup_{Q_{\rho}^{(\lambda)}} 
    \frac{\big|u-(u)_{2\rho}^{(\lambda)}\big|}{\lambda^{\frac{p-2}{2}}\rho}
    &\le
    C
    \Bigg[
    \lambda^{(p-2)\frac{N}{p}}
    \biint_{Q_{2\rho}^{(\lambda)}} 
    \frac{\big|u-(u)_{2\rho}^{(\lambda)}\big|^\sfr}{(2\lambda^{\frac{p-2}{2}}\rho)^\sfr} \,\dx\dt
    \Bigg]^\frac{p}{\boldsymbol\lambda_\sfr}\\
     &
    \phantom{\le\,}
    +C  
    \bigg[ 
    \biint_{Q_{2\rho}^{(\lambda)}} 
    \Big[| \sfF|^{qp} + \big|\lambda^{\frac{p-2}{2}}\rho\, \sff\big|^{qp'}\Big] \,\dx\dt
    \bigg]^\frac{1}{pq} + C\lambda,
\end{align*}
where $\boldsymbol\lambda_\sfr=N(p-2)+p\sfr$. 
In view of the sub-intrinsic condition \eqref{sub-intr}, we have
\begin{align*}
    \lambda^{(p-2)\frac{N}{p}}
    \biint_{Q_{2\rho}^{(\lambda)}} 
    \frac{\big|u-(u)_{2\rho}^{(\lambda)}\big|^\sfr}{(2\lambda^{\frac{p-2}{2}}\rho)^\sfr} \,\dx\dt
    \le 
    \sfK\lambda^{(p-2)\frac{N}{p}-\frac{p-2}2\sfr}
    \lambda^{\frac{p\sfr}2}
    =
    \sfK \lambda^\frac{\boldsymbol\lambda_\sfr}{p}.
\end{align*}
Consequently,
\begin{equation*}
    \Bigg[
    \lambda^{(p-2)\frac{N}{p}}
    \iint_{Q_{2\rho}^{(\lambda)}}
    \frac{\big|u-(u)_{2\rho}^{(\lambda)}\big|^{\sfr}}{(2\lambda^{\frac{p-2}{2}}\rho)^{\sfr}}\,\dx\dt
    \Bigg]^{\frac{p}{\boldsymbol{\lambda}_\sfr}}
    \le
    \sfK^\frac{p}{\boldsymbol{\lambda}_\sfr}\lambda,
\end{equation*}
so that the conclusion follows.
\end{proof}

Assuming both the sub-intrinsic coupling \eqref{sub-intr} and the super-intrinsic bound \eqref{super-intr}, we arrive at the following estimate.
\begin{lemma}[intrinsic scale bound]\label{lem:sub-intr-2}
Let $p\in\bigl(1,\tfrac{2N}{N+2}\bigr]$ and $\sfK\ge1$. There exists a constant
$C=C(N,p,C_o,C_1,q,\sfr,\sfK)>0$ with the following property. If $u$ is a weak
solution to \eqref{eq-par-gen} in $E_T$ in the sense of
Definition~\ref{def:weak_solution}, satisfying \eqref{growth-a*}, and if
$Q_{2\rho}^{(\lambda)}(z_o)\subseteq E_T$ is a parabolic cylinder fulfilling
\eqref{sub-intr} and \eqref{super-intr}, then
\begin{align*}
	\lambda^p
    \le
	C \biint_{Q_{\rho}^{(\lambda)}(z_o)} |Du|^p \,\dx\dt +
    C \bigg[\biint_{Q_{2\rho}^{(\lambda)}(z_o)} \Big[| \sfF|^{qp} + \big|\lambda^{\frac{p-2}{2}}\rho\, \sff\big|^{qp'}\Big]\,\dx\dt
	\bigg]^{\frac{1}{q}}.
\end{align*}
\end{lemma}
The cylinders $Q_\rho^{(\lambda)}$ are defined in the intrinsic geometry of the parabolic $p$-Laplacian, with spatial radius $\lambda^{\frac{p-2}{2}}\rho$ and time height $\rho^2$; the parameter $\lambda$ encodes the local scale of oscillation/degeneracy. The sub-intrinsic and super-intrinsic conditions calibrate this scale by tying $\lambda$ to the $L^\sfr$-oscillation of $u$ on $Q_{2\rho}^{(\lambda)}$ and $Q_\rho^{(\lambda)}$. Then, the lemma shows that this intrinsic scale is bounded -- quantitatively controlled -- by the local energy and the source terms.
Thus, the result is a scale bound for $\lambda$: it prevents an uncalibrated choice of intrinsic cylinders, converts geometric coupling into energetic control, and furnishes the normalization needed for subsequent local estimates.

\begin{proof}
Throughout the proof, the reference point $z_o$ is suppressed.  From the super-intrinsic property \eqref{super-intr} it follows that
\begin{align*}
	\lambda^\sfr
	&\le
	\sfK\biint_{Q_\rho^{(\lambda)}}
	\frac{\big|u-(u)_{\rho}^{(\lambda)}\big|^\sfr}{(\lambda^{\frac{p-2}{2}}\rho)^\sfr} \,\dx\dt \\
    &\le
    \sfK
    \Bigg[\sup_{Q_\rho^{(\lambda)}
    } \frac{\big|u-(u)_{\rho}^{(\lambda)}\big|}{(\lambda^{\frac{p-2}{2}}\rho)}\Bigg]^{\sfr-p} 
    \biint_{Q_\rho^{(\lambda)}}
	\frac{\big|u-(u)_{\rho}^{(\lambda)}\big|^p}{(\lambda^{\frac{p-2}{2}}\rho)^p} \,\dx\dt.
\end{align*}
To estimate the sup-term, observe that
%In the supremum term the mean over $Q_{\rho}^{(\lambda)}$ may be replaced by the mean over $Q_{2\rho}^{(\lambda)}$. Indeed,
\begin{align*}
    \sup_{Q_\rho^{(\lambda)}}\big|u-(u)_{\rho}^{(\lambda)}\big|
    &\le 
    \sup_{Q_\rho^{(\lambda)}}\big|u-(u)_{2\rho}^{(\lambda)}\big| + 
    \big|(u)_{\rho}^{(\lambda)} - (u)_{2\rho}^{(\lambda)}\big| \\
    &\le 
    \sup_{Q_\rho^{(\lambda)}}\big|u-(u)_{2\rho}^{(\lambda)}\big| + 
    \biint_{Q_\rho^{(\lambda)}} \big|u - (u)_{2\rho}^{(\lambda)}\big| \,\dx\dt \\
    &\le 
    2\sup_{Q_\rho^{(\lambda)}}\big|u-(u)_{2\rho}^{(\lambda)}\big|  .
\end{align*}
Then, Lemma~\ref{lem:intr-sup}  yields
\begin{align*}
	\sup_{Q_\rho^{(\lambda)}} \frac{\big|u-(u)_{\rho}^{(\lambda)}\big|}{\lambda^{\frac{p-2}{2}}\rho}
    \le 
    C\lambda + 
    C\bigg[ 
    \biint_{Q_{2\rho}^{(\lambda)}} 
    \Big[| \sfF|^{qp} + \big|\lambda^{\frac{p-2}{2}}\rho\, \sff\big|^{qp'}\Big] \,\dx\dt
    \bigg]^\frac{1}{pq}.
\end{align*}
Substituting the preceding bound and applying the Poincaré inequality from Corollary~\ref{cor:poin} together with Hölder’s inequality, we obtain
\begin{align*}
	\lambda^\sfr
    &\le
    C\Bigg[
    \lambda +
    \bigg[ 
    \biint_{Q_{2\rho}^{(\lambda)}} 
    \Big[| \sfF|^{qp} + \big|\lambda^{\frac{p-2}{2}}\rho\, \sff\big|^{qp'}\Big] \,\dx\dt
    \bigg]^\frac{1}{pq} 
    \Bigg]^{\sfr-p} \\
    &\quad\cdot
    \bigg[\epsilon\lambda^p +
	\frac{1}{\epsilon^{\frac{2-p}{p-1}}}\biint_{Q_\rho^{(\lambda)}} \Big[|Du|^{p} + |\sfF|^{p}+\big|\lambda^{\frac{p-2}{2}}\rho \,\sff\big|^{p'}\Big]\,\dx\dt
	\bigg]\\
    &\le
    C\epsilon\Bigg[
    \lambda^p + 
    \bigg[ 
    \biint_{Q_{2\rho}^{(\lambda)}} 
    \Big[| \sfF|^{qp} + \big|\lambda^{\frac{p-2}{2}}\rho \,\sff\big|^{qp'}\Big] \,\dx\dt
    \bigg]^\frac{1}{q} 
    \Bigg]^\frac{\sfr-p}{p} \\
    &\quad\cdot
    \Bigg[\lambda^p +
	\frac{1}{\epsilon^{\frac{1}{p-1}}}\bigg\{\biint_{Q_\rho^{(\lambda)}} |Du|^{p}\,\dx\dt + 
    \bigg[
    \biint_{Q_{2\rho}^{(\lambda)}}
    \Big[|\sfF|^{qp}+\big|\lambda^{\frac{p-2}{2}}\rho\, \sff\big|^{qp'}\Big]\,\dx\dt\bigg]^\frac1q
	\bigg\}\Bigg]\\
    &\le C\lambda^\sfr\epsilon
    +
    C_\epsilon
    \Bigg[
    \biint_{Q_\rho^{(\lambda)}} |Du|^{p}\,\dx\dt
    +
    \bigg[\biint_{Q_{2\rho}^{(\lambda)}} 
    \Big[| \sfF|^{qp} + \big|\lambda^{\frac{p-2}{2}}\rho \,\sff\big|^{qp'}\Big] \,\dx\dt
    \bigg]^\frac{1}{q}
    \Bigg]^\frac{\sfr}{p}
\end{align*}
for any $\epsilon\in(0,1]$. Choosing $\varepsilon>0$ so small that $C\,\varepsilon\le \tfrac12$, the first term on the right can be absorbed into the left–hand side, and the previous estimate yields
\begin{equation*}
    \lambda^{\sfr}
    \le
    C
    \Bigg[
    \biint_{Q_\rho^{(\lambda)}} |Du|^{p}\,\dx\dt
    +
    \bigg[\biint_{Q_{2\rho}^{(\lambda)}} 
    \Big[| \sfF|^{qp} + \big|\lambda^{\frac{p-2}{2}}\rho\, \sff\big|^{qp'}\Big] \,\dx\dt
    \bigg]^\frac{1}{q}
    \Bigg]^\frac{\sfr}{p}.
\end{equation*}
Raising both sides to the power $p/\sfr$ gives the desired bound.
\end{proof}

\subsection{Reverse Hölder-type inequality}\label{sec:rev-hoel}
Under the sub- and super-intrinsic assumptions \eqref{sub-intr-rh} -- \eqref{super-intr-rh}, the intrinsic parameter $\lambda$ is calibrated by the local energy and the source terms (cf. Lemma~\ref{lem:sub-intr-2}). In combination with the Caccioppoli estimate and the parabolic Poincaré inequality on intrinsic cylinders (Corollary~\ref{cor:poin}), this calibration prevents loss of scale and yields a self-improvement of integrability for $|Du|$ in the Gehring sense. The next proposition records the resulting reverse Hölder inequality on $Q_{\rho}^{(\lambda)}(z_o)$, with constants depending only on the structural quantities and the coupling parameter $\sfK$.

\begin{proposition}[Reverse Hölder on intrinsic cylinders]\label{prop:RevH}
Let $p>1$ and $\sfK\ge1$. There exists a constant $C=C(N,p,C_o,C_1,q,\sfr,\sfK)>0$ with the following property. If $u$ is a weak solution to \eqref{eq-par-gen} in $E_T$ in the sense of Definition~\ref{def:weak_solution}, satisfying \eqref{growth-a*}, and if $Q_{2\rho}^{(\lambda)}(z_o)\subseteq E_T$ is a parabolic cylinder on which both
\begin{align}\label{sub-intr-rh}
	\biint_{Q_{2\rho}^{(\lambda)}(z_o)}
	\frac{\big|u-(u)_{z_o;2\rho}^{(\lambda)}\big|^\sfr}{(2\rho)^\sfr} \,\dx\dt 
    \le
    \sfK\lambda^\frac{p\sfr}{2}
\end{align}
and 
\begin{align}\label{super-intr-rh}
    \lambda^p
	\le
	\sfK\Bigg[\biint_{Q_\rho^{(\lambda)}(z_o)}
	|Du|^p \,\dx\dt + 
    \bigg[\biint_{Q_{2\rho}^{(\lambda)}(z_o)} \Big[| \sfF|^{qp} + \big|\lambda^{\frac{p-2}{2}}\rho\, \sff\big|^{qp'}\Big]\,\dx\dt \bigg]^{\frac{1}{q}} \Bigg]
\end{align}
hold, then the reverse Hölder inequality
\begin{align*}%\label{reverse-Hoelder}\nonumber
	\biint_{Q_{\rho}^{(\lambda)}(z_o)} |Du|^{p}  \,\dx\dt
    &\le
	C \bigg[\biint_{Q_{2\rho}^{(\lambda)}(z_o)}
	|Du|^{\theta}  \,\dx\dt\bigg]^{\frac{p}{\theta}} \\
    &\quad +
	C \bigg[\biint_{Q_{2\rho}^{(\lambda)}(z_o)} \Big[| \sfF|^{qp} + \big|\lambda^{\frac{p-2}{2}}\rho\, \sff\big|^{qp'}\Big]\,\dx\dt \bigg]^{\frac{1}{q}},
\end{align*}
is valid, where $\theta:=\frac{Np}{N+2}$. 
\end{proposition}

\begin{proof}
Throughout the proof, the reference point $z_o$ is suppressed from the notation. Set $a:=(u)_{2\rho}^{(\lambda)}$ and define 
\begin{align*}
    \phi(r)
    &:=
    \sup_{t\in\Lambda_{r}}
    \mint_{B_{r}^{(\lambda)}}
    \frac{|u(t)-a|^{2}}{r^{2}}\,\dx
    +
\biint_{Q_{r}^{(\lambda)}} |Du|^{p}\,\dx\dt.
\end{align*}
for $\rho\le r\le 2\rho$. Applying Lemma~\ref{lm:energy-without-zeta} on nested cylinders $Q_r^{(\lambda)}\subset Q_s^{(\lambda)}$ for $\rho\le r<s\le 2\rho$ gives the energy inequality
\begin{align*}
    \phi(r)
    &\le
    C\biint_{Q_s^{(\lambda)}}\bigg[
    \Big(\frac{s}{s-r}\Big)^p\frac{|u-a|^p}{(\lambda^{\frac{p-2}{2}}s)^p} +
    \Big(\frac{s}{s-r}\Big)^2\frac{|u-a|^2}{s^2}\bigg]\dx\dt\\
    &\phantom{\le\,}
    +
    C\biint_{Q_s^{(\lambda)}}\Big[
    |\sfF|^p+\big|\lambda^{\frac{p-2}{2}}s\,\sff\big|^{p'} \Big]\,\dx\dt. 
\end{align*}
Since $p<2<\sfr$, interpolate between $L^p$ and $L^\sfr$ to control the $L^2$-term on the right-hand side. Using \eqref{sub-intr-rh} and then Young’s inequality with conjugate exponents $\tfrac{\sfr-p}{2-p}$ and $\tfrac{\sfr-p}{\sfr-2}$, we obtain
\begin{align*}
    \biint_{Q_{s}^{(\lambda)}} &
    \frac{|u-a|^2}{s^2} \,\dx\dt \\ 
	&\le
    \bigg[\biint_{Q_{s}^{(\lambda)}}
    \frac{|u-a|^\sfr}{s^\sfr} \,\dx\dt\bigg]^{\frac{2-p}{\sfr-p}} 
    \bigg[\biint_{Q_{s}^{(\lambda)}}
    \frac{|u-a|^p}{s^p} \,\dx\dt\bigg]^{\frac{\sfr-2}{\sfr-p}} \\
    &\le
    C\lambda^{\frac{p(p-2)(\sfr-2)}{2(\sfr-p)}} 
    \bigg[
    \underbrace{\biint_{Q_{2\rho}^{(\lambda)}}
    \frac{|u-a|^\sfr}{(2\rho)^\sfr} \,\dx\dt}_{\le\sfK 
    \lambda^\frac{p\sfr}{2}}\bigg]^{\frac{2-p}{\sfr-p}} 
    \bigg[\biint_{Q_{s}^{(\lambda)}}
    \frac{|u-a|^p}{(\lambda^{\frac{
    p-2}{2}}s)^p} \,\dx\dt\bigg]^{\frac{\sfr-2}{\sfr-p}} \\
    &= 
	C\lambda^{\frac{p(2-p)}{\sfr-p}} \bigg[\biint_{Q_{s}^{(\lambda)}}
    \frac{|u-a|^p}{(\lambda^{\frac{p-2}{2}}s)^p} \,\dx\dt\bigg]^{\frac{\sfr-2}{\sfr-p}} \\
    &\le 
    \tilde\epsilon\lambda^p + 
    \frac{C}{\tilde\epsilon^{\frac{2-p}{\sfr-2}}}
    \biint_{Q_{s}^{(\lambda)}}
    \frac{|u-a|^p}{(\lambda^{\frac{p-2}{2}}s)^p} \,\dx\dt,
\end{align*}
for any $\tilde\epsilon\in (0,1]$. Substituting the preceding interpolation bound into the energy inequality gives, for $\rho\le r<s\le2\rho$,
\begin{align*}
    \phi(r)
    &\le
    C\sfR_{r,s}^2 \lambda^p \tilde\epsilon +
	\frac{C\,\sfR_{r,s}^p}{\tilde\epsilon^{\frac{2-p}{\sfr-2}}}
    \biint_{Q_{s}^{(\lambda)}}
    \frac{|u-a|^p}{(\lambda^{\frac{p-2}{2}}s)^p} \,\dx\dt +
    C \sfI,
\end{align*}
where $\sfR_{r,s}:=\frac{\rho}{s-r}$ and
$$
    \sfI
    :=
    \bigg[\biint_{Q_\rho^{(\lambda)}(z_o)} \Big[| \sfF|^{qp} + \big|\lambda^{\frac{p-2}{2}}\rho\, \sff\big|^{qp'}\Big]\,\dx\dt \bigg]^{\frac{1}{q}}. 
$$
Here $\sfI$ arises from the source terms after an application of Hölder’s inequality in the energy estimate. 
Applying the Gagliardo-Nirenberg inequality from Lemma \ref{lem:gag} on $Q_s^{(\lambda)}$ (with $|u-a|/(\lambda^\frac{p-2}2s)$) yields 
\begin{align*}
	\biint_{Q_{s}^{(\lambda)}} &
    \frac{|u-a|^p}{(\lambda^{\frac{p-2}{2}}s)^p} \,\dx\dt \\ 
    &\le
	C\Bigg[
    \lambda^{2-p}\sup_{t\in\Lambda_{s}}
    \underbrace{\mint_{B_{s}^{(\lambda)}}
    \frac{|u(t)-a|^2}{s^2} \,\dx}_{\le \phi(s)}
	\Bigg]^{\frac{p}{N+2}}
    \biint_{Q_{s}^{(\lambda)}} 
    \bigg[ |Du|^\theta + \frac{|u-a|^\theta}{(\lambda^{\frac{p-2}{2}}s)^\theta} \bigg] \,\dx\dt \\
    &\le 
    C \big[\lambda^{2-p}\phi(s)\big]^{\frac{p}{N+2}}
    \biint_{Q_{2\rho}^{(\lambda)}} 
    \bigg[ |Du|^\theta + \frac{|u-a|^\theta}{(\lambda^{\frac{p-2}{2}}\rho)^\theta} \bigg] \,\dx\dt ,
\end{align*}
since $Q_{s}^{(\lambda)}\subset Q_{2\rho}^{(\lambda)}$ and enlarging the domain of integration only increases the integral.
Applying Young’s inequality with conjugate exponents $\frac{N+2}{p}$ and $\frac{N+2}{N+2-p}$ to the factor $\phi^\frac{p}{N+2}$
in the last bound, we obtain
\begin{align*}
    \phi(r)
    &\le
	C\sfR_{r,s}^2 \lambda^p \tilde\epsilon
    +
     C\sfI \\
     &\phantom{\le \,}
    +\phi^\frac{p}{N+2}(s)
    \cdot
    \frac{C\,\sfR_{r,s}^p
    \lambda^{\frac{p(2-p)}{N+2}}}{\tilde\epsilon^{\frac{2-p}{\sfr-2}}}
    \biint_{Q_{2\rho}^{(\lambda)}} 
    \bigg[ |Du|^\theta + \frac{|u-a|^\theta}{(\lambda^{\frac{p-2}{2}}\rho)^\theta} \bigg] \,\dx\dt
   \\
    &\le 
    \tfrac12 \phi(s)
    +
    C\sfR_{r,s}^2 \lambda^p \tilde\epsilon
    +
    C\sfI\\
    &\phantom{\le\,}
    +
    C_{\tilde\epsilon}
    \sfR_{r,s}^{\frac{Np}{N-\theta}}
    \lambda^{\frac{p(2-p)}{N+2-p}} 
    \Bigg[ \biint_{Q_{2\rho}^{(\lambda)}} 
    \bigg[ |Du|^\theta + \frac{|u-a|^\theta}{(\lambda^{\frac{p-2}{2}}\rho)^\theta} \bigg] \,\dx\dt \Bigg]^{\frac{N+2}{N+2-p}}
    .
\end{align*}
Applying the iteration lemma (Lemma~\ref{lem:tech-classical}) to reabsorb the term $\tfrac12\phi(s)$ yields
\begin{equation*}
    \phi(\rho)
    \le 
    C \lambda^p \tilde\epsilon +
    C_{\tilde\epsilon} \lambda^{\frac{p(2-p)}{N+2-p}} 
    \Bigg[ \biint_{Q_{2\rho}^{(\lambda)}} 
    \bigg[ |Du|^\theta + \frac{|u-a|^\theta}{(\lambda^{\frac{p-2}{2}}\rho)^\theta} \bigg] \,\dx\dt \Bigg]^{\frac{N+2}{N+2-p}} +
    C\sfI.
\end{equation*}
A further application of Young’s inequality with conjugate exponents
$\frac{N+2-p}{2-p}$ and $\frac{N+2-p}{N}$ gives
\begin{equation*}
    \phi(\rho)
    \le 
    C \lambda^p \tilde\epsilon+
    C_{\tilde\epsilon} 
    \Bigg[ \biint_{Q_{2\rho}^{(\lambda)}} 
    \bigg[ |Du|^\theta + \frac{|u-a|^\theta}{(\lambda^{\frac{p-2}{2}}\rho)^\theta} \bigg] \,\dx\dt \Bigg]^{\frac{p}{\theta}} +
    C\sfI,
\end{equation*}
where $\tilde\epsilon\in (0,1]$. 
Choose $\tilde\varepsilon=(4C\sfK)^{-1}$, where $C$ is the constant in front of $\lambda^{p}$ in the preceding bound and $\sfK$ is the constant from \eqref{super-intr-rh}. Applying Corollary~\ref{cor:poin} on $Q_{2\rho}^{(\lambda)}$ with $\sigma=\theta\in[p-1,p]$ to $w=u-a$, and using Hölder’s inequality, we infer that for any $\epsilon\in(0,1]$,
\begin{align*}
    \phi(\rho)
    &\le 
    \frac{\lambda^p}{4\sfK}  +
    C\Bigg[ \epsilon\lambda^\theta +
    \frac{1}{\epsilon^{\frac{2-p}{p-1}}} \biint_{Q_{2\rho}^{(\lambda)}} \Big[|Du|^{\theta} + |\sfF|^{\theta}+\big|\lambda^{\frac{p-2}{2}}\rho\,\sff\big|^\frac{\theta}{p-1}\Big]\,\dx\dt \Bigg]^{\frac{p}{\theta}} +
    C\sfI \\
    &\le 
    \big[\tfrac{1}{4\sfK} + C\epsilon\big] \lambda^p +
    \frac{C}{\epsilon^{\frac{2-p}{p-1}\cdot\frac{p}{\theta}}} 
    \bigg[ 
    \biint_{Q_{2\rho}^{(\lambda)}} |Du|^{\theta} \,\dx\dt \bigg]^{\frac{p}{\theta}} +
    \frac{C}{\epsilon^{\frac{2-p}{p-1}\cdot\frac{p}{\theta}}}\sfI .
\end{align*}
Choosing $\varepsilon=\frac{1}{4C\sfK}$ in the preceding bound gives
\begin{equation*}
    \biint_{Q_{\rho}^{(\lambda)}} |Du|^{p}\,\dx\dt
    \le
    \phi(\rho)
    \le
    \frac{\lambda^{p}}{2\sfK}\,
    +
    C\bigg[\biint_{Q_{2\rho}^{(\lambda)}} |Du|^{\theta}\,\dx\dt\bigg]^{\frac{p}{\theta}}
    +
    C\sfI.
\end{equation*}
By the super–intrinsic assumption \eqref{super-intr-rh},
\begin{align*}
    \frac{\lambda^{p}}{2\sfK}
    \le
    \tfrac12\biint_{Q_{\rho}^{(\lambda)}} |Du|^{p}\,\dx\dt
    +
    \tfrac12
    \bigg[\biint_{Q_{2\rho}^{(\lambda)}}
    \Big[|\sfF|^{qp}
    +|\lambda^{\frac{p-2}{2}}\rho\,\sff|^{qp'}\Big]\,\dx\dt\bigg]^{\frac{1}{q}}.
\end{align*}
Inserting this and absorbing $\tfrac12\biint_{Q_{\rho}^{(\lambda)}} |Du|^{p}\dx\dt$ into the left–hand side yields the claimed reverse Hölder inequality.
\end{proof}

\section{Higher integrability}\label{S:5}
To set up the higher integrability scheme, fix a standard cylinder
$$
	Q_{8R}(y_o,\tau_o)
	\equiv
	B_{8R}(y_o)\times \big(\tau_o-(8R)^2,\tau_o+(8R)^2\big)
	\subseteq E_T.
$$
Throughout this section the center $(y_o,\tau_o)$ is suppressed from the notation; write $Q_\rho:=Q_\rho(y_o,\tau_o)$ for $\rho\in(0,8R]$. Choose $\sfr>2$ so that \eqref{Eq:lm>0} is satisfied 
and introduce a parameter $\lambda_o\ge1$ (to be fixed below) satisfying
\begin{equation}\label{lambda-0}
  \lambda_o
  \ge
  1+\bigg[\biint_{Q_{4R}} 
  \frac{|u-(u)_{4R}|^{\sfr}}{(4 R)^{\sfr}} \,\dx\dt
  \bigg]^{\frac{2}{ \boldsymbol \lambda_\sfr}}.
\end{equation}

For $z_o\equiv(x_o,t_o)\in Q_{2R}$, $\rho\in(0,R]$, and $\lambda\ge1$, consider the cylinders $Q_{\rho}^{(\lambda)}(z_o)$ as in \eqref{def-Q}. Since the spatial radius is $\lambda^{\frac{p-2}{2}}\rho$ and $p\le2$, these cylinders are monotone in $\lambda$:
$$
1\le \lambda_1<\lambda_2
\quad\Longrightarrow\quad
Q_{\rho}^{(\lambda_2)}(z_o)\subset Q_{\rho}^{(\lambda_1)}(z_o).
$$
Moreover, %if $z_o\in Q_{2R}$ and $\rho\le R$, then
$$
B^{(\lm)}_{\rho}(x_o)\subset B_{3R}(y_o)\subset B_{4R}(y_o),
\qquad
\Lambda_{\rho}(t_o)\subset\big(\tau_o-(4R)^2,\tau_o+(4R)^2\big);
$$
hence, $Q_{\rho}^{(\lambda)}(z_o)\subset Q_{4R}$ for all $\lambda\ge1$.

\subsection{Construction of an intrinsic, non-uniform cylinder system}\label{sec:cylinders}
The following construction of intrinsic cylinders has its origin in \cite{Gianazza-Schwarzacher, Schwarzacher}. 
Lemma \ref{lem:crazy-cylinders} records the basic properties of the scale parameter $\rho\mapsto\lambda_{z_o;\rho}$ attached to a fixed center $z_o\in Q_{R}$ and depending on the radius $\rho$. It is well defined for $\rho\in(0,R]$, continuous, and non-increasing; it calibrates the intrinsic geometry of $u$ at radius $\rho$, is uniformly controlled from below by the initial level $\lambda_o$, and exhibits the correct growth under changes of scale. In particular, the associated cylinders are nested, satisfy a sub-intrinsic oscillation bound,
and satisfy a quantitative dichotomy: either they become intrinsic at some larger radius
$\widetilde\rho\in[\rho,R/2]$ or a global bound for $\lambda_{z_o;\rho}$ holds.

\begin{lemma}\label{lem:crazy-cylinders}
Let $\lambda_{o}\ge1$, $z_{o}\in Q_{R}$, and $u\in L^{\sfr}(Q_{2R};\mathbb{R}^{k})$. There exists a function $(0,R]\ni\rho\mapsto \lambda_{z_o;\rho}\in[\lambda_{o},\infty)$ with the following properties:
\begin{itemize}
\item[a)]
The function $\rho\mapsto \lambda_{z_o;\rho}$ is continuous and non-increasing.

\item[b)]For every $s\in(\rho,R]$,
$$
\lambda_{z_o;\rho} 
\le
4^{\frac{2\sfr}{\boldsymbol\lambda_\sfr}}
\Big(\frac{s}{\rho}\Big)^{\frac{2}{\boldsymbol\lambda_\sfr}(N+2+\sfr)}\,
\lambda_{z_o;s},
$$
and, moreover,
$$
\lambda_{z_o;R}
\le
4^{\frac{2}{\boldsymbol\lambda_\sfr}(N+2+2\sfr)}\,\lambda_{o}.
$$
\item[c)] The associated scaled cylinders are nested:
$$
Q_{\rho}^{(\lambda_{z_o;\rho})}(z_o)\subset Q_{s}^{(\lambda_{z_o;s})}(z_o)\;\subseteq\;Q_{2R}
\qquad\text{whenever }\,0<\rho<s\le R.
$$
\item[d)] Sub–intrinsic control holds on each scale:
$$
\biint_{Q_{s}^{(\lambda_{z_o;\rho})}(z_o)}
\frac{\big|u-(u)_{z_o;s}^{(\lambda_{z_o;\rho})}\big|^{\sfr}}{s^{\sfr}}\,\mathrm{d}x\mathrm{d}t
\le
2^{\sfr}\lambda_{z_o;\rho}^{\frac{p\sfr}{2}}
\quad\text{for every }s\in[\rho,R].
$$
\item[e)] Either
$$
    \lambda_{z_o;\rho}
    \le
    4^{\frac{2}{\boldsymbol\lambda_\sfr}(N+2+3\sfr)}\lambda_{o}
$$
or there exists $\widetilde\rho\in\big[\rho,\tfrac{R}{2}\big]$ such that
$$
\biint_{Q_{\widetilde\rho}^{(\lambda_{z_o;\widetilde\rho})}(z_o)}
\frac{\big|u-(u)_{z_o;\widetilde\rho}^{(\lambda_{z_o;\widetilde\rho})}\big|^{\sfr}}{\widetilde\rho^{\sfr}}\,\mathrm{d}x\mathrm{d}t
=
\lambda_{z_o;\widetilde\rho}^{\frac{p\sfr}{2}}
=
\lambda_{z_o;\rho}^{\frac{p\sfr}{2}}.
$$
\end{itemize}
\end{lemma}

\begin{proof}
The proof is divided into several steps.

\emph{Step 1. Definition of $\widetilde\lambda_{z_o;\rho}$.}
For $z_o\in Q_{2R}$, $\rho\in(0,R]$, and $\lambda\ge1 $, consider the cylinders $Q_\rho^{(\lambda)}(z_o)$ from \eqref{def-Q}. Define
$$
	\widetilde\lambda_\rho
	\equiv
	\widetilde\lambda_{z_o;\rho}
	:=
	\inf\bigg\{\lambda\in[\lambda_o,\infty):
	\frac{1}{|Q_\rho|}
	\iint_{Q^{(\lambda)}_\rho(z_o)} 
	\frac{\big|u-(u)_{z_o;\rho}^{(\lambda)}\big|^{\sfr}}{\rho^{\sfr}} \,\dx\dt 
	\le 
	\lambda^{\frac{\boldsymbol{\lambda}_\sfr}{2}} \bigg\}.
$$
This quantity is well defined: the set over which the infimum is taken is nonempty, since as $\lambda\to\infty$ the spatial radius $\lambda^{\frac{p-2}{2}}\rho$ tends to $0$ (since $p\le2$), hence, the left–hand side tends to $0$ by the absolute continuity of the integral, whereas the right–hand side grows like $\lambda^{\frac{\boldsymbol{\lambda}_\sfr}{2}}$. The choice of the exponent is natural:  the condition is equivalent to
$$
  \biint_{Q^{(\lambda)}_\rho(z_o)}
  \frac{\big|u-(u)_{z_o;\rho}^{(\lambda)}\big|^{\sfr}}{\rho^{\sfr}}\,\dx\dt 
  \le 
  \lambda^\frac{p\sfr}{2},
$$
which matches the sub/super–intrinsic couplings used in \S\S~\ref{sec:est-intrinsic-cyl} and~\ref{sec:rev-hoel}.
As an immediate consequence of the definition of
$\widetilde\lambda_\rho$, we obtain the following dichotomy: either
\begin{equation}\label{theta=lambda}
	\widetilde\lambda_\rho=\lambda_o
	\quad\mbox{and}\quad
	\biint_{Q_{\rho}^{(\widetilde\lambda_\rho)}(z_o)}
    \frac{\big|u-(u)_{z_o;\rho}^{(\widetilde\lambda_\rho)}\big|^{\sfr}}{\rho^\sfr} \,\dx\dt
	\le
	\widetilde\lambda_\rho^\frac{p\sfr}{2}
	=
	\lambda_o^\frac{p\sfr}{2}
\end{equation}
or
\begin{equation}\label{theta>lambda}
	\widetilde\lambda_\rho>\lambda_o
	\quad\mbox{and}\quad
	\biint_{Q_{\rho}^{(\widetilde\lambda_\rho)}(z_o)} 
	\frac{\big|u-(u)_{z_o;\rho}^{(\widetilde\lambda_\rho)}\big|^{\sfr}}{\rho^\sfr}\,\dx\dt
	=
	\widetilde\lambda_\rho^\frac{p\sfr}{2}.
\end{equation}
Also in the case $\rho=R$ we have $\widetilde\lambda_{R}\ge\lambda_{o}
\ge1$. Moreover, if $\widetilde\lambda_{R}>\lambda_{o}$, then by \eqref{theta>lambda}, Lemma~\ref{lem:mean}, the inclusion $Q_{R}^{(\widetilde\lambda_{R})}(z_o)\subset Q_{4R}$, and \eqref{lambda-0}, we find
\begin{align*}
	\widetilde\lambda_{R}^{\frac{\boldsymbol\lambda_\sfr}{2}}
	&=
	\frac{1}{|Q_{R}|}
	\iint_{Q_{R}^{(\widetilde\lambda_{R})}(z_o)}
	\frac{\big|u-(u)_{z_o;R}^{(\widetilde\lambda_R)}\big|^\sfr}{R^{\sfr}} \,\dx\dt\\ 
	&\le
	\frac{2^\sfr}{|Q_{R}|}
	\iint_{Q_{R}^{(\widetilde\lambda_{R})}(z_o)}
	\frac{|u-(u)_{4R}|^\sfr}{R^{\sfr}} \,\dx\dt\\ 
	&\le
	\frac{8^{\sfr}}{|Q_{R}|}
	\iint_{Q_{4R}} \frac{|u-(u)_{4R}|^{\sfr}}{(4R)^{\sfr}} \,\dx\dt\\
	&\le 
	4^{N+2+2\sfr}\lambda_o^{\frac{\boldsymbol \lambda_\sfr}{2}}.
\end{align*}
Therefore,
\begin{equation}\label{bound-theta-R}
    \widetilde\lambda_{R}
    \le
    4^{\frac{2}{\boldsymbol\lambda_\sfr}(N+2+2\sfr)}\,\lambda_{o}.
\end{equation}

\textit{Step 2. Continuity of $(0,R]\ni\rho\mapsto
\widetilde\lambda_\rho$.}
Our next goal is to show that the map $(0,R]\ni\rho\mapsto\widetilde\lambda_\rho$ is continuous. Fix $\rho\in(0,R]$ and $\varepsilon>0$, and set $\lambda_{+}:=\widetilde\lambda_\rho+\varepsilon$. We claim that there exists $\delta=\delta(\varepsilon,\rho)>0$ such that for all $r\in(0,R]$ with $|r-\rho|<\delta$,
$$
  \frac{1}{|Q_r|}
  \iint_{Q_{r}^{(\lambda_{+})}(z_o)}
  \frac{\big|u-(u)_{z_o;r}^{(\lambda_{+})}\big|^{\sfr}}{r^{\sfr}}\,\mathrm{d}x\mathrm{d}t
  <
  \lambda_{+}^{\frac{\boldsymbol\lambda_\sfr}{2}}.
$$
Indeed, for $r=\rho$ this is immediate from the definition of $\widetilde\lambda_\rho$, since $\lambda_{+}>\widetilde\lambda_\rho$. By the absolute continuity of the integral and the continuity of the involved averages with respect to the radius, the strict inequality persists for $|r-\rho|$ sufficiently small. By the definition of $\widetilde\lambda_r$ it follows that $\widetilde\lambda_r<\lambda_{+}=\widetilde\lambda_\rho+\varepsilon$ whenever $|r-\rho|<\delta$. In particular,
$$
\limsup_{r\to\rho}\,\widetilde\lambda_r\;\le\;\widetilde\lambda_\rho.
$$
For the reverse inequality, set $\lambda_{-}:=\max\{\lambda_{o},\,\widetilde\lambda_\rho-\varepsilon\}$. If $\widetilde\lambda_\rho=\lambda_{o}$, then $\widetilde\lambda_r\ge\lambda_{o}\ge\widetilde\lambda_\rho-\varepsilon$ for all $r$, and there is nothing to prove. If $\widetilde\lambda_\rho>\lambda_{o}$, then, by the defining equality at the infimum,
$$
\frac{1}{|Q_\rho|}
\iint_{Q_{\rho}^{(\widetilde\lambda_\rho)}(z_o)}
\frac{\big|u-(u)_{z_o;\rho}^{(\widetilde\lambda_\rho)}\big|^{\sfr}}{\rho^{\sfr}}\,\mathrm{d}x\mathrm{d}t
=
\widetilde\lambda_\rho^{\frac{\boldsymbol\lambda_\sfr}{2}}
>
\lambda_-^{\frac{\boldsymbol\lambda_\sfr}{2}}.
$$
By continuity in the radius (again using absolute continuity), there exists $\delta'=\delta'(\varepsilon,\rho)>0$ such that for all $r$ with $|r-\rho|<\delta'$,
$$
\frac{1}{|Q_r|}
\iint_{Q_{r}^{(\lambda_-)}(z_o)}
\frac{\big|u-(u)_{z_o;r}^{(\lambda_-)}\big|^{\sfr}}{r^{\sfr}}\,\mathrm{d}x\mathrm{d}t
>
\lambda_-^{\frac{\boldsymbol\lambda_\sfr}{2}}.
$$
Hence, by the definition of $\widetilde\lambda_r$, one must have $\widetilde\lambda_r\ge\lambda_- \ge \widetilde\lambda_\rho-\varepsilon$ for $|r-\rho|<\delta'$. Therefore,
$$
\liminf_{r\to\rho}\,\widetilde\lambda_r\;\ge\;\widetilde\lambda_\rho-\varepsilon.
$$
Since $\varepsilon>0$ is arbitrary, we conclude that $\widetilde\lambda_r\to\widetilde\lambda_\rho$ as $r\to\rho$, i.e., the map $\rho\mapsto\widetilde\lambda_\rho$ is continuous on $(0,R]$.

\textit{Step 3. Definition of $\lambda_{z_o;\rho}$.}
Since the map $(0,R]\ni\rho\mapsto\widetilde\lambda_\rho$ need not be decreasing, introduce its decreasing envelope
$$
	\lambda_\rho
	\equiv
	\lambda_{z_o;\rho}
	:=
	\max_{r\in[\rho,R]} \widetilde\lambda_{z_o;r}\,. 
$$
By construction, $\lambda_\rho$ is monotone with respect to $\rho$: if $\rho_1<\rho_2$, then $[\rho_2,R]\subset[\rho_1,R]$ and hence $\lambda_{\rho_2}\le\lambda_{\rho_1}$. Moreover, since $\widetilde\lambda_{z_o;r}$ is continuous on every compact interval and maxima are attained, the map $\rho\mapsto\lambda_\rho=\max_{r\in[\rho,R]}\widetilde\lambda_{r;z_o}$ is continuous on $(0,R]$. In particular, property a) holds. Finally, $\lambda_\rho\ge\widetilde
\lambda_{\rho,z_o}\ge\lambda_o\ge1$, so property c) is also satisfied.

\textit{Step 4. Sub-intrinsic property.}
In general, the cylinders $Q_{\rho}^{(\lambda_\rho)}(z_o)$ cannot 
be expected to be intrinsic in the sense of \eqref{theta>lambda}.  
However, we can show that the cylinders $Q_{s}^{(\lambda_\rho)}(z_o)$
are sub-intrinsic for all radii $s\in[\rho,R]$ as stated in d). 
For the proof of this property, we use Lemma~\ref{lem:mean}, the chain of inequalities 
$\widetilde\lambda_s\le \lambda_{s}\le \lambda_{\rho}$, which implies
$Q_{s}^{(\lambda_{\rho})}(z_o)\subseteq
Q_{s}^{(\widetilde\lambda_{s})}(z_o)$, and the fact that the latter
cylinder is sub-intrinsic, i.e. it satisfies either \eqref{theta=lambda} or \eqref{theta>lambda}. Subsequently, we recall that $\frac{pr}{2} - \frac{N(2-p)}{2}=\frac{\boldsymbol\lambda_r}{2}>0$ and again use $\widetilde\lambda_s\le \lambda_{\rho}$.
In this way, we deduce 
\begin{align*}
	\biint_{Q_{s}^{(\lambda_{\rho})}(z_o)} 
	\frac{\big|u-(u)_{z_o;s}^{(\lambda_\rho)}\big|^{\sfr}}{s^\sfr} \,\dx\dt 
	&\le 
    2^\sfr\biint_{Q_{s}^{(\lambda_{\rho})}(z_o)} 
	\frac{\big|u-(u)_{z_o;s}^{(\widetilde\lambda_s)}\big|^{\sfr}}{s^\sfr} \,\dx\dt \\
	&\le
	2^\sfr\Big(\frac{\lambda_{\rho}}{\widetilde\lambda_{s}}\Big)^{\frac{N(2-p)}{2}}
	\biint_{Q_{s}^{(\widetilde\lambda_{s})}(z_o)}
    \frac{\big|u-(u)_{z_o;s}^{(\widetilde\lambda_s)}\big|^{\sfr}}{s^{\sfr}} \,\dx\dt \\
	&\le 
	2^\sfr\Big(\frac{\lambda_{\rho}}{\widetilde\lambda_{s}}\Big)^{\frac{N(2-p)}{2}}
	\widetilde\lambda_{s}^\frac{p\sfr}{2}
	=
	2^\sfr\lambda_\rho^{\frac{N(2-p)}{2}}\,\widetilde\lambda_s^{\frac{p\sfr}{2} - \frac{N(2-p)}{2}} \\
	&\le 
	2^\sfr\lambda_\rho^\frac{p\sfr}{2},
\end{align*}
which is exactly assertion d). 

\textit{Step 5. Upper bound for $\lambda_{z_o;\rho}$.}
Define
\begin{equation*}%\label{rho-tilde}
	\widetilde\rho
	:=
	\left\{
	\begin{array}{cl}
	R, &
	\quad\mbox{if $\lambda_\rho=\lambda_o$,} \\[5pt]
	\min\big\{s\in[\rho, R]: \lambda_s=\widetilde \lambda_s \big\}, &
	\quad\mbox{if $\lambda_\rho>\lambda_o$.}
	\end{array}
	\right.
\end{equation*}
By construction, $\lambda_s=\widetilde\lambda_{\widetilde\rho}$ for every $s\in[\rho,\widetilde\rho]$, in particular $\lambda_{\widetilde\rho}=\widetilde\lambda_{\widetilde\rho}=\lambda_\rho$. The aim is to prove
\begin{align}\label{bound-theta}
  \lambda_\rho 
  \le
  4^\frac{2\sfr}{\boldsymbol\lambda_\sfr}
  \Big(\frac{s}{\rho}\Big)^{\frac{2}{\boldsymbol\lambda_\sfr}(N+2+\sfr)}
  \lambda_{s}  
  \quad\mbox{for any $s\in(\rho,R]$.}
\end{align}
If $\lambda_\rho=\lambda_o$, the bound is immediate since $\lambda_s\ge\lambda_o$. If $s\in(\rho,\widetilde\rho]$, then $\lambda_s=\widetilde\lambda_{\widetilde\rho}=\lambda_\rho$ and the inequality holds as well. It remains to consider the case $\lambda_\rho>\lambda_o$ and $s\in(\widetilde\rho,R]$.
Fix such an $s$. 
Recall that $\widetilde\lambda_{\widetilde\rho}=\lambda_\rho>\lambda_o$. Hence, thanks to \eqref{theta>lambda} applied at $\widetilde\rho$, and subsequently using Lemma~\ref{lem:mean}, we obtain
\begin{align*}    
    \lambda_\rho^{\frac{\boldsymbol\lambda_\sfr}{2}}
    =
    \widetilde\lambda_{\widetilde\rho}^{\frac{\boldsymbol\lambda_\sfr}{2}}
    &=
    \frac{1}{|Q_{\widetilde\rho}|}
    \iint_{Q_{\widetilde\rho}^{(\widetilde\lambda_{\widetilde\rho})}(z_o)}
    \frac{\big|u-(u)_{z_o;\widetilde\rho}^{(\widetilde\lambda_{\widetilde\rho})}
    \big|^{\sfr}}{\widetilde\rho^{\sfr}}\,\mathrm{d}x\mathrm{d}t \\
    &\le
    \frac{2^{\sfr}}{|Q_{\widetilde\rho}|}
    \iint_{Q_{\widetilde\rho}^{(\widetilde\lambda_{\widetilde\rho})}(z_o)}
    \frac{\big|u-(u)_{z_o;s}^{(\lambda_s)}
    \big|^{\sfr}}{\widetilde\rho^{\sfr}}\,\mathrm{d}x\mathrm{d}t .
\end{align*}
Next, we notice that $Q_{\widetilde\rho}^{(\widetilde\lambda_{\widetilde\rho})}(z_o)\subset Q_{s}^{(\lambda_s)}(z_o)$, since $\widetilde\rho<s$ as well as $\widetilde\lambda_{\widetilde\rho}\ge\lm_s$. This allows us to increase the domain of integration from $Q_{\widetilde\rho}^{(\widetilde\lambda_{\widetilde\rho})}(z_o)$ to $Q_{s}^{(\lambda_s)}(z_o)$. Applying subsequently the sub-intrinsic property d), we find that 
\begin{align*}    
    \lambda_\rho^{\frac{\boldsymbol\lambda_\sfr}{2}}
    \le
    2^\sfr\Big(\frac{s}{\widetilde\rho}\Big)^{N+2+\sfr}
    \frac{1}{|Q_{s}|}
    \iint_{Q_{s}^{(\lambda_s)}(z_o)}
    \frac{\big|u-(u)_{z_o;s}^{(\lambda_s)}\big|^{\sfr}}{s^{\sfr}}\,\mathrm{d}x\mathrm{d}t 
    \le
    4^\sfr\Big(\frac{s}{\widetilde\rho}\Big)^{N+2+\sfr}
    \lambda_s^{\frac{\boldsymbol\lambda_\sfr}{2}}.
\end{align*}
In turn, we obtain
\begin{equation*}
    \widetilde\lambda_{\widetilde\rho}
    \le
    4^{\frac{2\sfr}{\boldsymbol\lambda_\sfr}}\Big(\frac{s}{\widetilde\rho}\Big)^{\frac{2}{\boldsymbol\lambda_\sfr}(N+2+\sfr)}\,\lambda_s.
\end{equation*}
Recalling that $\rho\le \widetilde\rho\le s$, we arrive at
\begin{equation*}
    \lambda_\rho
    \le
    4^{\frac{2\sfr}{\boldsymbol\lambda_\sfr}}\Big(\frac{s}{\rho}\Big)^{\frac{2}{\boldsymbol\lambda_\sfr}(N+2+\sfr)}\,\lambda_s,
\end{equation*}
which is \eqref{bound-theta}. 
This proves the first assertion in b). The second one follows directly from \eqref{bound-theta-R} together with $\lambda_{R}=\widetilde\lambda_{R}$.

\textit{Step 6. Proof of e).} 
If $\widetilde\rho\in (\frac{R}{2},R]$, then by b) with $s=R$,
$$
    \lambda_{\rho}
    =
    \lambda_{\widetilde\rho}
    \le 
    4^\frac{2\sfr}{\boldsymbol\lambda_\sfr}\Big(\frac{R}{\widetilde\rho}\Big)^{\frac{2}{\boldsymbol\lambda_\sfr}(N+2+\sfr)}
    \lambda_{R}.
$$
Invoking \eqref{bound-theta-R} gives
$$
   \lambda_{\rho}
   \le
   4^\frac{2\sfr}{\boldsymbol\lambda_\sfr}4^{\frac{2}{\boldsymbol\lambda_\sfr}(N+2+2\sfr)}
    \lambda_o 
    \le 
    4^{\frac{2}{\boldsymbol\lambda_\sfr}(N+2+3\sfr)}
    \lambda_o.
$$
If instead $\widetilde\rho\in [\rho,\frac{R}{2}]$, then by definition $\lambda_{\widetilde\rho}=\lambda_\rho$,
and the cylinder $Q_{\widetilde\rho}^{\lambda_{\widetilde\rho}}(z_o)$ satisfies \eqref{theta>lambda}. This establishes e). 
\end{proof}

\subsection{Intrinsic Vitali covering lemma}
%\label{sec:covering}
The next step is a Vitali–type covering property adapted to the intrinsic geometry encoded by $\rho\mapsto\lambda_{z;\rho}$. Given any family of cylinders $Q_{4\rho}^{(\lambda_{z;\rho})}(z)$ as in Lemma \ref{lem:crazy-cylinders} with small radii $\rho<R/\hat c$, one can extract a countable subfamily of pairwise disjoint cylinders whose $\tfrac{\hat c}{4}$-enlargements cover the whole family, for a structural constant $\hat c=\hat c(N,p,\sfr)\ge 20$. The selection follows the classical Vitali scheme but hinges on the monotonicity, scale control, and sub–intrinsic bounds from Lemma \ref{lem:crazy-cylinders}, together with a quantitative inclusion of overlapping cylinders that compensates for the intrinsic anisotropy. This covering will be the backbone of the reverse Hölder improvement and higher–integrability arguments.

\begin{lemma} 
\label{lem:vitali}
There exists a constant $\hat c=\hat c(N,p,\sfr)\ge 20$ with the following property. If $\mathcal F$ is any collection of cylinders $Q_{4\rho}^{(\lambda_{z;\rho})}(z)$, where each $Q_{\rho}^{(\lambda_{z;\rho})}(z)$ is a cylinder as in Lemma~$\ref{lem:crazy-cylinders}$ with radius $\rho\in\bigl(0,\tfrac{R}{\hat c}\bigr)$, then there exists a countable subfamily $\mathcal G\subset\mathcal F$ of pairwise disjoint cylinders such that
\begin{equation}\label{covering}
    \bigcup_{Q\in\mathcal F} Q \subseteq
\bigcup_{Q\in\mathcal G} \widehat Q,
\end{equation}
where $\widehat Q$ denotes the $\tfrac{\hat c}{4}$-fold enlargement of $Q$ (same center and parameter $\lambda$): if $Q=Q_{4\rho}^{(\lambda_{z;\rho})}(z)$, then $\widehat Q=Q_{\hat c\,\rho}^{(\lambda_{z;\rho})}(z)$.
\end{lemma}

\begin{proof}
Fix $\hat c\ge20$. For $j\in\mathbb N$ subdivide $\mathcal F$ into
$$
  \mathcal F_j
  :=
  \Big\{\,Q_{4\rho}^{(\lambda_{z;\rho})}(z)\in \mathcal F:
  \tfrac{R}{2^j\hat c}<\rho\le \tfrac{R}{2^{j-1}\hat c}\,\Big\}.
$$
Choose $\mathcal G_1\subset\mathcal F_1$ as an arbitrary maximal subfamily of pairwise disjoint cylinders. By Lemma~\ref{lem:crazy-cylinders} b) and the definition of $\mathcal F_1$, each element of $\mathcal G_1$ has a uniform lower volume bound; hence $\mathcal G_1$ is finite.
Assume that $\mathcal G_1,\mathcal G_2,\dots,\mathcal G_{k-1}$ have been constructed for some integer $k\ge 2$. Define $\mathcal G_k$ to be any maximal subfamily of $\mathcal F_k$ consisting of cylinders that are pairwise disjoint and disjoint from all previously selected cylinders:
$$
\mathcal G_k\subset\Big\{\,Q\in\mathcal F_k:\ Q\cap Q^\ast=\varnothing\ \text{for every}\ Q^\ast\in\bigcup_{j=1}^{k-1}\mathcal G_j\,\Big\}.
$$
As in the case $k=1$, each element of $\mathcal F_k$ enjoys a uniform lower bound on its volume (by Lemma \ref{lem:crazy-cylinders} b) and the definition of $\mathcal F_k$), hence $\mathcal G_k$ is finite. Consequently,
$$
\mathcal G:=\bigcup_{j=1}^{\infty}\mathcal G_j\subseteq\mathcal F
$$
is a countable family of pairwise disjoint cylinders. It remains to show that for every 
$Q\in\mathcal F$ there exists 
$Q^\ast\in\mathcal
G$ with  $Q\subset \widehat {Q}^\ast$.
To this end, fix a cylinder
$Q=Q_{4\rho}^{(\lambda_{z;\rho})}(z)\in\mathcal F$, and choose
$j\in\N$ with $Q\in\mathcal F_j$.
If $Q\in \mathcal G_j$, then $Q\subset \widehat Q$ is trivial.
Suppose $Q\not\in \mathcal G_j$.
By maximality of the families $\mathcal G_i$
up to level $j$, there exists a cylinder
\begin{equation*}
    Q^\ast
    =Q_{4\rho_\ast}^{(\lambda_{z_\ast;\rho_\ast})}(z_\ast)\in
\bigcup_{i=1}^{j} \mathcal G_i
\quad\mbox{with}\quad Q\cap Q^\ast\not=\varnothing.
\end{equation*}
In the sequel we show that  $Q\subset \widehat {Q}^\ast$.
First, we observe that the properties $\rho\le\tfrac{R}{2^{j-1}\hat c}$
and $\rho_\ast>\tfrac{R}{2^j\hat c}$ imply $\rho\le 2\rho_\ast$.
Since $Q\cap Q^\ast\neq\varnothing$,   $\Lambda_{4\rho}(t)\cap \Lambda_{4\rho_\ast}(t_\ast)\not=\varnothing$. Hence
$$
|t-t_\ast|\le 16\rho^2+16\rho_\ast^2 .
$$
For any $\tau\in\Lambda_{4\rho}(t)$,
$$
|\tau-t_\ast|\le |\tau-t|+|t-t_\ast|
\le 16\rho^2 + (16\rho^2+16\rho_\ast^2)
=32\rho^2+16\rho_\ast^2
\le 144\,\rho_\ast^2,
$$
where we used $\rho\le 2\rho_\ast$. Since $\hat c\ge20$, one has $144\,\rho_\ast^2\le (\hat c\rho_\ast)^2$, and therefore
$
\Lambda_{4\rho}(t)\subseteq \Lambda_{\hat c\rho_\ast}(t_\ast)
$, which ensures
$\Lambda_{4\rho}(t)\subseteq \Lambda_{\hat c\rho_\ast}(t_\ast)$ since $\hat c\ge20$.
To prove the corresponding spatial inclusion $B_{4\rho}^{(\lambda_{z;\rho})}(x)\subseteq B_{\hat c\,\rho_\ast}^{(\lambda_{z_\ast;\rho_\ast})}(x_\ast)$, we first establish the comparison estimate
\begin{equation}\label{control-lambda-1}
  \lambda_{z_\ast;\rho_\ast}
  \le
  72^{\frac{2}{ \boldsymbol\lambda_\sfr}(N+2+\sfr)}\,
  \lambda_{z;\rho}\,.
\end{equation}
To this aim we use Lemma \ref{lem:crazy-cylinders} e), which consists of two alternatives. In the first alternative we have $\lambda_{z_\ast;\rho_\ast}\le 4^{\frac{2}{\boldsymbol\lambda_\sfr}(N+2+3\sfr)}\lambda_{o}$, which in view of $\lambda_{z;\rho}\ge\lambda_o$ immediately implies~\eqref{control-lambda-1}. In the second alternative of Lemma \ref{lem:crazy-cylinders} e) there exists a radius $\widetilde\rho_\ast\in[\rho_\ast,R]$ such that the cylinder $Q_{\widetilde\rho_\ast}^{(\lambda_{z_\ast;\rho_\ast})}(z_\ast)$ is intrinsic, i.e.
\begin{align}\label{control-lambda-2}
	\lambda_{z_\ast;\rho_\ast}^{\frac{ \boldsymbol\lambda_\sfr}{2}}
	=
	\frac{1}{|Q_{\widetilde \rho_\ast}|}
	\iint_{Q_{\widetilde \rho_\ast}^{(\lambda_{z_\ast;\rho_\ast})}(z_\ast)}
	\frac{\big|u-(u)_{z_\ast,\widetilde \rho_\ast}^{(\lambda_{z_\ast;\rho_\ast})}\big|^{\sfr}}{\widetilde \rho_\ast^{\sfr}} \,\dx\dt.
\end{align}
We distinguish between the cases $\widetilde \rho_\ast\le \frac{R}{9}$ and $\widetilde \rho_\ast> \frac{R}{9}$. We begin with the latter. Using \eqref{control-lambda-2}, Lemma~\ref{lem:mean}, the definitions of $\lambda_o$ and $\lambda_{z;\rho}\ge\lambda_o$, we estimate
\begin{align*}
	\lambda_{z_\ast;\rho_\ast}^{\frac{ \boldsymbol\lambda_\sfr}{2}}
	&\le
    2^\sfr
	\Big(\frac{4R}{\widetilde \rho_\ast}\Big)^{\sfr} \frac{|Q_{4R}|}{|Q_{\widetilde \rho_\ast}|}
	\biint_{Q_{4R}}
	\frac{\big| u-(u)_{4R}\big|^{\sfr}}{(4R)^{\sfr}} \,\dx\dt \\
	&\le
    2^\sfr
	\Big(\frac{4 R}{\widetilde \rho_\ast}\Big)^{N+2+\sfr} 
	\lambda_o^{\frac{ \boldsymbol\lambda_\sfr}{2}}\\
	&\le
    2^\sfr
	36^{N+2+\sfr} \lambda_{z;\rho}^{\frac{ \boldsymbol\lambda_\sfr}{2}},
\end{align*}
which can be rewritten as
\begin{align*}
	\lambda_{z_\ast;\rho_\ast}
	\le
    2^\frac{2\sfr}{\boldsymbol\lambda_\sfr}
	36^{\frac{2}{\boldsymbol\lambda_\sfr}(N+2+\sfr)}\,\lambda_{z;\rho}\,.
\end{align*}
This yields \eqref{control-lambda-1} in the present case. 

It remains to treat the case $\widetilde \rho_\ast\le \frac{R}{9}$. 
For that, we claim the inclusion
\begin{equation}\label{incl-cyl}
  Q_{\widetilde \rho_\ast}^{(\lambda_{z_\ast;\rho_\ast})}(z_\ast)
  \subseteq 
  Q_{9\widetilde \rho_\ast}^{(\lambda_{z;9\widetilde \rho_\ast})}(z).
\end{equation}
First, since $\widetilde\rho_\ast\ge\rho_\ast$ and $|t-t_\ast|<(4\rho)^2+(4\rho_\ast)^2\le80\,\rho_\ast^2\le80\,\widetilde\rho_\ast^2$, we have
$
    \Lambda_{\widetilde\rho_\ast}(t_\ast)\subseteq \Lambda_{\mu\,\widetilde\rho_\ast}(t)$.
Moreover,
\begin{equation}\label{x-x_ast}
	|x-x_\ast|
    \le 
    4\lambda_{z;\rho}^{\frac{p-2}{2}}\rho
	+
	4\lambda_{z_\ast;\rho_\ast}^{\frac{p-2}{2}}\rho_\ast.
\end{equation}
We may assume $\lambda_{z;\rho}\le \lambda_{z_\ast;\rho_\ast}$, since the comparison estimate \eqref{control-lambda-1} is trivial otherwise. Using the monotonicity of $\rho\mapsto \lambda_{z;\rho}$ and the chain $\rho\le 2\rho_\ast\le 2\widetilde\rho_\ast\le 9\widetilde\rho_\ast$, we obtain
$$
\lambda_{z_\ast;\rho_\ast}\ge \lambda_{z;\rho}\ge \lambda_{z;9\widetilde\rho_\ast}.
$$
Combining the last two displays gives
$$
    \lambda_{z_\ast;\rho_\ast}^{\frac{p-2}{2}}\widetilde\rho_\ast + |x-x_\ast|
    \le 
    5\lambda_{z_\ast;\rho_\ast}^{\frac{p-2}{2}} \widetilde\rho_\ast
    +
    4\lambda_{z;\rho}^{\frac{p-2}{2}}\rho
    \le
    9\lambda_{z;9\widetilde\rho_\ast}^{\frac{p-2}{2}}\widetilde\rho_\ast,
$$
and hence
$$
    B_{\widetilde\rho_\ast}^{(\lambda_{z_\ast;\rho_\ast})}(x_\ast)
  \subseteq 
  B_{9\widetilde\rho_\ast}^{(\lambda_{z;9\widetilde\rho_\ast})}(x).
$$
Together with the time–interval inclusion this proves the cylinder inclusion \eqref{incl-cyl}.
Using \eqref{control-lambda-2}, Lemma \ref{lem:mean}, the inclusion \eqref{incl-cyl}, and Lemma \ref{lem:crazy-cylinders} d) with
$\rho=9\rho_\ast$ and $s=9\widetilde \rho_\ast$, we obtain 
\begin{align*}
    \lambda_{z_\ast;\rho_\ast}^{\frac{ \boldsymbol\lambda_\sfr}{2}}
	&=
	\frac{1}{|Q_{\widetilde \rho_\ast}|}
	\iint_{Q_{\widetilde \rho_\ast}^{(\lambda_{z_\ast;\rho_\ast})}(z_\ast)}
	\frac{\big|u-(u)_{z_\ast,\widetilde \rho_\ast}^{(\lambda_{z_\ast;\rho_\ast})}\big|^{\sfr}}{\widetilde \rho_\ast^{\sfr}} \,\dx\dt\\
    &\le
    \frac{18^\sfr}{|Q_{\widetilde \rho_\ast}|}
    \iint_{Q_{9\widetilde \rho_\ast}^{(\lambda_{z;
    9\rho_\ast})}(z)}
    \frac{\big|u-(u)_{z,9\widetilde \rho_\ast}^{(\lambda_{z;9\rho_\ast})}\big|^{\sfr}}{(9\widetilde \rho_\ast)^{\sfr}} \,\dx\dt\\
    &\le
    36^\sfr \frac{\big|Q_{9\widetilde \rho_\ast}^{(\lambda_{z;
    9\rho_\ast})}\big|}{|Q_{\widetilde \rho_\ast}|}
    \lambda_{z;9\rho_\ast}^\frac{p\sfr}{2}
    =
    36^\sfr9^{N+2} \lambda_{z;9\rho_\ast}^\frac{\boldsymbol{\lambda}_\sfr}{2}.
\end{align*}
Hence
\begin{equation*}
    \lambda_{z_\ast;\rho_\ast}
    \le
    36^\frac{2\sfr}{\boldsymbol{\lambda}_\sfr}
    9^{\frac{2}{\boldsymbol{\lambda}_\sfr}(N+2)}
    \lambda_{z;9\rho_\ast},
\end{equation*}
which in view of the inequality $\lambda_{z;9\rho_\ast}\le \lambda_{z;\rho}$ that follows from the monotonicity of $\rho\mapsto\lambda_{z;\rho}$ proves \eqref{control-lambda-1} also in the remaining case.

The preliminary bound~\eqref{control-lambda-1} will now allow us to show the inclusion
$$
Q=Q_{4\rho}^{(\lambda_{z;\rho})}(z)\;\subseteq\;
Q_{\hat c\,\rho_\ast}^{(\lambda_{z_\ast;\rho_\ast})}(z_\ast)=\widehat Q^\ast,
$$
which yields \eqref{covering}. The time–interval inclusion $\Lambda_{4\rho}(t)\subseteq\Lambda_{\hat c\rho_\ast}(t_\ast)$ has already been shown, so it suffices to treat the spatial part. From \eqref{x-x_ast}, \eqref{control-lambda-1}, and $\rho\le2\rho_\ast$, we conclude
\begin{align*}
    \lambda_{z;\rho}^{\frac{p-2}{2}}\,4\rho + |x-x_\ast|
    &\le
    2\,\lambda_{z;\rho}^{\frac{p-2}{2}}\,4\rho
    \;+\;
    \lambda_{z_\ast;\rho_\ast}^{\frac{p-2}{2}}\,4\rho_\ast\\
    &\le
    4\Big[\,4\cdot 72^{\frac{2-p}{\boldsymbol\lambda_\sfr}(N+2+\sfr)}+1\,\Big]\,
    \lambda_{z_\ast;\rho_\ast}^{\frac{p-2}{2}}\,\rho_\ast\\
    &\le
    \lambda_{z_\ast;\rho_\ast}^{\frac{p-2}{2}}\,\hat c\,\rho_\ast,
\end{align*}
for a structural constant $\hat c=\hat c(N,p,\sfr)$. Hence
$$
B_{4\rho}^{(\lambda_{z;\rho})}(x)\;\subseteq\;B_{\hat c\rho_\ast}^{(\lambda_{z_\ast;\rho_\ast})}(x_\ast),
$$
and together with the time inclusion we obtain $Q\subset \widehat Q^\ast$. This proves \eqref{covering} and completes the argument.
\end{proof}

\subsection{Selection of critical intrinsic radii}\label{SS:selection-radii}
We localize the super-level sets of $|Du|$ by a stopping–time selection on the intrinsic scales $\lambda_{z_o;s}$: for a fixed threshold $\lambda>\lambda_{o}$, each Lebesgue point $z_{o}$ corresponds to a maximal radius at which the normalized energy first reaches $\lambda^{p}$. The Vitali covering of these critical cylinders furnishes the geometric backbone for the ensuing good–$\lambda$ estimate and higher–integrability. 

Fix $R>0$ and set
\begin{align}\label{Eq:def-lm0-new}
 	\lambda_o
 	&:=
 	1+
    \bigg[\biint_{Q_{4R}} 
 	\frac{|u-(u)_{4R}|^{\sfr}}{(4 R)^{\sfr}} \,\dx\dt \bigg]^{\frac{2}{\boldsymbol\lambda_\sfr}} \nonumber\\
    &\phantom{:=\,1\,}+
    \Bigg[\biint_{Q_{4R}}
    |Du|^p \,\dx\dt +
 	\bigg[ \biint_{Q_{4R}} \Big[|\sfF|^{qp}
    +|R\,\sff|^{qp'}\Big]\,\dx\dt \bigg]^{\frac1q} 
    \Bigg]^{\frac{\sfr}{\boldsymbol\lambda_\sfr}}.
\end{align}
For $\lambda>\lambda_{o}$ and $\rho\in(0,2R]$ define the \emph{super-level set} of $|Du|$ by
\[
	\sfE(\rho,\lambda)
	:=
	\Big\{z\in Q_{\rho}:
	\mbox{$z$ is a Lebesgue point of $|Du|$ and 
	$|Du|(z) > \lambda$}\Big\}.
\]
Lebesgue points here are understood with respect to the intrinsic cylinders constructed in Lemma $\ref{lem:crazy-cylinders}$. By the Vitali–type covering Lemma $\ref{lem:vitali}$ and the standard differentiation theorem (see, e.g., \cite[2.9.1]{Federer}), $\mathcal L^{N+1}$-a.e.~point is a Lebesgue point with respect to this family.
Fix radii $R\le R_{1}<R_{2}\le 2R$. For any $z_{o}\in Q_{R_{1}}$, $\kappa\ge1$, and $\rho\in(0,\,R_{2}-R_{1}]$ we have
\begin{equation*}
  	Q_\rho^{(\kappa)}(z_o)
	\subseteq 
	Q_{R_2}
	\subseteq 
	Q_{2R}.
\end{equation*}
In what follows the level parameter $\lambda$ will be restricted by 
\begin{equation}\label{choice_lambda-new}
	\lambda
	>
	\lambda_o^{\frac\sfr2},
	\quad\mbox{where}
	\quad
	\sfB
	:=
	\Big(\frac{32\hat c R}{R_2-R_1}\Big)^{\frac{2\sfr(N+p)}{p\boldsymbol\lambda_\sfr}}
	>1,
\end{equation}
where $\hat c=\hat c(N,p,\sfr)$ is the structural constant from the Vitali–type covering Lemma \ref{lem:vitali}. 
Fix $z_{o}\in \sfE(R_{1},\lambda)$ and, for $s\in(0,R]$, abbreviate $\lambda_{s}:=\lambda_{z_{o};s}$ as in Lemma $\ref{lem:crazy-cylinders}$. By the definition of $\sfE
(R_{1},\lambda)$ (and the Lebesgue differentiation property with respect to intrinsic cylinders),
\begin{align}\label{larger-lambda-new}
	\liminf_{s\downarrow 0} &
	\Bigg[ 
    \biint_{Q_{s}^{(\lambda_{s})}(z_o)} 
	|Du|^p \,\dx\dt + 
    \bigg[ \biint_{Q_{s}^{(\lambda_{s})}(z_o)} \Big[|\sfF|^{qp}
    +\big|\lambda_s^{\frac{p-2}{2}}s\,\sff\big|^{qp'}\Big] \,\dx\dt\bigg]^{\frac1q} \Bigg] \nonumber \\
	&\ge
	|Du|^p(z_o) 
	>
	\lambda^p.
\end{align}
Fix $s$ with
\begin{align}\label{radius-s-new}
	\frac{R_2-R_1}{\hat c}\le s\le R.
\end{align}
By the definition of $\lambda_o$, Lemma~\ref{lem:crazy-cylinders}\,b), and assumption \eqref{radius-s-new}, we estimate
\begin{align}\label{smaller-lambda-new}
	&\biint_{Q_{s}^{(\lambda_{s})}(z_o)} 
	|Du|^p \,\dx\dt + 
    \bigg[ \biint_{Q_{s}^{(\lambda_{s})}(z_o)} \Big[|\sfF|^{qp}
    +\big|\lambda_s^{\frac{p-2}{2}}s\,\sff\big|^{qp'}\Big] \,\dx\dt\bigg]^{\frac1q} \nonumber\\
	&\qquad\le
	\frac{|Q_{4R}|}{|Q_{s}^{(\lambda_s)}|}
	\biint_{Q_{4R}} |Du|^p \,\dx\dt + 
    \bigg[ \frac{|Q_{4R}|}{|Q_{s}^{(\lambda_s)}|}
    \biint_{Q_{4R}} \Big[|\sfF|^{qp}
    +|R\,\sff|^{qp'}\Big] \,\dx\dt\bigg]^{\frac1q} \nonumber\\
	&\qquad\le
	\frac{|Q_{4R}|}{|Q_{s}^{(\lambda_s)}|}
	\Bigg[ 
    \biint_{Q_{4R}} |Du|^p \,\dx\dt + 
    \bigg[ 
    \biint_{Q_{4R}} \Big[|\sfF|^{qp}
    +|R\,\sff|^{qp'}\Big] \,\dx\dt\bigg]^{\frac1q} \Bigg]\nonumber\\
	&\qquad\le
	\frac{|Q_{4R}|}{|Q_{s}|}\,
	\lambda_s^{\frac{N(2-p)}{2}}
	\lambda_o^{\frac{\boldsymbol\lambda_\sfr}{\sfr}}
    \le
    \Big(\frac{4R}{s}\Big)^{N+2}\,
	\lambda_s^{\frac{N(2-p)}{2}}
	\lambda_o^{\frac{\boldsymbol\lambda_\sfr}{2}}
     \nonumber\\
    &\qquad\le
    \Big(\frac{4R}{s}\Big)^{N+2} 
    \bigg[4^\frac{2\sfr}{\boldsymbol\lambda_\sfr}
    \Big(\frac{R}{s}\Big)^{\frac{2}{\boldsymbol\lambda_\sfr}(N+2+\sfr)}
    4^{\frac{2}{\boldsymbol\lambda_\sfr}
    (N+2+2\sfr)}
    \lambda_o
    \bigg]^\frac{N(2-p)}{2}\lambda_o^\frac{\boldsymbol\lambda_\sfr}{2}
    \nonumber\\
    &\qquad=
    4^\frac{4\sfr N(2-p)}{\boldsymbol\lambda_\sfr}
    \Big(\frac{4R}{s}\Big)^{N+2+(N+2+\sfr)\frac{N(2-p)}{\boldsymbol\lambda_\sfr}}\lambda_o^{\frac{N(2-p)}{2}+\frac{\boldsymbol\lambda_\sfr}{2}} \nonumber
    \nonumber\\
    &\qquad=
    8^\frac{2\sfr N(2-p)}{\boldsymbol\lambda_\sfr}
    \Big(\frac{4R}{s}\Big)^\frac{2\sfr (N+p)}{\boldsymbol\lambda_\sfr}\lambda_o^{\frac{p\sfr}{2}}
    \nonumber\\
    &\qquad\le
    \Big(\frac{32\hat cR}{R_2-R_1}\Big)^\frac{2\sfr (N+p)}{\boldsymbol\lambda_\sfr}\lambda_o^{\frac{p\sfr}{2}}
    =
	\big(\sfB \lambda_o^{\frac{\sfr }2}\big)^p
	<
	\lambda^{p},
\end{align}
where we also used the algebraic identity 
\begin{align*}
    & N+2+ (N+2+\sfr)\frac{N(2-p)}{\boldsymbol\lambda_\sfr}\\
    &\qquad=
    \frac{1}{\boldsymbol\lambda_\sfr} 
    \big[(N+2)\big(N(p-2)+p\sfr\big) + N(N+2+\sfr)(2-p)
    \big]\\
    &\qquad=
    \frac{1}{\boldsymbol\lambda_\sfr} 
    \big[(N+2)p\sfr + N\sfr(2-p)
    \big]\\
    &\qquad=
    \frac{2\sfr(N+p)}{\boldsymbol\lambda_\sfr},
\end{align*}
and the choice of $\sfB$ in \eqref{choice_lambda-new}.
By continuity of $s\mapsto\lambda_{s}$ and absolute continuity of the integral, the left–hand side of \eqref{smaller-lambda-new} depends continuously on $s$. Hence, combining \eqref{larger-lambda-new} and \eqref{smaller-lambda-new}, there exists a maximal radius $0<\rho_{z_o}<\tfrac{R_{2}-R_{1}}{\hat c}$ at which equality is attained, i.e.
\begin{equation}\label{=lambda-new}
	\biint_{Q_{\rho_{z_o}}^{(\lambda_{\rho_{z_o}})}(z_o)} 
	|Du|^p \,\dx\dt + \bigg[\biint_{Q_{\rho_{z_o}}^{(\lambda_{\rho_{z_o}})}(z_o)} \Big[|\sfF|^{qp}
    +\big|\lambda_{\rho_{z_o}}^{\frac{p-2}{2}}\rho_{z_o}\,\sff\big|^{qp'}\Big] \dx\dt\bigg]^{\frac1q}
	=
	\lambda^p.
\end{equation}
By the maximality of $\rho_{z_o}$, it follows that for every $s\in(\rho_{z_o},R]$,
\begin{equation*}
	\biint_{Q_{s}^{(\lambda_{s})}(z_o)} 
	|Du|^p \,\dx\dt + 
    \bigg[\biint_{Q_{s}^{(\lambda_{s})}(z_o)} \Big[|\sfF|^{qp}
    +\big|\lambda_{s}^{\frac{p-2}{2}}s\,\sff\big|^{qp'}\Big] \dx\dt\bigg]^{\frac1q} 
	<
	\lambda^p.
\end{equation*}
By the monotonicity of $s\mapsto\lambda_s$
and the scale control in Lemma~\ref{lem:crazy-cylinders} b), we obtain for all $\rho_{z_o}\le s<\sigma\le R$ that 
\begin{align}\label{<lambda-new}
	\biint_{Q_{\sigma}^{(\lambda_{s})}(z_o)} &
	|Du|^p \,\dx\dt + \bigg[\biint_{Q_{\sigma}^{(\lambda_{s})}(z_o)} \Big[|\sfF|^{qp}
    +\big|\lambda_{s}^{\frac{p-2}{2}}\sigma\,\sff\big|^{qp'}\Big] \dx\dt\bigg]^{\frac1q} \nonumber\\
	&\le
	\Big(\frac{\lambda_{s}}{\lambda_\sigma}\Big)^{\frac{N(2-p)}{2}}
    \Bigg[ 
	\biint_{Q_{\sigma}^{(\lambda_{\sigma})}(z_o)} 
	|Du|^p \,\dx\dt\nonumber\\
    &\qquad\qquad\qquad\quad+ 
    \bigg[ \biint_{Q_{\sigma}^{(\lambda_{\sigma})}(z_o)} \Big[|\sfF|^{qp}
    +\big|\lambda_{\sigma}^{\frac{p-2}{2}}\sigma\,\sff\big|^{qp'}\Big] \dx\dt\bigg]^{\frac1q} 
    \Bigg] \nonumber\\
	&<
	4^{\frac{N(2-p)\sfr}{\boldsymbol\lambda_\sfr}}
    \Big(\frac{\sigma}{s}\Big)^{\frac{N(2-p)}{\boldsymbol\lambda_\sfr}(N+2+\sfr)}\,
	\lambda^{p}.
\end{align} 
Finally, recall that the cylinders were chosen so that their $\tfrac{\hat c}{4}$-enlargements remain inside $Q_{R_2}$; in particular,
$$
  Q_{\hat c\,\rho_{z_o}}^{(\lambda_{\rho_{z_o}})}(z_o)
  \subseteq
  Q_{\hat c\,\rho_{z_o}}(z_o)
  \subseteq
  Q_{R_2}.
$$

\subsection{Reverse Hölder inequality on intrinsic cylinders}

A reverse Hölder estimate is now obtained at the intrinsic stopping scale. For $z_{o}\in \sfE (R_{1},\lambda)$ with stopping radius $\rho_{z_{o}}$ and parameter $\lambda_{\rho_{z_{o}}}$, the sub– and super–intrinsic couplings verified above  
allow us to invoke Proposition~\ref{prop:RevH} on $Q_{2\rho_{z_{o}}}^{(\lambda_{\rho_{z_{o}}})}(z_{o})$. This yields a reverse Hölder inequality for $Du$ on intrinsic cylinders, which will serve as the input for Gehring’s self–improvement and the ensuing higher integrability.

Fix $\lambda$ as in \eqref{choice_lambda-new} and let $z_o\in \sfE(R_1,\lambda)$.
Denote by $\rho_{z_o}$ the stopping radius from \eqref{=lambda-new}, and set $\lambda_{\rho_{z_o}}:=\lambda_{z_o;\rho_{z_o}}$
(cf.~Lemma~\ref{lem:crazy-cylinders}).
Our aim is to establish a reverse Hölder inequality for $Du$ on the intrinsic cylinder $Q_{2\rho_{z_o}}^{(\lambda_{\rho_{z_o}})}(z_o)$. To verify the assumptions \eqref{sub-intr-rh} and \eqref{super-intr-rh} of Proposition~\ref{prop:RevH}, we first observe that Lemma~\ref{lem:crazy-cylinders} d) with $s=2\rho_{z_o}$ implies
\begin{equation*}
  \biint_{Q_{2\rho_{z_o}}^{(\lambda_{\rho_{z_o}})}(z_o)} 
  \frac{\big|u-(u)_{z_o;2\rho_{z_o}}^{(\lambda_{\rho_{z_o}})}\big|^{\sfr}}{(2\rho_{z_o})^{\sfr}} \dx\dt
  \le 
  2^{\sfr}\lambda_{\rho_{z_o}}^\frac{p\sfr}{2}.
\end{equation*}
Thus, the sub–intrinsic coupling \eqref{sub-intr-rh} holds on $Q_{2\rho_{z_o}}^{(\lambda_{\rho_{z_o}})}(z_o)$ with $\sfK=2^{\sfr}$. 

To verify \eqref{super-intr-rh}, we first claim that $\lambda_{\rho_{z_o}}$ is bounded in terms of $\lambda$.
Indeed, by Lemma~\ref{lem:crazy-cylinders} e), either an explicit upper bound holds:
$$
    \lambda_{\rho_{z_o}}
    \le 
    4^{\frac{2}{\boldsymbol\lambda_\sfr}(N+2+3\sfr)}\lambda_{o} 
    \le 
    \sfB\lambda_{o}^{\frac\sfr2}
    <
    \lambda,
$$
where in the second inequality we used the definition of the constant $\sfB$ from \eqref{choice_lambda-new}. Alternatively, there exists a radius $\widetilde\rho\in\big[\rho_{z_o},\tfrac{R}{2}\big]$ such that
$$
    \lambda_{\rho_{z_o}}^{\frac{p\sfr}{2}}
    =
    \lambda_{\widetilde\rho}^{\frac{p\sfr}{2}}
    =
    \biint_{Q_{\widetilde\rho}^{(\lambda_{\widetilde\rho})}(z_o)}
    \frac{\big|u-(u)_{z_o;\widetilde\rho}^{(\lambda_{\widetilde\rho})}\big|^{\sfr}}{\widetilde\rho^{\sfr}}\,\dx\dt.
$$
The latter case also yields a bound similar to the former case in view of Lemma~\ref{lem:sub-intr-2} applied on $Q_{\widetilde\rho}^{(\lambda_{\widetilde\rho})}(z_o)$. In fact, the sub–intrinsic condition \eqref{sub-intr} holds with $\sfK=2^{\sfr}$ by Lemma~\ref{lem:crazy-cylinders} d), while the super–intrinsic identity \eqref{super-intr} with $\sfK=1$ follows from the preceding display. Hence, Lemma~\ref{lem:sub-intr-2} gives
\begin{align*}
	\lambda_{\rho_{z_o}}^p
    =
    \lambda_{\widetilde\rho}^p
    &\le
	C \biint_{Q_{\widetilde\rho}^{(\lambda_{\widetilde\rho})}(z_o)} |Du|^p \,\dx\dt +
    C \bigg[\biint_{Q_{2\widetilde\rho}^{(\lambda_{\widetilde\rho})}(z_o)} \Big[| \sfF|^{qp} + \big|\lambda_{\widetilde\rho}^{\frac{p-2}{2}}\widetilde\rho\, \sff\big|^{qp'}\Big] \dx\dt
	\bigg]^{\frac{1}{q}}.
\end{align*}
Moreover, using \eqref{<lambda-new} with $s=\widetilde\rho$ and $\sigma=2\widetilde\rho\in(s,R]$, we obtain 
$$
\lambda_{\rho_{z_o}}^{p}
\le
C2^{N+2+\frac{N(2-p)}{\boldsymbol\lambda_\sfr}(N+2+3\sfr)}\,\lambda^{p}.
$$
In either alternative we thus have $\lambda_{\rho_{z_o}}\le C\lambda$ as claimed.
Combining this with \eqref{=lambda-new} yields
\begin{align*}
  	\lambda_{\rho_{z_o}}^p
  	\le
    C\lambda^p
    &=
  	C \Bigg[\biint_{Q_{\rho_{z_o}}^{(\lambda_{\rho_{z_o}})}(z_o)} 
	|Du|^p \,\dx\dt \\
    &\qquad\quad+ \bigg[\biint_{Q_{\rho_{z_o}}^{(\lambda_{\rho_{z_o}})}(z_o)} \Big[|\sfF|^{qp}
    +\big|\lambda_{\rho_{z_o}}^{\frac{p-2}{2}}\rho_{z_o}\,\sff\big|^{qp'}\Big] \,\dx\dt\bigg]^{\frac1q} \Bigg].
\end{align*}
Since enlarging the cylinders on the right to $Q_{2\rho_{z_o}}^{(\lambda_{\rho_{z_o}})}(z_o)$ only increases the integrals, the super–intrinsic condition \eqref{super-intr-rh} holds on
$Q_{2\rho_{z_o}}^{(\lambda_{\rho_{z_o}})}$ with a  constant $\sfK\equiv \sfK(N,p,C_o,C_1,q,\sfr)$.
In any case, all hypotheses of Proposition~\ref{prop:RevH} are met. Hence the proposition yields the reverse Hölder estimate
\begin{align}\label{rev-hoelder}
	\nonumber\biint_{Q_{2\rho_{z_o}}^{(\lambda_{\rho_{z_o}})}(z_o)}  
	|Du|^p \,\dx\dt 
	&\le
	C\bigg[\biint_{Q_{4\rho_{z_o}}^{(\lambda_{\rho_{z_o}})}(z_o)} 
	|Du|^{\theta} \,\dx\dt \bigg]^{\frac{p}{\theta}}\\\nonumber 
	&\phantom{\le\,}+
    C\bigg[ \biint_{Q_{4\rho_{z_o}}^{(\lambda_{\rho_{z_o}})}(z_o)} \Big[| \sfF|^{qp} + \big|\lambda_{\rho_{z_o}}^{\frac{p-2}{2}}\rho_{z_o} \sff\big|^{qp'}\Big] \,\dx\dt\bigg]^{\frac1q} \\\nonumber
    &\le
	C\bigg[\biint_{Q_{4\rho_{z_o}}^{(\lambda_{\rho_{z_o}})}(z_o)} 
	|Du|^{\theta} \,\dx\dt \bigg]^{\frac{p}{\theta}}\\
    &\phantom{\le\,}+
	C\bigg[ \biint_{Q_{4\rho_{z_o}}^{(\lambda_{\rho_{z_o}})}(z_o)} \Big[| \sfF|^{qp} + \big|R\, \sff\big|^{qp'}\Big] \,\dx\dt\bigg]^{\frac1q},
\end{align}
where $\theta=\frac{Np}{N+2}<1$
and  $C=C(N,p,C_o,C_1,q,\sfr)$.

\subsection{Weak-type estimate for the super-level sets of \texorpdfstring{$|Du|$}{|Du|}}
So far, for every $\lambda$ as in \eqref{choice_lambda-new} and every $z_o\in \sfE(R_1,\lambda)$, we have constructed a cylinder $Q_{\rho_{z_o}}^{(\lambda_{z_o;\rho_{z_o}})}(z_o)$ with $Q_{\hat c\,\rho_{z_o}}^{(\lambda_{z_o;\rho_{z_o}})}(z_o)\subseteq Q_{R_2}$ that satisfies the stopping time identity \eqref{=lambda-new}, the strict decay for larger radii \eqref{<lambda-new}, and the reverse Hölder estimate \eqref{rev-hoelder}. These ingredients permit a Vitali-type covering argument yielding a reverse Hölder (weak-type) bound for the distribution function of $|Du|^{p}$. The argument proceeds as follows.
We define the {\em super-level set of the source terms} by
\begin{align*}
    &\sfG(\rho,\lambda)\\
	&\quad:=
	\Big\{z\in Q_{\rho}: 
	\mbox{$z$ is a Lebesgue point of $| \sfF| + |R\,\sff|^{\frac{1}{p-1}}$ and 
	$\big[| \sfF| + |R\,\sff|^{\frac{1}{p-1}}\big](z)>\lambda$}\Big\}.
\end{align*}
As before, Lebesgue points are understood with respect to the intrinsic cylinders constructed in \S~\ref{sec:cylinders}; by Lemma~\ref{lem:vitali} and the differentiation theorem, $\mathcal L^{N+1}$-a.e. point of $Q_\rho$ is such a Lebesgue point.
Applying \eqref{=lambda-new} and \eqref{rev-hoelder} we obtain, for every $z_o\in \sfE (R_1,\lambda)$ and every $\eta\in (0,1)$,
\begin{align*}
	\lambda^{p}
	&=
	\biint_{Q_{\rho_{z_o}}^{(\lambda_{\rho_{z_o}})}(z_o)} 
	|Du|^p \,\dx\dt + \bigg[ \biint_{Q_{\rho_{z_o}}^{(\lambda_{\rho_{z_o}})}(z_o)} \Big[|\sfF|^{qp}
    +\big|\lambda_{\rho_{z_o}}^{\frac{p-2}{2}}\rho_{z_o}\,\sff\big|^{qp'}\Big] \,\dx\dt\bigg]^{\frac1q} \\
	&\le
	C\bigg[\biint_{Q_{4\rho_{z_o}}^{(\lambda_{\rho_{z_o}})}(z_o)} 
	|Du|^{\theta} \,\dx\dt \bigg]^{\frac{p}{\theta}} +
	C\bigg[ \biint_{Q_{4\rho_{z_o}}^{(\lambda_{\rho_{z_o}})}(z_o)} \Big[| \sfF|^{qp} + |R\,\sff|^{qp'}\Big] \,\dx\dt\bigg]^{\frac1q} \\
	&\le
	C\,\eta^p\lambda^p +
	C\sfI^{\frac{p}{\theta}}+
	C\sfI\sfI^{\frac1q} ,
\end{align*}
for some $C=C(N,p,C_o,C_1,q,\sfr)$, where we again abbreviated $\lambda_{\rho_{z_o}}=\lambda_{z_o;\rho_{z_o}}$ and 
\begin{align*}
    \sfI
    &:=
    \frac{1}{\big|Q_{4\rho_{z_o}}^{(\lambda_{\rho_{z_o}})}(z_o)\big|}
	\iint_{Q_{4\rho_{z_o}}^{(\lambda_{\rho_{z_o}})}(z_o)\cap \sfE(R_2,\eta\lambda)} 
	|Du|^{\theta} \,\dx\dt,
\end{align*}
\begin{align*}
    \sfI\sfI
    &:=
    \frac{1}{\big|Q_{4\rho_{z_o}}^{(\lambda_{\rho_{z_o}})}(z_o)\big|}
	\iint_{Q_{4\rho_{z_o}}^{(\lambda_{\rho_{z_o}})}(z_o)\cap \sfG (R_2,\eta\lambda)} 
	\Big[| \sfF|^{qp} + |R\,\sff|^{qp'}\Big] \,\dx\dt.
\end{align*} 
Moreover, to get the last estimate, we decomposed the averages over $Q_{4\rho_{z_o}}^{(\lambda_{\rho_{z_o}})}$ according to the sets $\sfE (R_{2},\eta\lambda)$ and $Q_{4\rho_{z_o}}^{(\lambda_{\rho_{z_o}})}\setminus \sfE(R_{2},\eta\lambda)$, and analogously for $\sfG$ on the source terms; the sub-level contribution is bounded by $(\eta\lambda)^{p}$ and can be absorbed after choosing $\eta>0$ small enough. 

In order to control $\sfI$, we apply Hölder’s inequality and employ \eqref{<lambda-new} with $s=\rho_{z_o}$ and $\sigma=4\rho_{z_o}$. This yields
\begin{align*}
	\sfI^{\frac{p}{\theta}-1}
	&\le
	\bigg[\biint_{Q_{4\rho_{z_o}}^{(\lambda_{\rho_{z_o}})}(z_o)} 
	|Du|^{p} \dx\dt \bigg]^{1-\frac{\theta}{p}}
	\le
	C\,\lambda^{p-\theta}.
\end{align*}
For $\sfI\sfI$,   we use Young's inequality with conjugate exponents $q'=\frac{q}{q-1}$ and $q$. This gives
\begin{align*}
	C\sfI\sfI^{\frac1q}
    &\le 
    \tfrac14 \lambda^{p} +
    \frac{C\,\lambda^{-p(q-1)}}{\big|Q_{4\rho_{z_o}}^{(\lambda_{\rho_{z_o}})}(z_o)\big|} 
	\iint_{Q_{4\rho_{z_o}}^{(\lambda_{\rho_{z_o}})}(z_o)\cap \sfG (R_2,\eta\lambda)} 
	\Big[| \sfF|^{qp} + |R\,\sff|^{qp'}\Big] \,\dx\dt .
\end{align*}
Choose and fix a parameter
\begin{equation}\label{choice:eta}
\eta=\eta(N,p,C_o,C_1,q,\sfr)\in(0,1)
\quad\text{such that}\quad
\eta^{p}=\frac{1}{4C},
\end{equation}
where $C=C(N,p,C_o,C_1,q,\sfr)$ is the constant from the previous estimate. With this choice, the $\lambda^{p}$-terms are absorbed into the left-hand side. Multiplying the resulting inequality by $\big|Q_{4\rho_{z_o}}^{(\lambda_{\rho_{z_o}})}(z_o)\big|$ yields
\begin{align*}
	\lambda^{p}\big|Q_{4\rho_{z_o}}^{(\lambda_{\rho_{z_o}})}(z_o)\big|
	&\le
	C\iint_{Q_{4\rho_{z_o}}^{(\lambda_{\rho_{z_o}})}(z_o)\cap \sfE (R_2,\eta\lambda)} 
	\lambda^{p-\theta}|Du|^{\theta} \,\dx\dt \\
	&\quad+
	C
	\iint_{Q_{4\rho_{z_o}}^{(\lambda_{\rho_{z_o}})}(z_o)\cap \sfG (R_2,\eta\lambda)} 
	\lambda^{-p(q-1)} \Big[| \sfF|^{qp} + |R\,\sff|^{qp'}\Big] \,\dx\dt,
\end{align*}
with $C=C(N,m,\nu,L,p,r)$. In particular, $\eta$ is fixed once and for all in terms of the structural parameters only, independent of $z_o$, $\lambda$, and $R_1,R_2$. 

Applying \eqref{<lambda-new} with $s=\rho_{z_o}$ and $\sigma=\hat c\,\rho_{z_o}$, and recalling that $\hat c=\hat c(N,p,\sfr)$, we obtain
\begin{align*}
	\lambda^{p}
    \ge
    \frac{1}{C}
	\biint_{Q_{\hat c\rho_{z_o}}^{(\lambda_{\rho_{z_o}})}(z_o)} 
	|Du|^p \,\dx\dt.
\end{align*}
Combining the preceding bounds, we obtain
\begin{align}\label{level-est}
	\iint_{Q_{\hat c\rho_{z_o}}^{(\lambda_{\rho_{z_o}})}(z_o)} 
	|Du|^p \,\dx\dt 
	&\le
    C
    \iint_{Q_{4\rho_{z_o}}^{(\lambda_{\rho_{z_o}})}(z_o)\cap \sfE (R_2,\eta\lambda)} \!\!
	\lambda^{p-\theta}|Du|^{\theta} \dx\dt \nonumber\\
	&\quad +
	C
	\iint_{Q_{4\rho_{z_o}}^{(\lambda_{\rho_{z_o}})}(z_o)\cap \sfG (R_2,\eta\lambda)} \!\!
	\lambda^{-p(q-1)} \Big[| \sfF|^{qp} + |R\,\sff|^{qp'}\Big]\dx\dt
\end{align}
for a constant $C=C(N,p,C_o,C_1,q,\sfr)$.
Since \eqref{level-est} holds for every center \(z_o\in \sfE(R_{1},\lambda)\),  
the super-level set \(\sfE(R_{1},\lambda)\) can be covered by the family of cylinders  
\[
\mathcal{F}
:=
\Big\{
Q_{4\rho_{z_o}}^{(\lambda_{z_o;\rho_{z_o}})}(z_o)
:\ z_o\in \sfE(R_{1},\lambda)
\Big\},
\]
whose elements are contained in \(Q_{R_{2}}\) and on each of which \eqref{level-est} is valid.
By the Vitali-type covering Lemma~\ref{lem:vitali}, there exists a countable disjoint subfamily
$$
\Big\{\,Q_{4\rho_{z_i}}^{(\lambda_{z_i;\rho_{z_i}})}(z_i)\,\Big\}_{i\in\mathbb N}\subseteq \mathcal F
$$
with
$$
\sfE(R_{1},\lambda)
\subseteq
\bigcup_{i=1}^{\infty} Q_{\hat c\,\rho_{z_i}}^{(\lambda_{z_i;\rho_{z_i}})}(z_i)
\subseteq
Q_{R_{2}}.
$$
Applying \eqref{level-est} to each cylinder $Q_{4\rho_{z_i}}^{(\lambda_{z_i;\rho_{z_i}})}(z_i)$ and summing over the disjoint subfamily yields 
\begin{align*}
	\iint_{\sfE (R_1,\lambda)} 
	|Du|^p \,\dx\dt
	&\le
	C\iint_{\sfE (R_2,\eta\lambda)} 
	\lambda^{p-\theta}|Du|^{\theta} \,\dx\dt \\
    &\quad +
	C
    \iint_{\sfG  (R_2,\eta\lambda)} 
	\lambda^{-p(q-1)}\Big[| \sfF|^{qp} + |R\,\sff|^{qp'}\Big] \,\dx\dt,
\end{align*}
with $C=C(N,p,C_o,C_1,q,\sfr)$.
To balance the levels, consider the intermediate layer $\sfE(R_{1},\eta\lambda)\setminus \sfE(R_{1},\lambda)$. Since $R_{1}<R_{2}$, one has $\sfE(R_{1},\eta\lambda)\subset \sfE(R_{2},\eta\lambda)$, and on this layer $ |Du|\le \lambda$. Hence
\begin{align*}
	\iint_{\sfE (R_1,\eta\lambda)\setminus \sfE (R_1,\lambda)} 
	|Du|^p \dx\dt
	&\le
	\iint_{\sfE (R_2,\eta\lambda)} 
	\lambda^{p-\theta}|Du|^{\theta} \dx\dt.
\end{align*}
Adding the two preceding bounds and then replacing $\eta\lambda$ by $\lambda$ (recalling that $\eta\in(0,1)$ is fixed in terms of $N,p,C_o,C_1,q,\sfr$ only) yields: for every $\lambda\ge \lambda_{1}:=\eta \sfB\lambda_{o}^{\frac \sfr2}$, we have
\begin{align}\label{pre-1}
	\iint_{\sfE(R_1,\lambda)} &
	|Du|^p \dx\dt \nonumber\\
	&\le
	C\iint_{\sfE(R_2,\lambda)} \!\!
	\lambda^{p-\theta}|Du|^{\theta} \dx\dt +
	C \iint_{\sfG(R_2,\lambda)}\!\! \lambda^{-p(q-1)}
    \Big[| \sfF|^{qp} + |R\,\sff|^{qp'}\Big]
    \dx\dt,
\end{align}
where $C=C(N,p,C_o,C_1,q,\sfr)$ and $\theta=\tfrac{Np}{N+2}$. This is the desired reverse Hölder–type inequality for the distribution function of $Du$.

\subsection{Proof of  Theorem~\ref{thm:higherint}}
Having established the \emph{reverse Hölder–type level–set inequality} \eqref{pre-1} for \(|Du|^{p}\) in the previous subsection, we now turn to its quantitative consequences. 
The next step is to convert this level–set estimate into an integral gain. 
To this end, we multiply \eqref{pre-1} by \(\lambda^{p\varepsilon-1}\) and integrate with respect to \(\lambda\); using Fubini’s theorem together with Hölder's and Young's inequality to handle the source terms, and finally applying the iteration Lemma~\ref{lem:tech-classical}, we infer that \(|Du|\in L^{p(1+\varepsilon)}(Q_{R})\). 
The lower bound for \(\lambda_{o}\) in \eqref{Eq:def-lm0-new} fixes the constants and yields the stated quantitative estimate.

To proceed with the next step of the argument,
fix $\sfk>\lambda_{1}=\eta\sfB\lm_o^{\frac \sfr2}$ and define the \emph{truncation}
$$
|Du|_{\sfk}:=\min\{|Du|,\sfk\},
$$
and, for $\rho\in(0,2R]$, the corresponding super-level set
$$
\sfE_{\sfk}(\rho,\lambda):=\big\{z\in Q_{\rho}\colon\ |Du|_{\sfk}(z)>\lambda\big\}.
$$
Note that $ |Du|_{\sfk}\le |Du|$ a.e., that $\sfE_{\sfk}(\rho,\lambda)=\varnothing$ if $\sfk\le\lambda$, and that $\sfE_{\sfk}(\rho,\lambda)=\sfE(\rho,\lambda)$ if $\sfk>\lambda$. Hence, by \eqref{pre-1}, for every $\sfk>\lambda\ge\lambda_{1}$, we have
\begin{align*}
	&\iint_{\sfE_\sfk(R_1,\lambda)} 
	|Du|_\sfk^{p-\theta}|Du|^{\theta} \d x\d t\\
	&\qquad\le
	C\iint_{\sfE_\sfk(R_2,\lambda)} 
	\lambda^{p-\theta}|Du|^{\theta} \dx\dt +
	C 
    \iint_{\sfG(R_2,\lambda)}
    \lambda^{-p(q-1)}
    \Big[| \sfF|^{qp} + |R\,\sff|^{qp'}\Big]
 \dx\dt,
\end{align*}
where $C=C(N,p,C_{o},C_{1},q,\sfr)$ and $\theta=\tfrac{Np}{N+2}$.
Since $\sfE_{\sfk}(\rho,\lambda)=\varnothing$ for $\sfk\le\lambda$, the preceding estimate holds for all $\sfk>0$. Multiplying by $\lambda^{p\varepsilon-1}$ with $\varepsilon\in(0,1]$ (to be fixed below) and integrating over $(\lambda_{1},\infty)$ yields
\begin{align*}
	\int_{\lambda_1}^\infty \lambda^{p\epsilon-1} &
	\bigg[\iint_{\sfE_\sfk(R_1,\lambda)} 
	|Du|_\sfk^{p-\theta}
	|Du|^{\theta} \,\dx\dt \bigg] \d\lambda
	\\
	&\le
	C\int_{\lambda_1}^\infty \lambda^{p(1+\epsilon)-\theta-1} 
	\bigg[
	\iint_{\boldsymbol \sfE_\sfk(R_2,\lambda)} 
	|Du|^{\theta} \,\dx\dt\bigg]\d\lambda \\
	&\quad +
	C \int_{\lambda_1}^\infty \lambda^{p(1-q+\epsilon)-1}
	\bigg[
	\iint_{\sfG(R_2,\lambda)}
    \Big[| \sfF|^{qp} + |R\,\sff|^{qp'}\Big]
    \,\dx\dt\bigg]\d\lambda.
\end{align*}
% with $C=C(N,p,C_{o},C_{1},\sfr)$ and $\theta=\tfrac{Np}{N+2}$.
Applying Fubini’s theorem and reordering the integrations gives, for the left–hand side,
\begin{align*}
	\int_{\lambda_1}^\infty &\lambda^{p\epsilon-1} 
	\bigg[
	\iint_{\sfE_\sfk(R_1,\lambda)} 
	|Du|_\sfk^{p-\theta}
	|Du|^{\theta} \,\dx\dt \bigg]\,\d\lambda
	\\
	&=
	\iint_{\sfE_\sfk(R_1,\lambda_1)}
	|Du|_\sfk^{p-\theta}
	|Du|^{\theta}
	\bigg[
	\int_{\lambda_1}^{|Du|_\sfk} 
	\lambda^{p\epsilon-1} \,\d\lambda\bigg]
	 \,\dx\dt \\
	&=
	\tfrac{1}{p\epsilon} \iint_{\sfE_\sfk(R_1,\lambda_1)}
	\Big[|Du|_\sfk^{p(1+\epsilon)-\theta}
	|Du|^{\theta}%\\
%	&\phantom{=\,}
	-
	\lambda_1^{p\epsilon} |Du
	|_\sfk^{p-\theta}|Du|^{\theta} \Big]
	\,\dx\dt .
\end{align*}
For the first term on the right–hand side, since $\theta<p$,
\begin{align*}
	\int_{\lambda_1}^\infty &\lambda^{p(1+\epsilon)-\theta -1} 
	\bigg[
	\iint_{\sfE_\sfk(R_2,\lambda)} 
	|Du|^{\theta} \,\dx\dt\bigg]\,\d\lambda \\
	&=
	\iint_{\sfE_\sfk(R_2,\lambda_1)} |Du|^{\theta}
	\bigg[
	\int_{\lambda_1}^{|Du|_\sfk} 
	\lambda^{p(1+\epsilon) -\theta-1} \,\d\lambda\bigg]
	\,\dx\dt \\
	&\le 
	\tfrac{1}{p(1+\epsilon)-\theta} \iint_{\sfE_\sfk(R_2,\lambda_1)} 
	|Du|_\sfk^{p(1+\epsilon)-\theta} 
	|Du|^{\theta}
	\,\dx\dt \\
	&\le 
	\tfrac{1}{p-\theta} \iint_{\sfE_\sfk(R_2,\lambda_1)} 
	|Du|_\sfk^{p(1+\epsilon)-\theta} 
	|Du|^{\theta}
	\,\dx\dt.
\end{align*}
%Since $1-q+\varepsilon<0$,
%indeed
Note that 
$$
1-q+\varepsilon<-\tfrac{N}{p}+\varepsilon\le -\tfrac{N+2}{2}+\varepsilon<-\tfrac{N}{2}<0
$$ because $q>\tfrac{N+p}{p}$ and $p\le \tfrac{2N}{N+2}$. Then, Fubini’s theorem gives
\begin{align*}
	\int_{\lambda_1}^\infty &\lambda^{p(1-q+\epsilon)-1}
	\bigg[
	\iint_{\sfG(R_2,\lambda)}
    \Big[| \sfF|^{qp} + |R\,\sff|^{qp'}\Big]
    \,\dx\dt\bigg]\d\lambda \\
	&=
	\iint_{\sfG (R_2,\lambda_1)} \Big[| \sfF|^{qp} + |R\,\sff|^{qp'}\Big]
	\bigg[\int_{\lambda_1}^{| \sfF|^{qp} + |R\,\sff|^{qp'} } 
    \lambda^{p(1-q+\epsilon)-1}
	\,\d\lambda\bigg]
	 \,\dx\dt \\
     &\le 
     \iint_{\sfG (R_2,\lambda_1)} \Big[| \sfF|^{qp} + |R\,\sff|^{qp'}\Big]
	\bigg[\int_{\lambda_1}^{\infty } 
    \lambda^{p(1-q+\epsilon)-1}
	\,\d\lambda\bigg]
	 \,\dx\dt\\
	&=
	\frac{\lambda_1^{-p(q-1-\epsilon)}}{p(q-1-\epsilon)} \iint_{\sfG(R_2,\lambda_1)}
	\Big[| \sfF|^{qp} + |R\,\sff|^{qp'}\Big] \,\dx\dt \\
	 &\le
	\frac{\lambda_1^{-p(q-1-\epsilon)}}{p}
    \iint_{Q_{2R}}
	\Big[| \sfF|^{qp} + |R\,\sff|^{qp'}\Big] \,\dx\dt. 
\end{align*}
Inserting the preceding bounds and multiplying by $p\varepsilon$ yields
\begin{align*}
	 \iint_{ \sfE_\sfk(R_1,\lambda_1)}&
	|Du|_\sfk^{p(1+\epsilon)-\theta}
	|Du|^{\theta} \,\dx\dt \\
	&\le
	\lambda_1^{p\epsilon} 
	\iint_{\sfE_\sfk(R_1,\lambda_1)}
	|Du|_\sfk^{p-\theta}
	|Du|^{\theta} \,\dx\dt \\
	&\phantom{\le\,}+
	\frac{C\, \epsilon}{p-\theta} \iint_{\sfE_\sfk(R_2,\lambda_1)} 
	|Du|_{\sfk}^{p(1+\epsilon)-\theta} 
	|Du|^{\theta}
	\,\dx\dt \\
	&\phantom{\le\,}+
	\frac{C\,\epsilon}{\lambda_1^{p(q-1-\epsilon)}}
    \iint_{Q_{2R}} \Big[| \sfF|^{qp} + |R\,\sff|^{qp'}\Big] \,\dx\dt ,	
\end{align*}
with $C=C(N,p,C_{o},C_{1},q,\sfr)$ and $\theta=\tfrac{Np}{N+2}$.
On the complement $Q_{R_{1}}\setminus \sfE_{\sfk}(R_{1},\lambda_{1})$ one has $|Du|_{\sfk}\le \lambda_{1}$, hence
\begin{align*}
	\iint_{Q_{R_1}\setminus \sfE_\sfk(R_1,\lambda_1)}&
	|Du|_\sfk^{p(1+\epsilon)-\theta}
	|Du|^{\theta} \,\dx\dt \\
	&\le
	\lambda_1^{p\epsilon} 
	\iint_{Q_{R_1}\setminus \sfE_\sfk(R_1,\lambda_1)}
	|Du|_\sfk^{p-\theta}
	|Du|^{\theta} \,\dx\dt .
\end{align*}
Adding this to the inequality over $\sfE_{\sfk}(R_{1},\lambda_{1})$ and using $|Du|_{\sfk}\le |Du|$ yields
\begin{align*}
	 \iint_{Q_{R_1}}&
	|Du|_\sfk^{p(1+\epsilon)-\theta}
	|Du|^{\theta} \,\dx\dt \\
	&\le
	\frac{C_\ast \epsilon}{p-\theta} \iint_{Q_{R_2}} 
	|Du|_\sfk^{p(1+\epsilon)-\theta} 
	|Du|^{\theta}
	\,\dx\dt +
	\lambda_1^{p\epsilon} 
	\iint_{Q_{2R}} |Du|^{p} \dx\dt\\
	&\phantom{\le\ } +
	\frac{C\,\epsilon}{\lambda_1^{p(q-1-\epsilon)}}
    \iint_{Q_{2R}} \Big[| \sfF|^{qp} + |R\,\sff|^{qp'}\Big] \,\dx\dt,
\end{align*}
where $C_\ast=C_\ast(N,p,C_o,C_1,q,\sfr)\ge 1$.
Choose
$$
0<\varepsilon\le\varepsilon_{o},
\qquad
\varepsilon_{o}:=\frac{p-\theta}{2C_{\ast}}<\tfrac12,
$$
where $\varepsilon_{o}$ depends only on $N,p,C_{o},C_{1},q,\sfr$.
Since $\lambda_{1}=\eta\sfB\lambda_{o}^{\frac\sfr2}$ with $\sfB,\,\lm_o\ge1$ and $\eta\in(0,1]$  
from \eqref{choice:eta},
we have
$$
\lambda_{1}^{\,p\varepsilon}\le \sfB^{\,p\varepsilon}\lambda_{o}^{\,p\varepsilon\frac\sfr2}\le \sfB^p\,\lambda_{o}^{\,p\varepsilon\frac\sfr2},
\qquad
\lambda_{1}^{\,p(q-1-\varepsilon)}\ge \eta^{\,p(q-1)}\lambda_{o}^{\,p \frac\sfr2 (q-1-\varepsilon)}.
$$
Inserting these bounds into the preceding estimate and absorbing the $\eta,\sfB$-factors into the constant (recalling the definition of $\sfB$ in \eqref{choice_lambda-new}) yields: for any radii $R\le R_{1}<R_{2}\le2R$,
\begin{align*}
	& \iint_{Q_{R_1}}
	|Du|_\sfk^{p(1+\epsilon)-\theta} 
	|Du|^{\theta} \,\dx\dt \\
	&\quad\le
	\tfrac{1}{2} \iint_{Q_{R_2}} 
	|Du|_\sfk^{p(1+\epsilon)-\theta} 
	|Du|^{\theta}
	\,\dx\dt+
	C\,\Big(\frac{R}{R_2{-}R_1}\Big)^{\frac{2\sfr(N+p)}{{\boldsymbol\lambda_\sfr}}}
	\lambda_o^{p\epsilon\frac\sfr2} 
	\iint_{Q_{2R}} |Du|^{p} \,\dx\dt \\
    &\qquad +
	\frac{C}{\lambda_o^{p\frac\sfr2(q-1-\epsilon)}}
    \iint_{Q_{2R}} \Big[| \sfF|^{qp} + |R\,\sff|^{qp'}\Big] \,\dx\dt,
\end{align*}
where $C=C(N,p,C_{o},C_{1},q,\sfr)$, and $\boldsymbol\lambda_{\sfr}=N(p-2)+p\sfr$.
Applying the iteration Lemma~\ref{lem:tech-classical} to the preceding inequality yields
\begin{align*}
	\iint_{Q_{R}}&
	|Du|_\sfk^{p(1+\epsilon)-\theta} 
	|Du|^{\theta} \,\dx\dt \\
	&\le
	C\, 
	\lambda_o^{p\epsilon\frac\sfr2} 
	\iint_{Q_{2R}}
	|Du|^{p} \,\dx\dt  +
	\frac{C}{\lambda_o^{p\frac\sfr2(q-1-\epsilon)}}
    \iint_{Q_{2R}} \Big[| \sfF|^{qp} + |R\,\sff|^{qp'}\Big] \,\dx\dt
\end{align*}
with $C=C(N,p,C_{o},C_{1},q,\sfr)$.
Since the right–hand side of the inequality does not depend on $\sfk$, 
and the integrability exponent of $|Du|_\sfk$ satisfies $p(1+\varepsilon)-\theta>0$, 
we may let $\sfk\uparrow\infty$ and apply the monotone convergence theorem 
to obtain the  higher integrability estimate
\begin{align*}
	\biint_{Q_{R}}& 
	|Du|^{p(1+\epsilon)} \,\dx\dt \\
	&\le
	C\, 
	\lambda_o^{p\epsilon\frac\sfr2} 
	\biint_{Q_{2R}}
	|Du|^{p} \,\dx\dt +
	\frac{C}{\lambda_o^{p\frac\sfr2(q-1-\epsilon)}}
    \biint_{Q_{2R}} \Big[| \sfF|^{qp} + |R\,\sff|^{qp'}\Big] \,\dx\dt .
\end{align*}
Invoking the definition from \eqref{Eq:def-lm0-new},
$$
\lambda_{o}\;\ge\;\max\big\{1,\ \sfI^{\,\frac{\sfr}{q\boldsymbol\lambda_{\sfr}}}\big\},
\qquad 
\sfI:=\biint_{Q_{4R}}\Big[|\sfF|^{qp}+|R\,\sff|^{qp'}\Big]\,\mathrm{d}x\mathrm{d}t,
$$
we estimate
\begin{align*}
	\frac{1}{\lambda_o^{p\frac\sfr2(q-1-\epsilon)}}
    \biint_{Q_{2R}}& \Big[| \sfF|^{qp} + |R\,\sff|^{qp'}\Big] \,\dx\dt \\
    &\le
    \frac{C}{\lambda_o^{p\frac\sfr2(q-1-\epsilon)}}
    \sfI ^{\frac{q-1-\epsilon}{q}} 
    \bigg[\biint_{Q_{2R}} \Big[| \sfF|^{qp} + |R\,\sff|^{qp'}\Big] \,\dx\dt\bigg]^{\frac{1+\epsilon}{q}}\\
    &\le 
    \frac{C}{\lambda_o^{(q-1-\epsilon)(p\frac\sfr2-\frac{\boldsymbol\lambda_\sfr}\sfr)}} 
    \bigg[\biint_{Q_{2R}} \Big[| \sfF|^{qp} + |R\,\sff|^{qp'}\Big] \,\dx\dt\bigg]^{\frac{1+\epsilon}{q}} \\
    &\le  
    C\bigg[\biint_{Q_{2R}} \Big[| \sfF|^{qp} + |R\,\sff|^{qp'}\Big] \,\dx\dt \bigg]^{\frac{1+\epsilon}{q}}.
\end{align*}
To obtain the last line, we used that
\[
p\frac{\sfr}{2}-\frac{\boldsymbol\lambda_\sfr}{\sfr}
= p\sfr\Big(\frac{1}{2}-\frac{1}{\sfr}\Big)
+ \frac{N(2-p)}{\sfr} > 0,
\]
and that $\lambda_{o}\ge1$. 
Consequently, we have 
\(
\lambda_o^{(q-1-\varepsilon)
(p\frac{\sfr}{2}-\frac{\boldsymbol\lambda_\sfr}{\sfr})} \ge 1
\).
Substituting the foregoing bound for the source terms contribution 
into the higher–integrability estimate, we obtain
\begin{align*}
	\biint_{Q_{R}} &
	|Du|^{p(1+\epsilon)} \,\dx\dt\\
	&\le
    C\lambda_o^{p\epsilon\frac\sfr2}
	\biint_{Q_{2R}}
	|Du|^{p} \dx\dt+
	C\bigg[ \biint_{Q_{2R}} \Big[| \sfF|^{qp} + |R\,\sff|^{qp'}\Big] \dx\dt \bigg]^{\frac{1+\epsilon}{q}},
\end{align*}
with $C=C(N,p,C_{o},C_{1},q,\sfr)$. 
On the other hand, a direct application of the zero–order energy inequality 
\eqref{est:energy*} on $Q_{4R}$ with $a=(u)_{8R}$ yields
\begin{align*}
&\bigg[\biint_{Q_{4R}}|Du|^p\,\dx\dt
\bigg]^{\frac\sfr{\boldsymbol\lambda_\sfr}}
\\
&\qquad\le C\Bigg[1+\bigg[\biint_{Q_{8R}}\frac{|u-(u)_{8R}|^\sfr}{(8R)^\sfr}\,\dx\dt
\bigg]^{\frac2{\boldsymbol\lambda_\sfr}}+
\bigg[\biint_{Q_{8R}}
\Big[|\sfF|^{qp}+|R\sff|^{qp'}\Big]\,
\dx\dt\bigg]^{\frac\sfr{q\boldsymbol\lambda_\sfr}}\Bigg],
\end{align*}
and consequently,
\[
\lambda_o\le 
C\Bigg[1+\bigg[\biint_{Q_{8R}}\frac{|u-(u)_{8R}|^\sfr}{(8R)^\sfr}\,\dx\dt
\bigg]^{\frac2{\boldsymbol\lambda_\sfr}}+
\bigg[\biint_{Q_{8R}}
\Big[|\sfF|^{qp}+|R\sff|^{qp'}\Big]\,
\dx\dt\bigg]^{\frac\sfr{q\boldsymbol\lambda_\sfr}}\Bigg].
\]
This completes the proof of Theorem~\ref{thm:higherint}.
\qed

\end{document}